\documentclass[preprint,12pt]{elsarticle}

\usepackage{amssymb}
\usepackage{latexsym,amssymb,amsfonts,amsmath,amsthm,nccmath,enumitem,setspace,graphicx,float}
\usepackage[dvipsnames]{xcolor}
\usepackage{subcaption}
\usepackage{dsfont}
\usepackage[margin=2cm]{geometry}
\usepackage{float}
\usepackage[normalem]{ulem}
\theoremstyle{definition}
\newtheorem{theorem}{Theorem}[section]
\newtheorem{corollary}[theorem]{Corollary}
\newtheorem{assumption}[theorem]{Assumption}
\newtheorem{conj}[theorem]{Conjecture}
\newtheorem{lemma}[theorem]{Lemma}
\newtheorem{proposition}[theorem]{Proposition}
\newtheorem{remark}[theorem]{Remark}
\newtheorem{notation}[theorem]{Notation}
\newtheorem{definition}[theorem]{Definition}

\newtheorem{construction}[theorem]{Construction}
\newtheorem{example}[theorem]{Example}

\usepackage{hyperref}

\usepackage{tikz}
\tikzstyle{vertex}=[circle, draw=black, fill=black, minimum size=2pt, inner sep=2]
\usetikzlibrary{arrows, snakes,arrows.meta,calc,decorations.markings,math,arrows.meta}

\makeatletter

\tikzset{
  on each segment/.style={
    decorate,
    decoration={
      show path construction,
      moveto code={},
      lineto code={
        \path [#1]
        (\tikzinputsegmentfirst) -- (\tikzinputsegmentlast);
      },
      curveto code={
        \path [#1] (\tikzinputsegmentfirst)
        .. controls
        (\tikzinputsegmentsupporta) and (\tikzinputsegmentsupportb)
        ..
        (\tikzinputsegmentlast);
      },
      closepath code={
        \path [#1]
        (\tikzinputsegmentfirst) -- (\tikzinputsegmentlast);
      },
    },
  },
  mid arrow/.style={postaction={decorate,decoration={
        markings,
        mark=at position .5 with {\arrow[scale=1.2,#1]{stealth}}
      }}},
}

\let\oldproofname=\proofname
\renewcommand{\proofname}{\rm\bf{\oldproofname}}
\journal{Arxiv}

\begin{document}

\begin{frontmatter}



\title{Hamiltonian decompositions of the wreath product of hamiltonian decomposable digraphs}


\author{Alice Lacaze-Masmonteil} %% Author name

\affiliation{organization={University of Regina},%Department and Organization
            addressline={Department of Mathematics and Statistics, University of Regina, 3737 Wascana Pkwy}, 
            city={Regina},
            postcode={S4S 0A2}, 
            state={Saskachewan},
            country={Canada}}

%% Abstract
\begin{abstract}
We affirm most open cases of a conjecture that first appeared in Alspach et al.~(1987) which stipulates that the wreath (lexicographic) product of two hamiltonian decomposable directed graphs is also hamiltonian decomposable. Specifically, we show that the wreath product of  a hamiltonian decomposable directed graph $G$, such that $|V(G)|$ is even and $|V(G)|\geqslant 2$, with a hamiltonian decomposable directed graph $H$, such that $|V(H)| \geqslant 4$, is also hamiltonian decomposable except possibly when $G$ is a directed cycle and $H$ is a directed graph of odd order that admits a decomposition into $c$ directed hamiltonian cycle where $c$ is odd and $3\leqslant c \leqslant |V(H)|-2$. 
\end{abstract}

%% Keywords

\end{frontmatter}



\section{Introduction}\label{S:intro}

The \textit{wreath product} of directed graph (digraph for short) $G$ with digraph $H$, denoted $G \wr H$, is the digraph on vertex set $V(G) \times V(H)$ with a directed edge (arc) from $(g_1, h_1)$ to $(g_2, h_2)$ if and only if there exists an arc from $g_1$ to $g_2$ in $G$, or there exists an arc from $h_1$ to $h_2$ in $H$ and $g_1=g_2$.  A (directed) graph that admits a decomposition into (directed) hamiltonian cycles is a \textit{hamiltonian decomposable} (directed) graph. 

Regarding undirected graphs, Baranyai and  Szász  \cite{HamLex} have shown that the wreath product of two hamiltonian decomposable graphs is hamiltonian decomposable. An analogous statement has been shown to hold for the strong product of two hamiltonian decomposable graphs \cite{Fan, Zhou}. Stong \cite{Stong} has shown that, with some additional assumptions,  the Cartesian product of two hamiltonian decomposable graphs is also hamiltonian decomposable. Bermond \cite{Berm} has shown that the tensor product of two odd ordered hamiltonian decomposable graphs is also hamiltonian decomposable. In addition, Balakrishnan et al.~\cite{Bala+} have shown that the tensor product of two complete graphs of order at least three is hamiltonian decomposable. 

There are few results on the analogous problem for digraphs. The Cartesian product of a directed $n$-cycle with a directed $m$-cycle is hamiltonian decomposable if and only if there exist positive integers $R_1$ and $R_2$ such that $\textrm{gcd}(m, n)=R_1+R_2$ and $\textrm{gcd}(mn, R_1R_2)=1$ \cite{Kea}. As for the tensor product, Paulraja and Sivasankar \cite{Praja2} have shown that the tensor product of several specific classes of hamiltonian decomposable digraphs is also hamiltonian decomposable. 

This paper will focus on the following long-standing conjecture on the wreath product of two hamiltonian decomposable digraphs which first appeared in \cite{CD1}. 

\begin{conj} \cite{CD1}
\label{conj:main}
Let $G$ and $H$ be hamiltonian decomposable directed graphs. The wreath product $G \wr H$  is also hamiltonian decomposable. 
\end{conj}

In \cite{Ng}, Ng successfully adapts certain methods of Baranyai and Szász \cite{HamLex} to affirm Conjecture \ref{conj:main} for the following cases. 

\begin{theorem} \cite{Ng}
\label{thm:ng}
Let $G$ and $H$ be two hamiltonian decomposable digraphs such that $|V(G)|$ is odd and $|V(H)|\geqslant 3$. Then $G \wr H$ is hamiltonian decomposable.
\end{theorem}

Using very different methods than those used in this paper and in \cite{Ng}, the author \cite{Me} affirms Conjecture \ref{conj:main} when $H$ is a directed cycle or the complete symmetric digraph.

\begin{theorem} \cite{Me}
\label{thm:1}
Let $G$ be a hamiltonian decomposable digraph of even order $n$ and $H$ be the complete symmetric digraph on $m$ vertices.  The digraph $G \wr H$ is hamiltonian decomposable if and only if $(n,m) \neq(2,3)$ and $m \neq 2$.
\end{theorem}

\begin{theorem} \cite{Me}
\label{thm:2}
Let $G$ be a hamiltonian decomposable digraph of even order $n$ and $H$ be the directed $m$-cycle. The digraph $G\wr H$ is hamiltonian decomposable in each of the following cases:
\begin{enumerate} [label=\textbf{(S\arabic*)}]
\item $m$ is even, $m \geqslant 4$, and $n \geqslant 4$;
\item $m=4$ and $n=2$;
\item $m$ is odd, and $m\geqslant 5$.
\end{enumerate}
Furthermore, if $G$ is a directed $n$-cycle and $m\in \{2,3\}$, then $G\wr H$ is not hamiltonian decomposable. 
\end{theorem}

Theorem \ref{thm:al}, which summarizes this paper's result, complements current results by addressing almost all open cases of Conjecture \ref{conj:main}. 

\begin{theorem} 
\label{thm:al}
Let $G$ and $H$ be two hamiltonian decomposable digraphs such that $|V(G)|$ is even, $|V(H)|>3$, and $H$ is not a directed cycle or a complete symmetric digraph. Then $G \wr H$ is hamiltonian decomposable except possibly when $G$ is a directed cycle, $|V(H)|$ is even, and $H$ admits a decomposition into an odd number of directed hamiltonian cycles.
\end{theorem}

\begin{proof}
If $|V(H)|$ is even, the statement follows from Theorem \ref{thm:geneven}.  If $|V(H)|$ is odd, the statement follows from Theorem \ref{thm:oddmc}. \end{proof}

This paper is structured as follows. In Section \ref{S:Prm}, we introduce key terminology and definitions. Then, in Section \ref{S:tool}, we show that $G \wr H$ is hamiltonian decomposable when the wreath product of a directed 2-cycle with the empty graph admits a particular decomposition into directed cycles. Lastly, in Sections \ref{sec:meven} and \ref{sec:modd}, we establish the existence of these decompositions for all applicable cases. 

\section{Preliminaries}
\label{S:Prm}

In this paper, we assume that all digraphs are strict. A digraph is \textit{strict} when it does not contain loops or repeated arcs. We denote the vertex set of digraph  $G$ as $V(G)$ and its arc set  as $A(G)$. An arc with tail $x$ and head $y$ is denoted $(x, y)$. The directed cycle on $m$ vertices is denoted $\vec{C}_m$ while  the empty digraph on $m$ vertices is denoted $\overline{K}_m$. If $P$ is a directed path (dipath for short) of digraph $G$, then its \textit{source} is the vertex with in-degree 0, and is denoted $s(P)$, while its \textit{terminal} is the vertex with out-degree 0, and is denoted $t(P)$.

We will adopt the following notation.

\begin{notation} \rm
\label{not:conchp5}
We assume that $G$ and $H$ are two digraphs on $n$ and  $m$ vertices, respectively. Let $V(G)=\mathds{Z}_n$ and $V(H)=\mathds{Z}_m$. Then, $V(G\wr H)=\{(x, y)\ |\ x\in \mathds{Z}_n, y\in \mathds{Z}_m\}$. To avoid conflict with the notation used for arcs, we write $x_y$ for  $(x,y)$. For each $i \in \mathds{Z}_n$, we let $V_i=\{i_0, i_1, \ldots, i_{m-1}\}$.  \end{notation}

Using Notation \ref{not:conchp5}, we describe a particular set of arcs of $G\wr H$. 

\begin{definition}
An arc of $G \wr H$ of the form $(i_{j_1}, i_{j_2})$, where $j_1, j_2 \in \mathds{Z}_m$, is a \textit{horizontal arc}. 
\end{definition}

We now introduce a key property of the wreath product. The subdigraph of $G \wr H$ induced by the set of vertices $V_i$ is necessarily isomorphic to $H$. This observation gives rise to the following definition. 

\begin{definition} \rm
\label{def:embed}
Let $H_i$ be the subdigraph of $G \wr H$ induced by vertex set $V_i$. An \textit{embedding of $H$ into $V_i$} is an isomorphism $\phi: H \mapsto H_i$. 
\end{definition}

It follows from the definition of the wreath product that, in $G \wr H$, we can embed $H$ into each $V_i$ arbitrarily. In certain proofs, our constructions are made relative to specific embeddings of $H$.

Next, we introduce definitions related to cycle decompositions of digraphs. A \textit{decomposition} of  a digraph $G$ is a set $\{H_1, H_2, \ldots, H_r\}$ of pairwise arc-disjoint subdigraphs of $G$ such that $A(G)=A(H_1) \cup A(H_2) \cup \cdots \cup A(H_r)$. If $G$ admits the decomposition  $\{H_1, H_2, \ldots, H_r\}$, we write $G=H_1\oplus H_2 \oplus \cdots \oplus H_r$. A spanning subdigraph of $G$ comprised of the disjoint union of directed cycles of $G$ is a \textit{directed 2-factor} of $G$. A decomposition of $G$ into directed 2-factors is a directed \textit{2-factorization}. Note that a directed hamiltonian cycle of $G$ is necessarily a 2-factor of $G$ and a \textit{hamiltonian decomposition} of $G$ is a directed 2-factorization in which all subdigraphs are directed hamiltonian cycles. 

Most of our constructions are directed 2-factorizations of $\vec{C}_n \wr \overline{K}_m$. We will make the following standard assumption.

\begin{assumption}\rm
If $V(\vec{C}_n)=\mathds{Z}_n$, then $A(\vec{C}_n)=\{(i, i+1)\ |\ i \in \mathds{Z}_n\}$ with addition in modulo $n$.
\end{assumption}

We will describe 2-factors of  $\vec{C}_n \wr \overline{K}_m$ as $n$-tuples of permutations from the symmetric group $S_m$.  If $\pi, \sigma \in S_m$, then $i^{\pi\sigma}=(i^\pi)^\sigma$. A cycle of a permutation is written as $(a_0,a_1,a_2, \ldots, a_{m-1})$ instead of the standard $(a_0\ a_1\ a_2\  \ldots\ a_{m-1})$ for ease of notation in later computations. The identity permutation will be denoted as $id$. 
 
\begin{notation}\rm
\label{not:Sm}
A directed 2-factor $F$ of  $\vec{C}_n \wr \overline{K}_m$  corresponds to the $n$-tuple $(\sigma_0, \sigma_1, \ldots, \sigma_{n-1})$ where $\sigma_i\in S_m$ and  $A(F)=\bigcup_{i\in \mathds{Z}_n}\{(i_j, (i+1)_{j^{\sigma_i}}) \ |\ j \in \mathds{Z}_m\}$.
\end{notation}

Given $\sigma \in S_m$, we denote the number of cycles of $\sigma$ in its disjoint cycle notation as $T(\sigma)$. Clearly, the number of directed cycles in the directed 2-factor $F=(\sigma_0, \sigma_1, \ldots, \sigma_{n-1})$  of $\vec{C}_n \wr \overline{K}_m$  equals $T(\sigma_0\sigma_1 \cdots \sigma_{n-1})$. Therefore, $F=(\sigma_0, \sigma_1, \ldots, \sigma_{n-1})$  is a directed hamiltonian cycle of $\vec{C}_n \wr \overline{K}_m$ if and only if $T(\sigma_0\sigma_1\cdots \sigma_{n-1})=1$. 

A directed 2-factorization of $\vec{C}_n \wr \overline{K}_m$ necessarily contains $m$ directed 2-factors. If $D=\{F_0, F_1, F_2, \ldots, F_{m-1}\}$ is a directed 2-factorization of  $\vec{C}_n \wr \overline{K}_m$, then $D$ can be described as a set of $m$ $n$-tuples of elements of $S_m$ as follows:

\begin{center}
$D=\{F_k=(\sigma_{(k,1)}, \sigma_{(k,2)}, \ldots, \sigma_{(k,n)})\ |  \ k \in \mathds{Z}_m,  \sigma_{(k,j)} \in S_m \ \textrm{for all}\ j=1,2, \ldots, n \}$.
\end{center}

Conversely, it is not always the case that a set of $m$ $n$-tuples of permutations will correspond to a 2-factorization of  $\vec{C}_n \wr \overline{K}_m$. Namely, for each $j \in \mathds{Z}_n$, the set of permutations $T_j=\{\sigma_{(k, j)} \ | \  k \in \mathds{Z}_m\}$ must satisfy a very specific property given in Definition \ref{defn:decof} below. 

\begin{definition} \rm
\label{defn:decof}
Let $S=\{\sigma_0, \sigma_1, \ldots, \sigma_{m-1}\}$ be a set of $m$ permutations of $S_m$. The set $S$ is a \textit{regular permutation set of order $m$} if $j^{\sigma_{k_1}}\neq j ^{\sigma_{k_2}}$ for all $j\in \mathds{Z}_m$ and $k_1, k_2\in \mathds{Z}_m$ such that $k_1 \neq k_2$. A permutation $\sigma \in S_m$ is an \textit{$(m-1)$-stabilizer} if $(m-1)^\sigma=m-1$. 
\end{definition}

Therefore, the set of $m$ $n$-tuples $D=\{(\sigma_{(k,1)}, \sigma_{(k,2)}, \ldots, \sigma_{(k,n)})\ |  \ k \in \mathds{Z}_m,  \sigma_{(k,j)} \in S_m\ \textrm{for all}\ j=1,2, \ldots, n \}$ corresponds to a 2-factorization of $\vec{C}_n \wr \overline{K}_m$ precisely when, for each $j\in \mathds{Z}_n$, the set $T_j=\{\sigma_{(k, j)} \ | \  k \in \mathds{Z}_m\}$  is a regular set of order $m$. Definition \ref{defn:decof} implies that the permutations in $T_j$ partition the set of arcs with tail in $V_j$ and head in $V_{j+1}$. 

In Definition \ref{def:cyc} below, we give a very simple example of a regular permutation set of order $m$ which is also key to some of our constructions. 

\begin{definition}
\label{def:cyc}
The cyclic group $\langle (0,1,2,3, \ldots, m-1) \rangle$ is denoted $\Pi_m$. We have that $\Pi_m=\{\pi_i\ |\ i=0,1,\ldots, m-1\}$, where $\pi_i=(0,1,2,3, \ldots, m-1)^i$. 
\end{definition}

Lastly, in Lemma \ref{lem:dec}, we show how to construct a new regular permutation set of order $m$ from an existing one. Proof of Lemma \ref{lem:dec} is straightforward and thus omitted. 

\begin{lemma}
\label{lem:dec}
Let $S=\{\sigma_0, \sigma_1, \ldots, \sigma_{m-1}\}$ be a regular permutation set of order $m$ and $\mu \in S_m$.  Then  the set $\mu \cdot S=\{\mu \sigma_0, \mu \sigma_1, \ldots, \mu \sigma_{m-1}\}$ is also a regular permutation set of order $m$. 
\end{lemma}

\section{Reduction}
\label{S:tool}

We now proceed with the description of our general approach which is similar to that of Ng \cite{Ng}.  In his approach, Ng took the following reduction step formulated, in Lemma \ref{lem:redCn} below. Proof of Lemma \ref{lem:redCn} can be found in the proof of Proposition 1 of \cite{Ng}. 

\begin{lemma} \cite{Ng}
\label{lem:redCn}
Let $G$ and $H$ be two hamiltonian decomposable digraphs such that $|V(G)|=n$ and $|V(H)|>2$.  If $\vec{C}_n \wr H$  is hamiltonian decomposable, then $G \wr H$ is also hamiltonian decomposable.   
\end{lemma}

Also key to Ng's approach is the following lemma. 

\begin{lemma} \cite{Ng}
\label{lem:empty}
The digraph $\vec{C}_n \wr \overline{K}_m$ is hamiltonian decomposable for all $n, m \geqslant 2$.   
\end{lemma}

As a result of Theorem \ref{thm:ng} and Lemma \ref{lem:redCn}, it suffices to concentrate on the case of Conjecture \ref{conj:main} where $G=\vec{C}_n$ with $n$ even. In what follows, we show that, if  $n$ is even and $|V(H)|=m$, the digraph $\vec{C}_n \wr H$ is hamiltonian decomposable when $\vec{C}_2\wr\overline{K}_m$ admits a particular directed 2-factorization. Although our general approach closely follows that of Ng \cite{Ng}, our terminology differs significantly. This is due to the fact that we explain why this method gives rise to the desired decompositions of $\vec{C}_n \wr H$ in far more detail than \cite{Ng}. 

We begin with Lemma \ref{lem:central} in which we show that a carefully chosen subdigraph of $\vec{C}_n \wr \vec{C}_m$ is hamiltonian decomposable. To prove Lemma \ref{lem:central}, we will require the following definitions. 

\begin{definition}\rm
\label{defn:trunc}
Let $\sigma \in S_m$ such that $\sigma$ is not an $(m-1)$-stabilizer.  The \textit{truncation} of $\sigma$, denoted $\hat{\sigma}$, is the permutation $\hat{\sigma}=\sigma\,(m-1, (m-1)^\sigma)$. 
\end{definition}

See Example \ref{ex:truca} below for a very simple example of the truncation of a permutation. 

\begin{example}\label{ex:truca}
\rm Let $\sigma=(0,1,2,3,4,5,6,7)$ such that $\sigma \in S_8$. Then, we see that $7^\sigma=0$. Therefore, $\hat{\sigma}=(0,1,2,3,4,5,6,7)(7,0)=(0,1,2,3,4,5,6)(7)$. $\square$
\end{example}

It follows from Definition \ref{defn:trunc} that, if $\sigma$ is not an $(m-1)$-stabilizer, then $(m-1)^{\hat{\sigma}}=m-1$, so $\hat{\sigma}$ is an $(m-1)$-stabilizer. We will also require the following definition. 

\begin{definition} \rm
\label{defn:hp}
Let $F=(\sigma_0, \sigma_1, \ldots, \sigma_{n-1})$ be a directed 2-factor of $\vec{C}_n \wr \overline{K}_m$. 
\begin{enumerate} 
\item If $T(\sigma_0\sigma_1\cdots\sigma_{n-1})=1$, then $F$ is called a \textit{hamiltonian $n$-tuple}. 
\item If no permutation in $F$ is an $(m-1)$-stabilizer and $T(\hat{\sigma}_0\hat{\sigma}_1\cdots\hat{\sigma}_{n-1})=2$, then $F$ is called a \textit{truncated hamiltonian $n$-tuple}.
\end{enumerate}
\end{definition}

\begin{lemma}
\label{lem:central}
Let $F=(\sigma_0, \sigma_1, \ldots, \sigma_{n-1})$ be a directed 2-factor of $\vec{C}_n \wr \vec{C}_m$ that is also a truncated hamiltonian $n$-tuple. Moreover, let $\Gamma$ be a spanning subdigraph of $\vec{C}_n \wr \vec{C}_m$ obtained from $F$ by adjoining all horizontal arcs of $\vec{C}_n \wr \vec{C}_m$. Then $\Gamma$ is hamiltonian decomposable. 
\end{lemma}

\begin{proof} Every vertex in $\Gamma$ has out-degree 2 and thus, any hamiltonian decomposition of $\Gamma$ consists of two directed hamiltonian cycles. 

Before we proceed, we first specify how we embed $\vec{C}_m$ into each $V_i$ of $\vec{C}_n \wr \vec{C}_m$. The digraph $\vec{C}_m$ is embedded into $V_i$ such that $(i_{m-1}, i_{(m-1)^{\sigma_{i-1}}})$ is a horizontal arc in $V_i$. Observe that, in the embedding of $\vec{C}_m$ into $V_i$, we then have a dipath of length $m-1$ with source $i_{(m-1)^{\sigma_{i-1}}}$ and terminal $i_{m-1}$. We will call this dipath $P_i$.  

We now form a spanning subdigraph of $\Gamma$, which we call $C^0$, such that 

\begin{center}
$A(C^0)= \cup_{i=0}^{n-1} \{((i-1)_{m-1}, i_{(m-1)^{\sigma_{i-1}}})\} \bigcup (\cup_{i=0}^{n-1} A(P_i))$. 
\end{center}

\noindent There exists an arc in $C^0$ from $t(P_{i-1})$ to $s(P_i)$ for each $i \in \mathds{Z}_n$, namely the arc $((i-1)_{m-1}, i_{(m-1)^{\sigma_{i-1}}})$.  Therefore, the digraph $C^0$ is a directed hamiltonian cycle of $\Gamma$. 

Next, we let  $C^1$ be the spanning digraph of $\Gamma$ such that $A(C^1)=A(\Gamma)-A(C^0)$. For each $i \in \mathds{Z}_n$, we form the following set of arcs:

\begin{center}
 $A_i=\{ ((i-1)_j, i_{j^{\sigma_{i-1}}})\ | \ j\in \mathds{Z}_m / \{m-1\}\}\cup \{(i_{m-1}, i_{(m-1)^{\sigma_{i-1}}})\}$. 
\end{center}

The digraph $C^1$ is a spanning digraph comprised of disjoint directed cycles. We must now show that $C^1$ is indeed a directed hamiltonian cycle. 

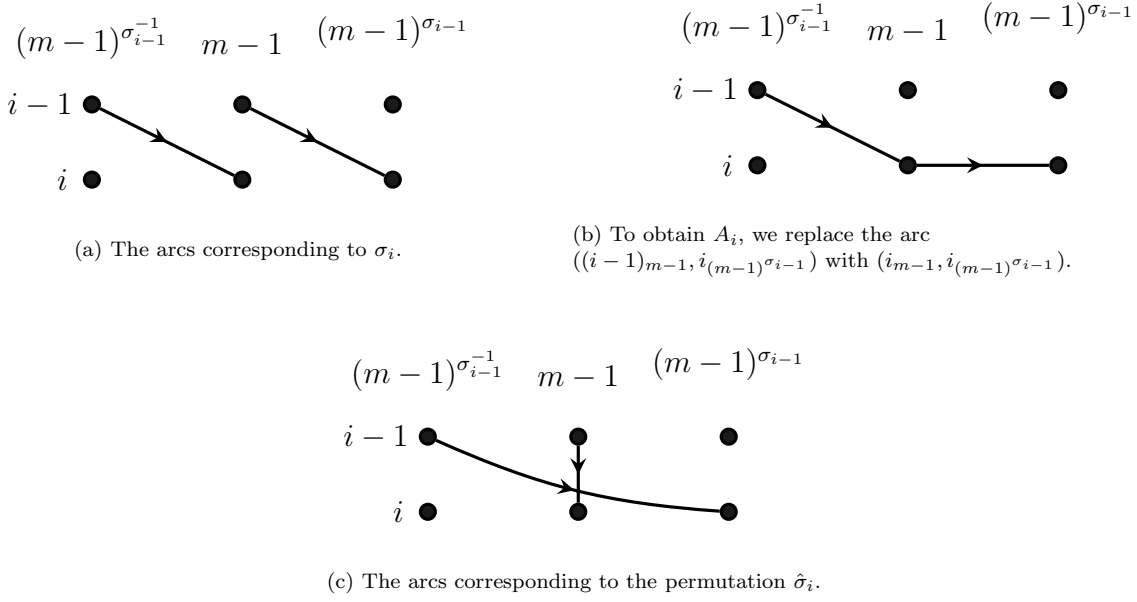
\begin{figure}[ht!]
\begin{center}
\begin{subfigure}[c]{0.49 \textwidth}
\begin{center}
\begin{tikzpicture}[
  very thick,
  every node/.style={circle,draw=black,fill=black!90}, inner sep=2]
    every node/.style={circle,draw=black,fill=black!90}, inner sep=1.5, scale=0.75]
  \node (x0) at (0.0,3.0) [label={[label distance=-4mm]above :$(m-1)^{\sigma^{-1}_{i-1}}$}] [label=left:$i-1$] {};
  \node (x1) at (2.0,3.0) [label={above:{$m-1$}}]{};
  \node (x2) at (4.0,3.0)[label={[label distance=-3mm]above:$(m-1)^{\sigma_{i-1}}$}] {};
  \node (y0) at (0.0,2.0) [label=left:$i$] {};
  \node (y1) at (2.0,2.0) {};
  \node (y2) at (4.0,2.0) {};
  
  \path [very thick, draw=black, postaction={very thick, on each segment={mid arrow}}]
	(x0) to (y1)
	(x1) to  (y2);
\end{tikzpicture}
\end{center}
\caption{The arcs corresponding to $\sigma_i$.}
\label{ex:swi1}
\end{subfigure}
\begin{subfigure}[c]{0.50\textwidth}
\begin{center}
\begin{tikzpicture}[
  very thick,
  every node/.style={circle,draw=black,fill=black!90}, inner sep=2]
    every node/.style={circle,draw=black,fill=black!90}, inner sep=1.5, scale=0.75]
  \node (x0) at (0.0,3.0) [label={[label distance=-4mm]above :$(m-1)^{\sigma^{-1}_{i-1}}$}] [label=left:$i-1$] {};
  \node (x1) at (2.0,3.0) [label={above:{$m-1$}}]{};
  \node (x2) at (4.0,3.0)[label={[label distance=-3mm]above:$(m-1)^{\sigma_{i-1}}$}] {};
  \node (y0) at (0.0,2.0) [label=left:$i$] {};
  \node (y1) at (2.0,2.0) {};
  \node (y2) at (4.0,2.0) {};
  
  \path [very thick, draw=black,postaction={very thick, on each segment={mid arrow}}]
	(x0) to (y1)
	(y1) to  (y2);
\end{tikzpicture}
\end{center}
\caption{To obtain $A_i$, we replace the arc\\ $((i-1)_{m-1}, i_{(m-1)^{\sigma_{i-1}}})$ with $(i_{m-1}, i_{(m-1)^{\sigma_{i-1}}})$.}
\label{ex:swi2}
\end{subfigure}
\begin{subfigure}[c]{0.55 \textwidth}
\begin{center}
\begin{tikzpicture}[
  very thick,
  every node/.style={circle,draw=black,fill=black!90}, inner sep=2]
    every node/.style={circle,draw=black,fill=black!90}, inner sep=1.5, scale=0.75]
  \node (x0) at (0.0,3.0) [label={[label distance=-4mm]above :$(m-1)^{\sigma^{-1}_{i-1}}$}] [label=left:$i-1$] {};
  \node (x1) at (2.0,3.0) [label={above:{$m-1$}}]{};
  \node (x2) at (4.0,3.0)[label={[label distance=-3mm]above:$(m-1)^{\sigma_{i-1}}$}] {};
  \node (y0) at (0.0,2.0) [label=left:$i$] {};
  \node (y1) at (2.0,2.0) {};
  \node (y2) at (4.0,2.0) {};
  
  \path [very thick, draw=black, postaction={very thick, on each segment={mid arrow}}]
	(x0) to [bend right=10] (y2)
	(x1) to  (y1);
\end{tikzpicture}
\end{center}
\caption{The arcs corresponding to the permutation $\hat{\sigma}_i$.}
\label{ex:swi3}
\end{subfigure}
\end{center}
\caption{Illustration of the truncation of $\sigma_i$.}
\label{fig:hat}
\end{figure}

In order to show that $C^1$ is a directed hamiltonian cycle, we will first show that  the permutation of elements of $\mathds{Z}_m/\{m-1\}$ by the arcs in $A_i$ and by $\hat{\sigma}_{i-1}$ is the same. The set $A_i$ is the set of arcs corresponding to $\sigma_{i-1}$  with tail in $V_{i-1}$ and head in $V_i$ and with the arc $((i-1)_{m-1}, i_{(m-1)^{\sigma_{i-1}}})$ replaced by the arc $(i_{m-1}, i_{(m-1)^{\sigma_{i-1}}})$. See Figures \ref{ex:swi1} and \ref{ex:swi2} for an illustration of this replacement. Observe that $\hat{\sigma}_{i-1}$ fixes $m-1$ and maps $(m-1)^{\sigma_{i-1}^{-1}}$ to $(m-1)^{\sigma_{i-1}}$. Otherwise, the permutation $\hat{\sigma}_{i-1}$ maps $j$ to $j^{\sigma_{i-1}}$ when $j \not\in \{m-1, (m-1)^{\sigma^{-1}_{i-1}} \}$. 

For each $i\in \mathds{Z}_n$, we obtain $B_i$ from $A_i$ by replacing the two arcs in Figure \ref{ex:swi2} with the two arcs in Figure \ref{ex:swi3}. Observe that $B_i$ is precisely the set of arcs corresponding to $\hat{\sigma}_i$. Let $\hat{F}=(\hat{\sigma}_0, \hat{\sigma}_1, \ldots, \hat{\sigma}_{n-1})$. That is, $\hat{F}$ is a directed 2-factor of $\vec{C}_n \wr \overline{K}_m$ with arc set $\cup_{i=0}^{n-1} B_i$. By assumption, $\hat{F}$ consists of two disjoint directed cycles, one of which is a directed $n$-cycle with vertex set $\{i_{m-1}\ | \ i\in \mathds{Z}_n\}$. It follows that the arc set  $\cup_{i=0}^{n-1}(B_i-\{(i-1)_{m-1}, i_{m-1}\})$ gives rise to a directed $(mn-n)$-cycle, and hence $\cup_{i=0}^{n-1} A_i$ gives rise to a directed $(mn)$-cycle of $\vec{C}_n \wr \vec{C}_m$. As a result, the set $\{C^0, C^1\}$ is a decomposition of $\Gamma$ into directed hamiltonian cycles.  \end{proof}

Our next objective is to use Lemma \ref{lem:central} to show that $\vec{C}_n \wr H$ is hamiltonian decomposable by using a carefully constructed directed 2-factorization of $\vec{C}_n \wr \overline{K}_m$. In Definition \ref{defn:twined} below, we use Definition \ref{defn:hp} to define this particular type of directed 2-factorization of $\vec{C}_n \wr \overline{K}_m$. 

\begin{definition}\rm
 \label{defn:twined}
Let $\{R_1, R_2, \ldots, R_{n}\}$ be a set of $n$ regular permutation sets of order $m$ and let  
\begin{center}
$D=\{(\sigma_{(k,1)}, \sigma_{(k,2)}, \ldots, \sigma_{(k,n)})\ |  \ k \in \mathds{Z}_m,  \sigma_{(k,j)} \in R_j \ \textrm{for all}\ j=1,2, \ldots, n \}$
\end{center}
\noindent  be a directed 2-factorization of $\vec{C}_n \wr \overline{K}_m$. Furthermore, assume that $0 \leqslant c \leqslant m-2$.  The set $D$ is called a \textit{$c$-twined 2-factorization of $\vec{C}_n \wr\overline{K}_m$} if there exists a partition of $D$ into two sets $D_T$ and $D_H$ such that $D_T$ contains $c$ truncated hamiltonian $n$-tuples and $D_H$  contains $m-c$ hamiltonian $n$-tuples. 
\end{definition}

In Proposition \ref{prop:twin}, we show that a $c$-twined 2-factorization of $\vec{C}_n \wr \overline{K}_m$ gives rise to a decomposition of $\vec{C}_n \wr H$ into directed hamiltonian cycles when $|V(H)|=m$ and $H$ admits a decomposition into $c$ directed hamiltonian cycles, with $0 \leqslant c \leqslant m-2$. 

\begin{proposition}
\label{prop:twin}
Let $H$ be a digraph on $m$ vertices that admits a decomposition into $c$ directed hamiltonian cycles, with $0 \leqslant c\leqslant m-2$. If $\vec{C}_n \wr \overline{K}_m$ admits a $c$-twined 2-factorization, then $\vec{C}_n \wr H$ is hamiltonian decomposable. 
\end{proposition}

\begin{proof} Let

\begin{center}
 $D=\{(\sigma_{(k,1)}, \sigma_{(k,2)}, \ldots, \sigma_{(k,n)})\ |  \ k \in \mathds{Z}_m \ \textrm{and}\  j=1,2, \ldots, n \}$ 
 \end{center}
\noindent  is a  $c$-twined 2-factorization of $\vec{C}_n \wr \overline{K}_m$. Furthermore, let
 
     \smallskip
{\centering
  $ \displaystyle
    \begin{aligned} 
&D_T&&=\{(\sigma_{(k,1)}, \sigma_{(k,2)}, \ldots, \sigma_{(k,n)})\ | \ k=0, 1,\, \ldots, c-1\}\ \textrm{and} \\
&D_H&&=\{(\sigma_{(k,1)}, \sigma_{(k,2)}, \ldots, \sigma_{(k,n)})\ |\ k=c, c+1,\, \ldots, m-1\}\\ 
     \end{aligned}
  $ 
\par}
   \smallskip

\noindent  form a partition of $D$ such that $D_T$ is comprised of $c$ truncated hamiltonian $n$-tuples and $D_H$ is comprised of $m-c$ hamiltonian $n$-tuples. For each $i\in \mathds{Z}_n$, we embed $H$ into each $V_i$ so that the set of $c$ out-neighbours of $i_{m-1}$ is 

\begin{center}
$\{i_{(m-1)^{\sigma_{(0, {i-1})}}}, i_{(m-1)^{\sigma_{(1, {i-1})}}}, \ldots, i_{(m-1)^{\sigma_{(c-1, {i-1})}}} \}$. 
\end{center}

\noindent Let $H_i$ be the image of the embedding of $H$ into $V_i$ described above, and let $D_i$ be the decomposition of $H_i$ into directed hamiltonian cycles inherited from $H$.  

Next, for each $L_k\in D_T$, where $L_k=(\sigma_{(k,1)}, \sigma_{(k,2)}, \ldots, \sigma_{(k,n)})$, we construct a spanning subdigraph of $\vec{C}_n \wr H$, which we call $\Gamma_k$. The subdigraph $\Gamma_k$ is comprised of the directed 2-factor described by $L_k$. In addition, the digraph $\Gamma_k$ also contains the directed hamiltonian cycle in $D_i$ that contains the arc $(i_{m-1}, i_{(m-1)^{\sigma_{(k,i-1)}}})$. Because $L_k$ is a truncated hamiltonian $n$-tuple, it follows from the proof of Lemma \ref{lem:central} that $\Gamma_k$ is hamiltonian decomposable. 

If $k_1\neq k_2$, then $\Gamma_{k_1}$ and  $\Gamma_{k_2}$ are arc-disjoint. Digraphs $\Gamma_{k_1}$ and  $\Gamma_{k_2}$ do not share a horizontal arc because because, for each $i \in \mathds{Z}_m$, they contain different directed hamiltonian cycles of $D_i$. 

For each $k \in \{c, \ldots, m-1\}$, let $C^k= (\sigma_{(k,1)}, \sigma_{(k,2)}, \ldots, \sigma_{(k,n)}) \in D_H$. Then $C^k$ is a directed hamiltonian cycle of $\vec{C}_n \wr \overline{K}_m$, and hence of $\vec{C}_n\wr H$. Since $D$ is a directed 2-factor of $\vec{C}_n \wr \overline{K}_m$, each non-horizontal arc of $\vec{C}_n \wr H$ appears precisely once in 
$\{\Gamma_0, \Gamma_1, \ldots, \Gamma_{c-1}\}\cup \{C^0, C^1, \ldots, C^{m-1}\}$. Furthermore,  each horizontal arc of $\vec{C}_n\wr H$ appears in a subdigraph in $\{\Gamma_0, \Gamma_1, \ldots, \Gamma_{c-1}\}$ exactly once. Therefore, we see that $\vec{C}_n \wr H=\Gamma_0\oplus\Gamma_1\oplus\cdots \oplus \Gamma_{c-1} \oplus C^c\oplus C^{c+1}\oplus\cdots\oplus C^{m-1}$.  

Since each subdigraph in  $\{\Gamma_0, \Gamma_1, \ldots, \Gamma_{c-1}\}$ is a spanning subdigraph of $\vec{C}_n \wr H$ that is hamiltonian decomposable, $\vec{C}_n \wr H$ is hamiltonian decomposable. \end{proof}

By Proposition \ref{prop:twin}, to show that $\vec{C}_n\wr H$ is hamiltonian decomposable when  $|V(H)|=m$ and $c$ is the number of directed hamiltonian cycles in a decomposition of $H$, it suffices to find a $c$-twined 2-factorization of $\vec{C}_n \wr \overline{K}_m$. In Lemma \ref{lem:red2}, we reduce the latter problem to the case $n=2$. 

\begin{lemma}
\label{lem:red2}
Let $0 \leqslant c \leqslant m-2$ and $n \geqslant 4$ be even. If $\vec{C}_2\wr \overline{K}_m$ admits a $c$-twined 2-factorization, then $\vec{C}_n \wr \overline{K}_m$ also admits a $c$-twined 2-factorization.  
\end{lemma}

\begin{proof} Let $D$ be a  $c$-twined 2-factorization of $\vec{C}_2\wr \overline{K}_m$. Then there exist two regular permutation sets of order $m$, $R_1'=\{\sigma_0, \sigma_1, \ldots, \sigma_{m-1}\}$ and $R_2'=\{\mu_0, \mu_1, \ldots, \mu_{m-1}\}$. Without loss of generality, we let $D=\{(\sigma_i, \mu_i)\ | \ i=0,1, \ldots, m-1\}$ and assume that $D_T=\{(\sigma_i, \mu_i)\ | \ i=0,1, \ldots, c-1\}$ is the subset of $D$ comprised of $c$ truncated hamiltonian pairs while $D_H=\{(\sigma_i, \mu_i)\ | \ i=c,c+1, \ldots, m-1\}$ is the subset of $D$ that contains $m-c$ hamiltonian pairs. 

For $i \in \{1, \ldots, n-3\}$, let $R_i=\Pi_m$, where the cyclic group $\Pi_m$ is defined in Definition \ref{def:cyc}; moreover, let $R_{n-1}=R'_1$, and $R_n=R'_2$. We then construct the following set of $n$-tuples: $D_T'=\{(\pi_{i+1}, \pi_{-i-1}, \ldots,\pi_{i+1}, \pi_{-i-1}, \sigma_i, \mu_i)\ |\ i=0,1, \ldots, c-1\}.$

\noindent Observe that no $n$-tuple of $D_T'$ contains an $(m-1)$-stabilizer. Furthermore, we see that, since $(\mu_i, \sigma_i)$ is a truncated hamiltonian pair, we have
$T(\hat{\pi}_{i+1}\hat{\pi}_{-i-1}\cdots\hat{\pi}_{i+1}\hat{\pi}_{-i-1}\hat{\sigma}_i\hat{\mu}_{i})=T(\hat{\sigma}_i\hat{\mu}_{i})=2$. As a result, each $n$-tuple of $D_T'$ is a truncated hamiltonian $n$-tuple.

Next, we construct the following set of $n$-tuples: $D'_H=\{(\pi_{i+1}, \pi_{-i-1}, \ldots, \pi_{i+1}, \pi_{-i-1}, \sigma_i, \mu_i)\ | \ i=c,c+1, \ldots, m-1\}$. Observe that  $T(\pi_{i+1}\pi_{-i-1} \pi_{i+1}\pi_{-i-1}\cdots\sigma_i\mu_i)=T(\sigma_i\mu_i)=1$. Therefore, each $n$-tuple in $D'_H$ is a hamiltonian $n$-tuple. 

For each $j \in \{1,2, \ldots, n\}$, each element of $R_j$ appears in exactly one $n$-tuple of $D'=D'_H\cup D'_T$. Since each $R_j$ is a regular permutation set of order $m$, it follows that $D'$ is a directed 2-factorization of $\vec{C}_n \wr \overline{K}_m$. In addition, note that $D_T'$ contains $c$ truncated hamiltonian $n$-tuples, the set $D_H'$ contains $m-c$ hamiltonian $n$-tuples, and that $D_T'$ and $D_H'$ are disjoint. Therefore, the set $D'$ is the desired $c$-twined 2-factorization of  $\vec{C}_n \wr \overline{K}_m$.  \end{proof}

In Corollary \ref{cor:main} below, we summarize the implications of Proposition \ref{prop:twin} and Lemma \ref{lem:red2}

\begin{corollary}
\label{cor:main}
Let $H$ be a digraph that admits a decomposition into $c$ directed hamiltonian cycles, and $n$ be an even integer. If the digraph $\vec{C}_2 \wr \overline{K}_m$ admits a $c$-twined 2-factorization, then $\vec{C}_n \wr H$ is hamiltonian decomposable. 
\end{corollary}

The implications of Corollary \ref{cor:main} are as follows. To show that $\vec{C}_n \wr H$ is hamiltonian decomposable, where $|V(H)|=m$ and $H$ admits a decomposition into $c$ directed hamiltonian cycles,  it suffices to construct two regular permutation sets of order $m$ from which we can form a $c$-twined 2-factorization of $\vec{C}_2\wr \overline{K}_m$ for all $2 \leqslant c \leqslant m-2$. If $c\in \{1, m-1\}$, then $H$ is a complete symmetric digraph or an $m$-cycle because $H$ is strict. These extremal cases are addressed in \cite{Me} and are thus not considered in our general approach. Lastly, the case $c=0$ is addressed in Lemma \ref{lem:empty}. 

The remainder of this paper is structured as follows. In Section \ref{sec:meven}, we construct a $c$-twined 2-factorization for $\vec{C}_2\wr \overline{K}_m$ when $m$ and $c$ are even. Corollary \ref{cor:main} and Lemma \ref{lem:redCn} will jointly imply that $G \wr H$ is hamiltonian decomposable when $|V(H)|$ is even and $H$ admits a decomposition into an even number of directed hamiltonian cycles. We then show that this result can be extended to the case $|V(H)|$ even, $c$ odd, and $G\neq \vec{C}_n$.  Then, in Section \ref{sec:modd}, we show that $\vec{C}_2\wr \overline{K}_m$ admits a $c$-twined 2-factorization for all $2 \leqslant c \leqslant m-2$ and all odd $m\geqslant 5$. Corollary \ref{cor:main} and Lemma \ref{lem:redCn}  imply that  $G \wr H$ is hamiltonian decomposable when  $|V(H)|\geqslant 5$ is odd. 

\section{The case $m$ is even}
\label{sec:meven}

Before we proceed with the case $m$ is even, we give a brief outline of the proof of Proposition \ref{prop:evenx2}, which also describes the proofs of Propositions \ref{prop:m1even}-\ref{prop:5mod6}.  Our goal is to construct a $c$-twined 2-factorization of  $\vec{C}_2 \wr \overline{K}_m$, call it $D$, for given $c$ and $m$ values. We will start with two regular permutations sets of order $m$, $R_1$ and $R_2$. The set $D$ will be constructed such that $D \subset R_1\times R_2$. We then construct two sets of pairs, $D_T$ and $D_H$, such that $D_T$ contains $c$ pairs and $D_H$ contains $m-c$ pairs. We then let $D=D_T\cup D_H$ and show that $D$ is indeed a 2-factorization of $\vec{C}_2\wr \overline{K}_m$ by verifying that each permutation of $R_1$ and $R_2$ appears exactly once in $D$. This implies that $D_T$ and $D_H$ are disjoint. The most important and complicated step is to then verify that all pairs in $D_T$ are truncated hamiltonian pairs and all pairs in $D_H$ are hamiltonian pairs. For truncated hamiltonian pairs of permutations, we compute the product of their respective truncation and show that the resulting permutation has two cycles in its disjoint cycle notation. Similarly, for hamiltonian pairs, we show that their product is a permutation on a single cycle. 

\begin{proposition}
\label{prop:evenx2}
Let $m$ and $c$ be even integers such that $m \geqslant 4$ and $2 \leq c \leq m-2$. The digraph $\vec{C}_2 \wr \overline{K}_m$ admits a $c$-twined 2-factorization. 
\end{proposition}

\begin{proof} 
\noindent \underline{Case 1:} $2 \leqslant c \leqslant m-4$.  We will construct a set of $m$ pairs by using the following two regular permutation sets of order $m$:

   \smallskip
{\centering
  $ \displaystyle
    \begin{aligned} 
&R_1=(1, m-2)\cdot \Pi_m=\{ \mu_i=(1,m-2) \pi_i\ | \ i=0,1, \ldots, m-1 \};&R_2=\Pi_m=\{\pi_0, \pi_1, \ldots, \pi_{m-1}\}. 
      \end{aligned}
  $ 
\par}
   \smallskip 

Refer to Definition \ref{def:cyc} for the definition of the cyclic group $\Pi_m$. Note that $\mu_0$ and $\pi_0$ are the $(m-1)$-stabilizers of $R_1$ and $R_2$, respectively.  Otherwise, if $i \neq 0$, then $(m-1)^{\mu_i}=(m-1)^{\pi_i}=i-1$.  It follows that, if $i \neq 0$, then  $\hat{\mu}_i=\mu_i (m-1, i-1)$ and $\hat{\pi}_i=\pi_i \, (m-1, i-1)$. 
 
Let $c=2t$ where $1 \leqslant t \leqslant \frac{m-4}{2}$, $\mathds{I}=\{3,5,7, \ldots, m-3\}$, and let $M_t$ be a subset of size $t$ of $\mathds{I}$. Observe that $\mathds{I}$ contains $\frac{m-4}{2}$ elements. Below, we construct two sets of pairs, $D_{T}$ and $D_H$, that contain $c$ pairs and $m-c$ pairs, respectively:

   \smallskip
{\centering
  $ \displaystyle
    \begin{aligned} 
&D_{T}&&=&&\{(\mu_i, \pi_{m-i+1}), (\mu_{i+1}, \pi_{m-i-2})\ | \ i \in M_t\};\\
&D_H&&=&&\{(\mu_i, \pi_{m-i-2}), (\mu_{i+1}, \pi_{m-i+1}) \ |\ i \ \textrm{odd}, i \in \mathds{I}/M_t\}\cup\{(\mu_0, \pi_{2}), (\mu_{m-1}, \pi_{m-1})\}.
      \end{aligned}
  $ 
\par}
   \smallskip 

Let $D=D_T\cup D_H$. Each permutation in $R_1$ appears exactly once as the first permutation of a pair in $D$. Similarly, each permutation in $R_2$ appears exactly once as the second permutation of a pair in $D$. Therefore, the set $D$ is a 2-factorization of $\vec{C}_2\wr\overline{K}_m$. We now verify that each pair in $D_T$ and $D_H$ is a truncated hamiltonian pair and a hamiltonian pair, respectively.  We do so in two steps. 

\noindent STEP 1.  We show that each of $(\mu_i, \pi_{m-i+1})$ and $(\mu_{i+1}, \pi_{m-i-2})$ is a truncated hamiltonian pairs. Note that $i$ is necessarily odd. Therefore:

   \smallskip
{\centering
  $ \scriptstyle
    \begin{aligned} 
&\hat{\mu}_{i}\hat{\pi}_{m-i+1}&&=&&(1,m-2) \pi_i (m-1, i-1)  \, \pi_{-i+1} (m-1, m-i)\\
&&&=&&(0,1,m-i, m-i+1, \ldots, m-2, 2, 3, 4, \ldots, m-i-1)(m-1);\\
&\hat{\mu}_{i+1}\hat{\pi}_{m-i-2}&&=&&(1, m-2)\, \pi_{i+1}\, (m-1, i)  \, \pi_{-i-2}\, (m-1, m-i-3)\\
&&&=&&(0,m-i-3, m-i-4, m-i-5, \ldots, 2, 1, m-3, m-4, m-5,\\
&&&&&  \ldots, m-i-2, m-2)(m-1). \\
      \end{aligned}
  $ 
\par}
   \smallskip

\noindent Observe that $T(\hat{\mu}_{i}\hat{\pi})=T(\hat{\mu}_{i+1}\hat{\pi}_{m-i-2})=2$. Therefore, the set $D_T$ is comprised of $2t$ truncated hamiltonian pairs.  

\noindent STEP 2. We show that each pair in $D_H$ is a hamiltonian pair. See below:

   \smallskip
{\centering
  $ \displaystyle
    \begin{aligned} 
&\mu_i\, \pi_{m-i-2} &&=&&(1, m-2) \pi_{i}\pi_{m-i-2}=(1, m-2) \pi_{-2}\\
&&&=&&(0, m-2, m-1, m-3, m-5, \ldots, 3, 1, m-4, m-6, \ldots, 2);\\
&\mu_{i+1}\pi_{m-i+1}&&=&&(1, m-2) \pi_{i+1}\pi_{m-i+1}=(1, m-2) \pi_2\\
&&&=&&( 0, 2, 4, 6, \ldots, m-2, 3, 5, 7, \ldots, m-1, 1);\\
&\mu_{0}\pi_{2}&&=&&( 0, 2, 4, 6, \ldots, m-2, 3, 5, 7, \ldots, m-1, 1);\\
&\mu_{m-1}\pi_{m-1}&&=&&(0, m-2, m-1, m-3, m-5, \ldots, 3, 1, m-4, m-6, \ldots, 2\\
      \end{aligned}
  $ 
\par}
   \smallskip 
 
 \noindent Since $T(\mu_i \pi_{m-i-2} )=T(\mu_{i+1}\pi_{m-i+1})=T(\mu_0\pi_2)=1$, it follows that each pair of permutations in $D_H$ is a hamiltonian pair. 
 
Since $D_T$ and $D_H$ are disjoint, the set $D$ is the desired $c$-twined 2-factorization of $\vec{C}_2\wr \overline{K}_m$.

\noindent \underline{Case 2}: $c=m-2$. We will use the following two regular permutation sets of order $m$:

   \smallskip
{\centering
  $ \displaystyle
    \begin{aligned} 
&R_1=(0,m-1) \cdot \Pi_m=\{ \omega_i=(0,m-1) {\pi}_i\ | \ i=0,1, \ldots, m-1\}; &R_2= \Pi_m=\{\pi_0, \pi_1, \ldots, \pi_{m-1}\}. 
      \end{aligned}
  $ 
\par}
   \smallskip 
 
\noindent Permutations $\omega_{m-1}$ and $\pi_0$ are the $(m-1)$-stabilizers of $R_1$ and $R_2$, respectively. Otherwise, if $i \neq m-1$, then $\hat{\omega}_i=\omega_i (m-1, i)$.  If $i \neq 0$, then  $\hat{\pi}_i=\pi_i\, (m-1, i-1)$. Below, we construct two sets of pairs, $D_{T}$ and $D_H$, that contain $c$ pairs and $m-c$ pairs, respectively:

   \smallskip
{\centering
  $ \displaystyle
    \begin{aligned} 
&D_T&&=&&\{(\omega_i, \pi_{m-i-1})\ |\ i \in \mathds{Z}_m/\{0, m-2, m-1 \}\} \cup \{(\omega_{0}, \pi_{1})\}; \\
&D_H&&=&&\{(\omega_{m-1}, \pi_{m-1}), (\omega_{m-2}, \pi_{0})\}.
      \end{aligned}
  $ 
\par}
   \smallskip 
 
\noindent It is not too hard to see that $D=D_T \cup D_H$ is a 2-factorization of $\vec{C}_2\wr \overline{K}_m$. 

\noindent STEP 1. We show that all pairs in $D_T$ are truncated hamiltonian pairs. To do so, we consider two sub-cases. 

\noindent SUB-CASE 1: $(\omega_{0}, \pi_{1})$.  We see that $\hat{\omega}_0=id$ and $T(\hat{\pi}_1)=2$.  This means that $(\omega_{0}, \pi_{1})$ is a truncated hamiltonian pair.

\noindent SUB-CASE 2: $(\omega_i, \pi_{m-i-1})$ where  $i \in \mathds{Z}_m/\{0, m-2, m-1 \}$. Then:

  \smallskip
{\centering
  $ \displaystyle
    \begin{aligned} 
    &\hat{\omega}_{i}\hat{\pi}_{m-i-1}&&=(0,m-1) \,  \pi_{i}\, (m-1, i) \ \pi_{-i-1}\, (m-1, -i-2) \\\
    &&&=(0,m-2, m-3, \ldots, 2,1)(m-1). 
      \end{aligned}
  $ 
\par}
   \smallskip 
 
In conclusion, all pairs in  $D_T$ are truncated hamiltonian pairs. 

\noindent STEP 2. We show that each pair in  $D_H$ is a hamiltonian pair. See below:
 
    \smallskip
{\centering
  $ \displaystyle
    \begin{aligned} 
&\omega_{m-1}\pi_{m-1}&&=\omega_{m-2}\pi_{0}=&&(0,m-3, m-5, m-7, \ldots, 1, m-1, m-2, m-4, \ldots, 2).\\
      \end{aligned}
  $ 
\par}
   \smallskip 

\noindent Therefore, both pairs in $D_H$ are hamiltonian. 

Since $D_H$ and $D_H$ are disjoint, the set $D$ is the desired $(m-2)$-twined 2-factorization. \end{proof}

We are not able to establish the existence of a $c$-twined 2-factorization of  $\vec{C}_2 \wr \overline{K}_m$ when $m$ is even and $c$ is odd. However, in the second case of the proof of Theorem \ref{thm:geneven}, we show that $G \wr H$ is hamiltonian decomposable when $|V(G)|$ and $|V(H)|$ are even, $c$ is odd, and $G$ is not a directed cycle. 

\begin{theorem}
\label{thm:geneven}
Let $G$ and $H$ be hamiltonian decomposable digraphs such that $|V(G)|=n$, where $n$ is even, and $|V(H)|=m$, where $m \geqslant 4$ and $m$ is even. Furthermore, assume that $H$ admits a decomposition into $c$ directed hamiltonian cycles where $2\leqslant c\leqslant m-2$. The digraph $G \wr H$ is hamiltonian decomposable except possibly when $G=\vec{C}_n$ and $c$ is odd. 
\end{theorem}

\begin{proof}  \underline{Case 1}: $c$ is even.  Proposition \ref{prop:evenx2}, in conjunction with Corollary \ref{cor:main}, implies that $\vec{C}_n \wr H$ is hamiltonian decomposable.  Lemma \ref{lem:redCn} then implies that $G \wr H$ is hamiltonian decomposable. 

\noindent \underline{Case 2}: $c$ is odd and $G \neq \vec{C}_n$. This means that $G$ admits a decomposition into at least two directed hamiltonian cycles. Let $H=\Gamma \oplus \vec{C}_m$ where $\Gamma$ is the digraph obtained from $H$ by removing the arcs of a single directed hamiltonian cycle. Hence, the digraph $\Gamma$ admits a decomposition into $c-1$ directed hamiltonian cycles, where $c-1$ is even. By our assumption, we have that $G=\vec{C}_n \oplus \vec{C}_n\oplus \cdots \oplus \vec{C}_n$. It follows from the properties of the wreath product  that  

    \smallskip
{\centering
  $ \displaystyle
    \begin{aligned} 
G\wr H&=(\vec{C}_n \oplus \vec{C}_n\oplus \cdots \oplus \vec{C}_n)\wr (\Gamma \oplus\vec{C}_m)\\
&=(\vec{C}_n\wr \Gamma)\oplus( \vec{C}_n \wr \vec{C}_m) \oplus (\vec{C}_n\wr\overline{K}_m)\oplus \cdots \oplus (\vec{C}_n\wr\overline{K}_m).
      \end{aligned}
  $ 
\par}
   \smallskip

If $n=2$, then $G=\vec{C}_2$ since $G$ is strict. Since $G \neq \vec{C}_n$, we may assume that $n\geqslant 4$. Theorem \ref{thm:2} implies that $\vec{C}_n \wr \vec{C}_m$ is hamiltonian decomposable when $m \geqslant 4$ while Lemma \ref{lem:empty} implies that $\vec{C}_n \wr \overline{K}_m$ is hamiltonian decomposable for all $m, n \geqslant 2$. Since we established that $\vec{C}_n \wr\Gamma$ is hamiltonian decomposable in Case 1, all spanning subdigraphs in the above decomposition of $G\wr H$ are hamiltonian decomposable and thus, $G\wr H$ is also hamiltonian decomposable. \end{proof}

\section{The case $m$ is odd}
\label{sec:modd}

In this section, we establish the existence of a $c$-twined 2-factorization of $\vec{C}_2\wr \overline{K}_m$ for all odd $m\geqslant 5$ and $3\leqslant c \leqslant m-2$. Note that, if $m=3$, then $H$ must be a directed 3-cycle or the complete symmetric digraph on three vertices and thus, we do not consider the case $m=3$. First, we construct a regular permutation set of order $m$ for all odd $m \geqslant 5$ that we will use in all of our constructions. To construct this set, we first introduce the following set of $m-1$ permutations. 

\begin{definition}
\label{def:gammas} \rm
Let $\gamma_1$ be the permutation $(0,1,\ldots, m-2)(m-1)$ in $S_m$. We call the permutation group $G_{m-1}=\langle \gamma_1 \rangle$ on $\mathds{Z}_m$ the \textit{truncated cyclic group of order $m-1$}. In addition, for any $i \in \mathds{Z}$, we let $\gamma_i= \gamma_1^i$. 
\end{definition}

See \ref{App2} for an example of $G_{m-1}$ when $m=13$. In Lemma \ref{lem:compgam} below, we establish some properties of the truncated cyclic group of order $m-1$ that will be key in our computations.

\begin{lemma} 
\label{lem:compgam}

Let $m$ be an odd integer such that $m\geqslant 5$, and let $G_{m-1}$ be the truncated cyclic group of order $m-1$. Then
\begin{enumerate} [label=\textbf{(S\arabic*)}]
\item $\gamma_{m-i}=\gamma_{-i+1}$ for all $i \in \mathds{Z}$;
\item If $i \in\{\pm 1,\pm 2, \ldots, \pm (m-2)\}$ and $s\in \mathds{Z}_m$, then 

$$
s^{\gamma_i}=
\begin{cases}
s+i, &\textrm{if}\ 0\leqslant s+i<m-1;\\
s+i-(m-1), &\textrm{if}\  s+i\geqslant m-1;\,\\
s+i+(m-1), &\textrm{if}\ s+i<0;\\
m-1, & \textrm{if}\ s=m-1.
\end{cases}
$$
\end{enumerate}
\end{lemma}

\begin{proof} \noindent {\bf(S1)} Observe that $\gamma_1^m=\gamma_1$. It follows that  $\gamma_{m-i}=\gamma_1^{m-i}=\gamma_1^{m}\gamma_1^{-i}=\gamma_1^{-i+1}=\gamma_{-i+1}$.

\noindent {\bf(S2)} Clearly, $(m-1)^{\gamma_i}=m-1$. Therefore, we assume that $s \in \{0, \ldots, m-2\}$, and let $t=s^{\gamma_i}$. By definition, $t=s+i \ (\textrm{mod}\ m-1)$. Since $-(m-2) \leqslant s+i\leqslant 2(m-2)$, we know that $t \in \{s+i-(m-1), s+i, s+i+(m-1)\}$ and the result follows.  \end{proof}

In Construction \ref{cons:Fodd} below, we use Definition \ref{def:gammas} to construct a set of $m$ permutations for all odd $m\geqslant 5$.

\begin{construction}
\label{cons:Fodd}
\rm
Let $m=2k+1$ where $k \geqslant 2$. We construct the set of permutations $\mathcal{F}_{m}=\{\sigma_0, \sigma_1, \ldots, \sigma_{m-1}\}$ as follows. If $i \not\in\{0, m-3, m-1\}$, then $\sigma_i$ is defined as described according to the following two cases.\\

\noindent Case 1: $i=2j+1$ and $0\leqslant j \leqslant k-1$: $\sigma_i=\gamma_i (m-1, j+2)$.\\

\noindent Case 2:  $i=2j$ and $1\leqslant j \leqslant k-2$ : $\sigma_i=\gamma_i (m-1, k+j+1)$ .\\

We also let $\sigma_{m-3}=\gamma_{m-3} (m-1, 0)$ and $\sigma_0=id$.  Lastly, we define $\sigma_{m-1}$ as follows. For $a\in \mathds{Z}_m$, we let

\vspace{-0.5cm}
\begin{center}
 $$  a^{\sigma_{m-1}}=\begin{cases}
m-a+1, &\textrm{if}\  3\leqslant a \leqslant k  ;\\
m-a+2, &\textrm{if}\  k+2\leqslant a \leqslant m-2  ;\\
3, &\textrm{if}\ a=0  ;\\
2, & \textrm{if}\ a=1;\\
0,& \textrm{if}\ a=2 ;\\
m-1, &\textrm{if}\ a=k+1;\\
1, &\textrm{if}\ a=m-1.\\
\end{cases} $$
\end{center}

\noindent We note that, if $m \geqslant 7$, then $\sigma_{m-1}$ can be written in cycle notation as follows: $\sigma_{m-1}=(0, 3, m-2, 4, m-3, 5, \ldots, k+3, k, k+2, k+1, m-1, 1, 2)$. If $m=5$, then $\sigma_{4}=(0,3, 4,1,2)$. Hence $\sigma_{m-1}$ is comprised of a single cycle.  \hfill $\square$
\end{construction}

See \ref{App2} for an example of Construction \ref{cons:Fodd} when $m=13$. We make the following remark, which is key to our computations. 

\begin{remark}
\label{rem:ned}
If $\sigma_i \in \mathcal{F}_m$ such that $i\not\in \{0, m-1\}$, then $\hat{\sigma}_i=\gamma_i$. 
\end{remark}

Next, we show that $\mathcal{F}_m$ is a regular permutation set of order $m$.

\begin{lemma}
\label{lem:fodd2}
Let $m\geqslant 5$ be an odd integer. The set of permutation $\mathcal{F}_{m}$ is a regular permutation set of order $m$. 
\end{lemma}

\begin{proof}
For each $a \in \mathds{Z}_m$, we show that $\{a^{\sigma_i}\ | \ i=0,1,\ldots, m-1\}=\mathds{Z}_m$. We assume that $\gamma_i \in G_{m-1}$ where $\gamma_i$ is defined in Definition \ref{def:gammas}. 

\noindent \underline{Case 1}: $3 \leqslant a\leqslant k$. Using Construction \ref{cons:Fodd}, it can be verified that $a^{\sigma_i}=a^{\gamma_i}$ except if $i=m-1$ or $i=2j$, where $j=k-a+1$. In the first exception, we have that $a^{\sigma_{m-1}}=m-a+1$, and in the second exception, $a^{\sigma_{2j}}=m-1$. Since $\{a^{\gamma_i} \ |\ i=0,1,\ldots, m-2\}=\{0,1,\ldots, m-2\}$ and $a^{\gamma_i}=m-a+1$ for $i=2(k-a+1)$, we conclude that $\{a^{\sigma_i} \ | \ i=0,1,\ldots, m-1\}=\mathds{Z}_m$.

\noindent\underline{Case 2}: $k+2 \leqslant a\leqslant m-2$. Similarly to Case 1,  Construction \ref{cons:Fodd} implies that $a^{\sigma_i}=a^{\gamma_i}$ except if $i=m-1$ or $i=2j+1$, where $j=m-a$. In the first exception, we have that $a^{\sigma_{m-1}}=m-a+2$, and in the second exception, $a^{\sigma_{2j+1}}=m-1$. Since $\{a^{\gamma_i} \ |\ i=0,1,\cdots, m-2\}=\{0,1,\ldots, m-2\}$ and $a^{\gamma_i}=m-a+2$ for $i=2(m-a)+1$, we conclude that $\{a^{\sigma_i} \ | \ i=0,1,\ldots, m-1\}=\mathds{Z}_m$. 

\noindent \underline{Case 3}: $a=0$. Using Construction \ref{cons:Fodd}, it can be verified that $0^{\sigma_i}=0^{\gamma_i}$ except if $i =3$ or $i=m-1$.  In the first exception, we see that $0^{\sigma_3}=m-1$ and in the second exception $0^{\sigma_{m-1}}=3$. 

\noindent \underline{Case 4}:  $a=1$.  Similarly it can be verified that $1^{\sigma_i}=1^{\gamma_i}$ except if $i =1$ or $i=m-1$.  In the first exception, we see that $1^{\sigma_1}=m-1$ and in the second exception $1^{\sigma_{m-1}}=2$. 

\noindent \underline{Case 5}: $a=2$. It can be verified that $2^{\sigma_i}=2^{\gamma_i}$ except if $i =m-3$ or $i=m-1$.  In the first exception, we see that $2^{\sigma_{m-3}}=m-1$ and in the second exception $2^{\sigma_{m-1}}=0$. 

\noindent \underline{Case 6}:  $a=k+1$.  It can be verified that $(k+1)^{\sigma_i}=(k+1)^{\gamma_i}$ except if $i=m-1$.  In this one exception, we see that $(k+1)^{\sigma_{m-1}}=m-1$.

\noindent \underline{Case 7}: $a=m-1$. We see that, if $i=2j$, where $1\leqslant j\leqslant k-2$,  Construction \ref{cons:Fodd} implies that $(m-1)^{\sigma_i}=k+j+1$. Furthermore, if $i=2j+1$, where $0\leqslant j\leqslant k-1$,  then Construction \ref{cons:Fodd} implies that $(m-1)^{\sigma_i}=j+2$. Lastly, we have $(m-1)^{\sigma_{m-3}}=0$, $(m-1)^{\sigma_0}=m-1$, and  $(m-1)^{\sigma_{m-1}}=1$.  In conclusion, we see that $\{a^{\sigma_i} \ | \ i=0,1,\cdots, m-1\}=\mathds{Z}_m$. \end{proof}

To construct the desired $c$-twined 2-factorizations of $\vec{C}_2 \wr \overline{K}_m$, we will use permutations of $\mathcal{F}_m$ of the form $\sigma_{m-i+t}$ where $t \in \{0, 1, \pm 2, -3, -4\}$. To facilitate some of our computations, we will express $\sigma_{m-i+t}$  as a product of specific permutations.

\begin{lemma}
\label{lem:comput}
Let $m=2k+1$ where $k \geqslant 2$, $t \in \{0, 1, \pm 2, -3, -4\}$, and $\sigma_i \in \mathcal{F}_m$, where $i \in \mathds{Z}_m$ such that $i \not\in \{ 0, m-1\}$. If $m-i+t\not\in \{0, m-3, m-1\}$, then we express $\sigma_{m-i+t}$ as a product of certain permutations described in Table \ref{tab:2}. 
\end{lemma}

\begin{table} [htpb!]
\begin{center}
\begin{tabular}{ |l|l|l| } 
 \hline
 & $i=2j+1$& $i=2j$ \\ \hline
 Case 1: $\sigma_{m-i+2}$ & $\gamma_{-i+3}(m-1, m-j+1)$ &$\gamma_{-i+3}(m-1, k-j+3)$  \\  \hline
Case 2: $\sigma_{m-i+1}$ & $\gamma_{-i+2}(m-1, k-j+2)$ &$\gamma_{-i+2}(m-1, m-j+1)$ \\  \hline
Case 3: $\sigma_{m-i}$ & $\gamma_{-i+1}(m-1, m-j)$ &$\gamma_{-i+1}(m-1, k-j+2)$ \\  \hline
Case 4: $\sigma_{m-i-2}$ & $\gamma_{-i-1}(m-1, m-j-1)$ &$\gamma_{-i-1}(m-1, k-j+1)$  \\  \hline
Case 5: $\sigma_{m-i-3}$ & $\gamma_{-i-2}(m-1, k-j)$ &$\gamma_{-i-2}(m-1, m-j-1)$   \\  \hline
Case 6: $\sigma_{m-i-4}$ & $\gamma_{-i-3}(m-1, m-j-2)$ &$\gamma_{-i-3}(m-1, k-j)$ \\  \hline
\end{tabular}
\end{center}
\caption{Computation of $\sigma_{m-i+t}$ for $t \in \{0, 1, \pm 2, -3, -4\}$.}
\label{tab:2}
\end{table}

\begin{proof} Recall that $m=2k+1$, where $k \geqslant 2$. For Case 1, we consider two sub-cases.

\noindent SUB-CASE 1.1: $i=2j+1$. Lemma \ref{lem:compgam} implies that $\gamma_{m-i+1}=\gamma_{-i+2}$. We see that $m-i+2=2k+1-(2j+1)+2=2(k-j+1)$. Since $2(k-j+1)$ is even, and  $m-i+2 \not\in\{0, m-3, m-1\}$, it follows from Construction \ref{cons:Fodd} that:

   \smallskip
{\centering
  $ \displaystyle
    \begin{aligned} 
&\sigma_{m-i+1}&&=&&\sigma_{2(k-j+1)}=\gamma_{m-i+1}(m-1, k+(k-j+1)+1)\\
&&&=&&\gamma_{-i+2}(m-1, 2k+1-j+1)=\gamma_{-i+2}(m-1, m-j+1).
      \end{aligned}
  $ 
\par}
   \bigskip

\noindent SUB-CASE 1.2: $i=2j$. We see that $m-i+2=2k+1-(2j)+2=2(k-j+1)+1$. Since $2(k-j+1)+1$ is odd,  $m-i+2 \not\in\{0, m-3, m-1\}$,  it follows from Construction \ref{cons:Fodd} that:

\begin{center}
$\sigma_{2(k-j+1)+1}=\gamma_{m-i+1}(m-1, (k-j+1)+2)=\gamma_{-i+2}(m-1, k-j+3)$.
\end{center}

A similar reasoning applies to the remaining five cases. \end{proof}

\begin{theorem}
\label{thm:oddmc}
Let $G$ and $H$ be hamiltonian decomposable digraphs such that  $|V(H)|=m$ is odd and $m\geqslant 5$. Furthermore, assume that $H$  admits a decomposition into $c$ directed hamiltonian cycles where $3 \leqslant c \leqslant m-2$. The digraph $G \wr H$ is hamiltonian decomposable.
\end{theorem}

\begin{proof} We partition our constructions into eleven cases. Each case is described in Table \ref{tab:1} where we refer the reader to the appropriate result(s) that imply that $\vec{C}_2 \wr \overline{K}_m$ admits a $c$-twined 2-factorization for all corresponding $c$ and $m$ values. 

\begin{table} [htpb!]
\begin{center}
\begin{tabular}{ |l|c|p{5cm}| } 
 \hline
 Conditions on $m$ & $c$& Reference \\ \hline
$m \equiv 1 \ (\textrm{mod} \ 4)$& even & Proposition \ref{prop:m1even}\\  \hline
$m \equiv 7,11 \ (\textrm{mod} \ 12)$ & even & Proposition \ref{prop:m7even}\\  \hline
$m \equiv 3 \ (\textrm{mod} \ 12)$ & even & Proposition \ref{prop:m3even}\\  \hline
$m \equiv 1\ (\textrm{mod}\ 6)$ and  $m \equiv 1 \ (\textrm{mod}\ 4)$ & odd & Propositions \ref{prop:1mod4<} and  \ref{prop:1mod6}\\  \hline
$m \equiv 1\ (\textrm{mod}\ 6)$ and $m \equiv 3 \ (\textrm{mod}\ 4)$ & odd & Propositions \ref{prop:3mod4} and \ref{prop:1mod6} \\  \hline
$m \equiv 3\ (\textrm{mod}\ 6)$ and $m \equiv 1 \ (\textrm{mod}\ 4)$ & odd & Propositions \ref{prop:1mod4<} and  \ref{prop:m3cong6}\\  \hline
$m \equiv 3\ (\textrm{mod}\ 6)$ and $m \equiv 3 \ (\textrm{mod}\ 4)$& odd & Propositions \ref{prop:3mod4} and \ref{prop:m3cong6} \\  \hline
$m \equiv 5\ (\textrm{mod}\ 6)$ and $m \equiv 1 \ (\textrm{mod}\ 4)$ & odd & Propositions \ref{prop:1mod4<} and  \ref{prop:5mod6}\\  \hline
$m \equiv 5\ (\textrm{mod}\ 6)$ and $m \equiv 3 \ (\textrm{mod}\ 4)$ & odd & Propositions \ref{prop:3mod4} and \ref{prop:5mod6} and Lemma \ref{lem:m11}\\  \hline
\end{tabular}
\end{center}
\caption{A table of all cases considered in the proof of Theorem \ref{thm:oddmc}.}
\label{tab:1}
\end{table}

It is easy to see that Propositions \ref{prop:m1even}- \ref{prop:m3even} jointly imply the existence of a $c$-twined 2-factorization of $\vec{C}_2\wr \overline{K}_m$ for all even $c$ such that $2\leqslant c\leqslant m-3$ and odd $m \geqslant 5$.  

For $c$ is odd, an additional step is required to verify that the given results apply to all applicable $c$ and $m$ values. As an example, we will consider the case $m \equiv 1\ (\textrm{mod}\ 6)$ and $m \equiv 3 \ (\textrm{mod}\ 4)$. The digraph $\vec{C}_2 \wr \overline{K}_m$ admits a $c$-twined 2-factorization if $ c\leqslant \frac{m-5}{2}$ or $c\geqslant \frac{m+2}{3}$ by Propositions \ref{prop:3mod4} and \ref{prop:1mod6}, respectively. If $\frac{m-5}{2} <c< \frac{m+2}{3}$, then $m=7$ and $1<c<3$, thus contradicting the assumption that $c$ is odd.  

A similar reasoning applies to all cases where $c$ is odd. Therefore, we have a $c$-twined 2-factorization of $\vec{C}_2\wr \overline{K}_m$ for all odd $c$ such that $3\leqslant c\leqslant m-2$. 

In conclusion, Corollary \ref{cor:main} implies that  $\vec{C}_n \wr H$ is hamiltonian decomposable. Lemma \ref{lem:redCn} then implies that $G \wr H$ is hamiltonian decomposable.  \end{proof}

\subsection{The sub-case $c$ is even}

In this subsection, we construct a $c$-twined 2-factorization of $\vec{C}_2\wr \overline{K}_m$ for all odd $m \geqslant 5$ and all even $c$ such that $2\leqslant c \leqslant m-3$.  We begin with Proposition \ref{prop:m1even}, stated below,  which addresses the case $m \equiv 1\ (\textrm{mod} \ 4)$.

\begin{proposition}
\label{prop:m1even}
Let $m \equiv 1\ (\textrm{mod} \ 4)$ such that $m \geqslant 5$, and $c$ be an even integer such that $2 \leq c \leq m-3$. The digraph $\vec{C}_2 \wr \overline{K}_m$ admits a $c$-twined 2-factorization. 
\end{proposition}

\begin{proof}
Let $m=2k+1$ where $k$ is even. For each $i \in \mathds{Z}_{m-1}$, $\gamma_i \in G_{m-1}$, where $G_{m-1}$ is defined in Definition \ref{def:gammas}. We will construct a $c$-twined 2-factorization of  $\vec{C}_2 \wr \overline{K}_m$ by using the following regular permutation sets of order $m$:

   \smallskip
{\centering
  $ \displaystyle
    \begin{aligned} 
&R_1=\mathcal{F}_m=\{\sigma_0, \sigma_1, \ldots, \sigma_{m-1}\};&&&& R_2=\mathcal{F}_m=\{\sigma_0, \sigma_1, \ldots, \sigma_{m-1}\}.
      \end{aligned}
  $ 
\par}
   \smallskip 

\noindent In each case, the $m$ pairs constructed are elements of $R_1\times R_2$.  We point out that $\sigma_0$ is the $(m-1)$-stabilizer of $R_1$ and $R_2$. 

Let $c=2+2t$, where $0 \leqslant t \leqslant \frac{m-7}{2}$, $\mathds{I}=\{i \ | \ i\equiv 1,2 \ (\textrm{mod}\ 4), 1\leqslant i \leqslant m-7\}$, and $M_t$ be a subset of order $t$ of  $\mathds{I}$. Note that $\mathds{I}$ contains $\frac{m-7}{2}$ elements.  Below, we construct two sets of pairs, $D_{T}$ and $D_H$, that contain $c$ pairs and $m-c$ pairs, respectively:

\smallskip
{\centering
  $ \displaystyle
    \begin{aligned} 
&D_{T}&&=&& \{(\sigma_{m-2}, \sigma_2), (\sigma_{m-3}, \sigma_3)\} \cup \{(\sigma_i, \sigma_{m-i-2}), (\sigma_{i+3}, \sigma_{m-i-3}) \ |\  i\equiv 1 \ (\textrm{mod}\ 4)\ \textrm{and}\ i \in M_t\}\cup \\
&&&&& \{(\sigma_i, \sigma_{m-i}), (\sigma_{i+1}, \sigma_{m-i-3}) \ | \  i\equiv 2 \ (\textrm{mod}\ 4)\ \textrm{and}\ i \in M_t\}; \\
&D_{H}&&=&&\{(\sigma_0, \sigma_{m-1}), (\sigma_{m-1}, \sigma_0), (\sigma_{m-4}, \sigma_1)\} \cup\\
&&&&&\{(\sigma_i, \sigma_{m-i-3}), (\sigma_{i+3}, \sigma_{m-i-2} ) \ | \ i \equiv 1\ (\textrm{mod}\ 4)\ \textrm{and}\  i \in \mathds{I}/M_t\}\cup \\
&&&&&  \{(\sigma_i, \sigma_{m-i-3}), (\sigma_{i+1}, \sigma_{m-i} ) \ | \ i \equiv 2\ (\textrm{mod}\ 4)\ \textrm{and}\ i \in \mathds{I} / M_t\}.
      \end{aligned}
  $ 
\par}
   \bigskip
   
\noindent Observe that $D_T$ contains precisely $2t+2$ pairs of permutations. Let $D=D_T\cup D_H$. Each permutation in $R_1$ appears exactly once as the first permutation of a pair in $D$. Similarly, each permutation in $R_2$ appears exactly once as the second permutation of a pair in $D$. Therefore, the set $D$ is a 2-factorization of $\vec{C}_2\wr\overline{K}_m$. 

Next, to show that $D$ is a $c$-twined 2-factorization of $\vec{C}_2\wr\overline{K}_m$, we must now show that all pairs in $D_T$ and $D_H$ are truncated hamiltonian pairs and hamiltonian pairs, respectively. We do so in two steps. 

\noindent STEP 1. We show that each pair in $D_T$ is a truncated hamiltonian pair by considering three cases. In each case, we show that the product of the truncation of the two permutations is $\gamma_1$ or $\gamma_{-1}$. Since $T(\gamma_1)=T(\gamma_{-1})=2$, the corresponding pair is truncated hamiltonian. 

\noindent \underline{Case 1}: $\{(\sigma_{m-2}, \sigma_2), (\sigma_{m-3}, \sigma_3)\}$. By Lemma \ref{lem:compgam}, we have $\gamma_{m-2}=\gamma_{-1}$ and $\gamma_{m-3}=\gamma_{-2}$. Therefore:
 
 \smallskip
{\centering
  $ \scriptstyle
    \begin{aligned} 
    \hat{\sigma}_{m-2}\hat{\sigma}_{2}=\gamma_{-1}\gamma_{2}=\gamma_1;&&&& \hat{\sigma}_{m-3}\hat{\sigma}_{3}=\gamma_{-2}\gamma_{3}=\gamma_1.
      \end{aligned}
  $ 
\par}
   \smallskip

\noindent \underline{Case 2}: $(\sigma_i, \sigma_{m-i-2})$ and  $(\sigma_{i+3}, \sigma_{m-i-3})$ where $i \in \mathds{I}$ and $i \equiv 1\ (\textrm{mod}\ 4)$.  By Remark \ref{rem:ned}, we have $\hat{\sigma}_{i}=\gamma_{i}$ for all $i \in \{1, \ldots, m-2\}$. Additionally, Lemma \ref{lem:compgam} implies that $\gamma_{m-i-2}=\gamma_{-i-1}$ and $\gamma_{m-i-3}=\gamma_{-i-2}$. Therefore:

   \smallskip
{\centering
  $ \scriptstyle
    \begin{aligned} 
&\hat{\sigma}_{i}\hat{\sigma}_{m-i-2}=\gamma_i\ \gamma_{-i-1}=\gamma_{-1}; &&&&\hat{\sigma}_{i+3}\hat{\sigma}_{m-i-3} =\gamma_{i+3}\ \gamma_{-i-2}=\gamma_{1}.\\
      \end{aligned}
  $ 
\par}
   \smallskip

\noindent \underline{Case 3}: $(\sigma_i, \sigma_{m-i})$ and $(\sigma_{i+1}, \sigma_{m-i-3})$ where $i \in \mathds{I}$ and $i \equiv 2\ (\textrm{mod}\ 4)$. We have:

   \smallskip
{\centering
  $ \scriptstyle
    \begin{aligned} 
&\hat{\sigma}_{i}\hat{\sigma}_{m-i}=\gamma_i\ \, \gamma_{-i+1}=\gamma_{1};&&&& \hat{\sigma}_{i+1}\hat{\sigma}_{m-i-3}=\gamma_{i+1}\ \, \gamma_{-i-2}=\gamma_{-1}.\\
      \end{aligned}
  $ 
  \par}
   \smallskip

\noindent In conclusion, all pairs in $D_T$ are truncated hamiltonian. 

\noindent STEP 2. We now show that each pair of $D_H$ is a hamiltonian pair. 

\noindent \underline{Case 1}: $(\sigma_0, \sigma_{m-1})$ and $(\sigma_{m-1}, \sigma_0)$. Clearly, both are hamiltonian pairs since $\sigma_0=id$ and $T(\sigma_{m-1})=1$.

\noindent \underline{Case 2}: pairs of the form $(\sigma_i, \sigma_{m-i-3})$ where $i \equiv 1 \ (\textrm{mod}\ 4)$ and $1 \leqslant i \leqslant m-4$. Then $i=2j+1$, where $j$ is even and $0 \leqslant j \leqslant k-2$. Since $i$ is odd, and $i \neq m-2$, Lemma \ref{lem:comput} implies that $\sigma_{m-i-3}=\gamma_{-i-2}\, (m-1, k-j)$. Furthermore, Lemma \ref {lem:compgam} implies that $(j+2)^{\gamma_{-i-2}}=m-j-2$ and $(m-j)^{\gamma_i}=j+2$. Therefore:

   \medskip
{\centering
  $ \scriptstyle
    \begin{aligned} 
&\sigma_{i}\sigma_{m-i-3}&&=&&\gamma_i\ (m-1, j+2) \, \gamma_{-i-2}\, (m-1, k-j)\\
&&& =&&(0, m-3, m-5, \ldots, k-j+2, m-1, m-j-2, m-j-4, \ldots, 3,1, m-2, m-4, \\ 
&&&&&\ldots, m-j, k-j, k-j-2, k-j-4, \ldots, 4, 2).\\
      \end{aligned}
  $ 
\par}
\medskip

\noindent Since $T( \sigma_i \sigma_{m-i-3})=1$, the pair $(\sigma_i, \sigma_{m-i-3})$ is a hamiltonian pair.

\noindent \underline{Case 3}: pairs of the form $(\sigma_i, \sigma_{m-i-3})$ where $i \equiv 2 \ (\textrm{mod}\ 4)$ and $2 \leqslant i \leqslant m-7$. Then $i=2j$, where $j$ is odd and $1 \leqslant j \leqslant k-2$. Since $i$ is even, Lemma \ref{lem:comput} implies that $\sigma_{m-i-3}=\gamma_{-i-2}\, (m-1, m-j-1)$. Note that $k$ is even and $j$ is odd. This means that $k-j+1$ is even and $m-j-1$ is odd. Therefore:

   \medskip
{\centering
  $ \scriptstyle
    \begin{aligned} 
&\sigma_{i}\sigma_{m-i-3}&&=&&\gamma_i\ (m-1, k+j+1) \, \gamma_{-i-2}\, (m-1,m-j-1)\\
&&& =&&(0,  m-3, m-5, m-7, \ldots, k-j+1, m-j-1,m-j-3, \ldots, 3, 1, m-2, m-4,  \\
&&&&&  \ldots, m-j+1, m-1, k-j-1, k-j-3, \ldots, 4,2).\\
      \end{aligned}
  $ 
\par}
\medskip

\noindent Since $T( \sigma_i \sigma_{m-i-3})=1$, the pair $(\sigma_i, \sigma_{m-i-3})$ is a hamiltonian pair.

We must now verify that pairs in the set below are hamiltonian pairs:

\begin{center}
$A=\{(\sigma_{i+3}, \sigma_{m-i-2} ) \ | \ i \equiv 1\ (\textrm{mod}\ 4)\ \textrm{and} \ i \in\mathds{I}/M_t\} \cup\{(\sigma_{i+1}, \sigma_{m-i} ) \ | \ i \equiv 2\ (\textrm{mod}\ 4)\ \textrm{and}\ i \in \mathds{I}/M_t\}$.
\end{center}

\noindent It is not too difficult to see that $A \subseteq \{(\sigma_i, \sigma_{m-i+1})\ | \ i \equiv 0,3 \ (\textrm{mod}\ 4)\ \textrm{and}\ 3\leqslant i \leqslant m-5\}.$ It suffices to show that all pairs in $ \{(\sigma_i, \sigma_{m-i+1})\ | \ i \equiv 0,3 \ (\textrm{mod}\ 4)\ \textrm{and}\ 3\leqslant i \leqslant m-5\}$ are hamiltonian pairs. Below, we consider two additional cases.

\noindent \underline{Case 4}: pairs of the form $(\sigma_i, \sigma_{m-i+1})$ where $i \equiv 3 \ (\textrm{mod}\ 4)$ and $3 \leqslant i \leqslant m-5$. 

\noindent SUB-CASE 4.1: $i\in \{3,7\}$. See below:

   \smallskip
{\centering
  $ \scriptstyle
    \begin{aligned} 
&\sigma_{3}\sigma_{m-2}&&=&&\gamma_3\ (m-1, 3) \, \gamma_{-1}\, (m-1,k+1)\\
&&& =&&(0, k+1, k+3, \ldots, m-2, 1, 3, 5, \ldots, k-1, m-1, 2, 4, 6, \ldots, m-3 ); \\
&\sigma_{7}\sigma_{m-6}&&=&&\gamma_7\ (m-1, 5) \, \gamma_{-5}\, (m-1,k-1)\\
&&& =&&(0, 2, 4, \ldots, m-3,  k-1, k+1, k+3, \ldots, m-2, 1, 3, 5, \ldots, k-3, m-1).
      \end{aligned}
  $ 
\par}
   \smallskip

\noindent Since $T( \sigma_3 \sigma_{m-2})=T(\sigma_{7}\sigma_{m-6})=1$, both pairs are hamiltonian pairs.

\noindent SUB-CASE 4.2: $i \equiv 3 \ (\textrm{mod}\ 4)$ and $11 \leqslant i \leqslant m-5$. Let $i=2j+1$ where $5 \leqslant j \leqslant k-2$ and $j$ is odd. Since $i$ is odd, Lemma \ref{lem:comput} implies that $\sigma_{m-i+1}=\gamma_{-i+2}\, (m-1, k-j+2)$. Furthermore, since $i=2j+1$, we have $(m-j)^{\gamma_i}=j+2$ and that $(j+2)^{\gamma_{-i+2}}=m-j+2$. This means that

   \smallskip
{\centering
  $ \scriptstyle
    \begin{aligned} 
&\sigma_{i}\sigma_{m-i+1}&&=&&\gamma_i\ (m-1, j+2) \, \gamma_{-i+2}\, (m-1,k-j+2)\\
&&& =&&(0, 2, 4, \ldots, m-j, k-j+2, k-j+4, k-j+6, \ldots, m-2, 1, 3,5, \ldots, \\
&&&&&  k-j, m-1, m-j+2,m-j+4, \ldots, m-3). 
      \end{aligned}
  $ 
\par}
   \smallskip

\noindent Since $T( \sigma_i \sigma_{m-i+1})=1$, the pair $(\sigma_i, \sigma_{m-i+1})$ is a hamiltonian pair.

\noindent \underline{Case 5}: pairs of the form $(\sigma_i, \sigma_{m-i+1})$ where $i \equiv 0 \ (\textrm{mod}\ 4)$ and $4 \leqslant i \leqslant m-5$. If $m=5$, we have no pairs of the form $(\sigma_i, \sigma_{m-i+1})$. We now consider two sub-cases below.

\noindent SUB-CASE 5.1: $i=4$. Since $m \geqslant 9$, we have $k \geqslant 4$. Therefore

   \smallskip
{\centering
  $ \scriptstyle
    \begin{aligned} 
&\sigma_{4}\sigma_{m-3}&&=\gamma_4\ (m-1, k+3) \, \gamma_{-2}\, (m-1, 0)\\
&&& =(0, 2, 4, \ldots, m-3, m-1, k+1, k+3, \ldots, m-2, 1, 3,5, \ldots , k-1). \\
      \end{aligned}
  $ 
\par}
   \smallskip

\noindent Since $T(\sigma_{4}\sigma_{m-3})=1$, the pair $(\sigma_{4}, \sigma_{m-3})$ is a hamiltonian pair.

\noindent SUB-CASE 5.2: $i\neq 4$.  Observe that $i=2j$ where $j$ is even and $4 \leqslant j \leqslant k-2$. Note that, since $i$ is even, Lemma \ref{lem:comput} implies that $\sigma_{m-i+1}=\gamma_{-i+2}\, (m-1, m-j+1)$.  Moreover, since $k$ and $j$ are even, we see that $k+j+1$ is odd and $m-j-1$ is even. These observations jointly imply that

   \smallskip
{\centering
  $ \scriptstyle
    \begin{aligned}
&\sigma_{i}\sigma_{m-i+1}&&=&&\gamma_i\ (m-1, k+j+1) \, \gamma_{-i+2}\, (m-1, m-j+1)\\
&&&=&&(0, 2, 4, \ldots, m-j-1, m-1, k-j+3, k-j+5, \ldots, m-2, 1, 3, 5, \ldots, \\
&&&&& k-j+1, m-j+1, m-j+3, \ldots, m-3).\\
      \end{aligned}
  $ 
\par}
   \smallskip

\noindent Since $T( \sigma_i \sigma_{m-i+1})=1$, the pair $(\sigma_i, \sigma_{m-i+1})$ is a hamiltonian pair.

In summary, all pairs in $D_H$ are hamiltonian pairs. Since $D_T$ and $D_H$ are disjoint, it follows that $\vec{C}_2 \wr \overline{K}_m$ admits the desired $c$-twined 2-factorization.  \end{proof}

In Proposition \ref{prop:m7even} below, we consider the case  $m \equiv 7$ or $11\  (\textrm{mod} \ 12)$. Henceforth, we will refer the reader to \ref{App} for all computations. 

\begin{proposition}
\label{prop:m7even}
Let $m \equiv 7$ or $11\  (\textrm{mod} \ 12)$ such that $m \geqslant 7$, and let $c$ be an even integer such that $0 \leq c \leq m-3$. The digraph $\vec{C}_2 \wr \overline{K}_m$ admits a $c$-twined 2-factorization.  
\end{proposition}

\begin{proof} We will construct the desired set of pairs by using the following two regular permutation sets:

   \smallskip
{\centering
  $ \displaystyle
    \begin{aligned} 
&R_1&=&&\gamma_1 \cdot \mathcal{F}_m&&=&&\{\mu_i=\gamma_1 \sigma_i\ | \ i=0,1, \ldots, m-1\};\\
&R_2&=&&\gamma_{-1} \cdot \mathcal{F}_m&&=&&\{\tau_i=\gamma_{-1} \sigma_i \ | \ i=0,1, \ldots, m-1\}. 
      \end{aligned}
  $ 
\par}
   \smallskip 

Let $c=2t$ where $1 \leqslant t \leqslant \frac{m-3}{2}$, $\mathds{I}=\{i \ | \ i \equiv 1,2\ (\textrm{mod}\ 4)\ \textrm{and}\ 1 \leqslant i \leqslant m-3 \}$ and $M_t$ be a subset of $\mathds{I}$ of size $t$.  Below, we construct two sets of pairs of permutations from $R_1\times R_2$, $D_{T}$ and $D_H$, that contain $c$ pairs and $m-c$ pairs, respectively:

   \smallskip
{\centering
  $ \displaystyle
    \begin{aligned} 
&D_{T}=&&\{(\mu_i, \tau_{m-i-2}), (\mu_{i+3}, \tau_{m-i-3})\ | \ i\equiv 1 \ (\textrm{mod}\ 4) \ \textrm{and}\ i \in M_t\}\cup\\
&&& \{ (\mu_i, \tau_{m-i}), (\mu_{i+1}, \tau_{m-i-3})\ | \ i\equiv 2 \ (\textrm{mod}\ 4)\ \textrm{and}\ i \in M_t\};\\
    &D_{H}=&& \{(\mu_0, \tau_1), (\mu_{m-2}, \tau_{0}), (\mu_{m-1}, \tau_{m-1})\}\cup\\
    &&&\{(\mu_i, \tau_{m-i-3}), (\mu_{i+3}, \tau_{m-i-2} ) \ | \ i \equiv 1\ (\textrm{mod}\ 4) \ \textrm{and}\  i \in  \mathds{I}/M_t \}\cup \\
    &&& \{(\mu_i, \tau_{m-i-3}), (\mu_{i+1}, \tau_{m-i} ) \ | \ i \equiv 2\ (\textrm{mod}\ 4) \ \textrm{and}\   i \in \mathds{I}/M_t  \}. 
      \end{aligned}
  $ 
\par}
   \smallskip 

See \ref{App} for proof that all pairs of $D_T$ and $D_H$ are truncated hamiltonian and hamiltonian, respectively. The set $D_T\cup D_H$ is a $c$-twined  2-factorization of $\vec{C}_2\wr \overline{K}_m$. \end{proof}

Next, we consider the case $m \equiv 3\ (\textrm{mod} \ 12)$. Although we use the same regular permutation sets as in the proof of Proposition \ref{prop:m7even}, the constructions given in the proof of Proposition \ref{prop:m3even} differ from those of Proposition  \ref{prop:m7even} because $m$ is not relatively prime with 3. Instead, we will use the fact that $m-1$ is relatively prime with 3.

\begin{proposition}
\label{prop:m3even}
Let $m \equiv 3\ (\textrm{mod} \ 12)$ such that $m \geqslant 15$, and let $c$ be an even integer such that $2 \leqslant c \leqslant m-3$. The digraph $\vec{C}_2 \wr \overline{K}_m$ admits a $c$-twined 2-factorization. 
\end{proposition}

\begin{proof} We use the following two regular permutation sets:

   \smallskip
{\centering
  $ \displaystyle
    \begin{aligned} 
&R_1= \gamma_1\,\cdot \mathcal{F}_m&=&&\{ \mu_i=\gamma_1 \sigma_i \ | \ i=0,1, \ldots, m-1\};\\
&R_2=\gamma_{-1}\, \cdot \mathcal{F}_m&=&&\{\tau_i=\gamma_{-1} \sigma_i\ | \ i=0,1, \ldots, m-1\}. 
      \end{aligned}
  $ 
\par}
   \smallskip 

Let $c=2t+2$ where $0\leqslant t \leqslant \frac{m-5}{2}$.  Let $\mathds{I}=\{i \ |\ i\  \textrm{is odd}\ \textrm{and} \ 3 \leqslant i \leqslant m-4\}$ and let $M_t$ be a subset of size $t$ of $\mathds{I}$. Below, we construct two subsets of pairs of permutations from $R_1\times R_2$, $D_{T}$ and $D_H$, that contain $c$ pairs and $m-c$ pairs, respectively:

   \smallskip
{\centering
  $ \displaystyle
    \begin{aligned} 
&D_T=\{(\mu_1, \tau_2), (\mu_2, \tau_1)\}\cup\{(\mu_i, \tau_{m-i}), (\mu_{i+1}, \tau_{m-i+1})\ | \ i\in M_t\}; \\
&D_H= \{(\mu_0, \tau_{m-1}), (\mu_{m-1}, \tau_0)\}\cup\{(\mu_i, \tau_{m-i+1})\ |\ 3\leqslant i \leqslant m-2\ \textrm{and}\ i, i-1 \not\in M_t\}.
      \end{aligned}
  $ 
\par}
   \smallskip 

It is straightforward to verify that $D=D_T\cup D_H$ is indeed a 2-factorization of $\vec{C}_2\wr \overline{K}_m$. See \ref{App} for proof that all pairs of $D_T$ and $D_H$ are truncated hamiltonian and hamiltonian, respectively. \end{proof}

\subsection{The sub-case $c$ is odd}

In this subsection, we construct a $c$-twined 2-factorization of $\vec{C}_2\wr \overline{K}_m$ for all odd $m \geqslant 5$ and all odd $c$ such that $3\leqslant c \leqslant m-2$.   We consider five cases. In the first two cases, addressed in Propositions \ref{prop:1mod4<} and \ref{prop:3mod4} below, we create two regular permutation sets that allow us to form up to $\frac{m+1}{2}$ truncated pairs when $m \equiv 1\ (\textrm{mod}\ 4)$, $\frac{m-5}{2}$ truncated hamiltonian pairs when $m \equiv 3\ (\textrm{mod}\ 8)$, and up to $\frac{m-1}{2}$ truncated hamiltonian pairs when $m \equiv 7\  (\textrm{mod}\ 8)$.  Proposition \ref{prop:1mod4<}  addresses the two cases where $m \equiv 1\ (\textrm{mod}\ 4)$ and Proposition \ref{prop:3mod4} addresses the case $m \equiv 3\ (\textrm{mod}\ 4)$.

\begin{proposition}
\label{prop:1mod4<}
Let $m \equiv 1\ (\textrm{mod} \ 4)$ where $m\geqslant 5$, and let $c$ be an odd integer such that $3\leqslant c \leqslant \frac{m+1}{2}$. The digraph $\vec{C}_2 \wr \overline{K}_m$ admits a $c$-twined 2-factorization. 
\end{proposition}

\begin{proof} To obtain the desired $c$-twined 2-factorization of $\vec{C}_2 \wr \overline{K}_m$, we use the following two regular sets of permutations of order $m$:

   \smallskip
{\centering
  $ \displaystyle
    \begin{aligned} 
&R_1=(1,m-1)\cdot \mathcal{F}_m=\{\mu_i=(1, m-1)\sigma_i \ | \ i=0,1, \ldots, m-1\}; && R_2=\mathcal{F}_m=\{\sigma_0, \sigma_1, \ldots, \sigma_{m-1}\}. 
      \end{aligned}
  $ 
\par}
   \smallskip 

Let $c=3+2t$ where $0\leqslant t \leqslant \frac{m-5}{4}$, $\mathds{I}=\{i \ |\ i \equiv 3\ (\textrm{mod}\ 4)\ \textrm{and} \ 3 \leqslant i \leqslant m-6\}$, and $M_t$ be a subset of size $t$ of $\mathds{I}$. Below, we construct two sets of pairs of permutations from $R_1\times R_2$, $D_{T}$ and $D_H$, that contain $c$ pairs and $m-c$ pairs, respectively:

   \smallskip
{\centering
  $ \displaystyle
    \begin{aligned} 
&D_T&&=&& \{(\mu_{0}, \sigma_{1}), (\mu_{m-2}, \sigma_{m-2}), (\mu_{m-1}, \sigma_{m-1})\}\cup\{(\mu_{i}, \sigma_{m-i-3}), (\mu_{i+3}, \sigma_{m-i-2})\ |\ i \in M_t\};\\
&D_H&&=&& \{(\mu_2, \sigma_0)\}\cup \{(\mu_i, \sigma_{m-i-2})\ | \ i\ \textrm{odd}, \ 1 \leqslant i \leqslant m-4, \textrm{and}\ i\not\in M_t\}\cup\\
&&&&&\{(\mu_i, \sigma_{m-i})\ | \ i\ \textrm{even},\ 4 \leqslant i \leqslant m-3,  \textrm{and}\ i, i-3\not\in M_t\}.
      \end{aligned}
  $ 
\par}
   \smallskip 
   
It is straightforward to verify that $D=D_T\cup D_H$ is indeed a 2-factorization of $\vec{C}_2\wr \overline{K}_m$. See \ref{App} for proof that all pairs of $D_T$ and $D_H$ are truncated hamiltonian and hamiltonian, respectively. \end{proof}
\begin{proposition}
\label{prop:3mod4}
Let $m \geqslant 7$ and $c$ be odd. The digraph $\vec{C}_2 \wr \overline{K}_m$ admits a $c$-twined 2-factorization in each of the following cases:
\begin{enumerate} [label=\textbf{(S\arabic*)}]
\item $m\equiv 3 \ (\textrm{mod}\ 8)$ and $3\leqslant c \leqslant \frac{m-5}{2}$;
\item $m\equiv 7 \ (\textrm{mod}\ 8)$ and $3\leqslant c \leqslant \frac{m-1}{2}$.
\end{enumerate}
\end{proposition}

\begin{proof} We use the following two regular permutation sets of order $m$:

     \smallskip
{\centering
  $ \displaystyle
    \begin{aligned} 
&R_1=(0,m-1)\cdot \mathcal{F}_m=\{ \mu_i=(0, m-1)\sigma_i\ |\ i=0,1,\cdots, m-1 \}; && R_2=\mathcal{F}_m=\{\sigma_0, \sigma_1, \ldots, \sigma_{m-1}\}. 
      \end{aligned}
  $ 
\par}
   \smallskip 

\noindent \underline{Case 1}:  $m\equiv 3 \ (\textrm{mod}\ 8)$ and $3\leqslant c \leqslant \frac{m-5}{2}$. Let $c=3+2t$ where $0 \leqslant t \leqslant \frac{m-11}{4}$. Let $\mathds{I}=\{i \ |\ i \equiv 5, 6\ (\textrm{mod}\ 8)\ \textrm{and} \ 5\leqslant i\leqslant m-13\}$ and let $M_t$ be a subset of size $t$ of $\mathds{I}$. Note that $\mathds{I}$ contains $\frac{m-5}{4}$ elements. Below, we construct two sets of pairs of permutations from $R_1\times R_2$, $D_{T}$ and $D_H$, that contain $c$ pairs and $m-c$ pairs, respectively:

   \smallskip
{\centering
  $ \displaystyle
    \begin{aligned} 
    &D_T&&=&& \{(\mu_{m-2}, \sigma_{m-2}), (\mu_{m-1}, \sigma_{m-1}), (\mu_{0}, \sigma_1)\}\cup \\
    &&&&&\{(\mu_i, \sigma_{m-i-5}), (\mu_{i+5}, \sigma_{m-i-2})\ |\ i\in M_t\ \textrm{and}\ i \equiv 5\ (\textrm{mod}\ 8)\}\cup \\ 
   &&&&& \{(\mu_i, \sigma_{m-i-5}), (\mu_{i+3}, \sigma_{m-i})\ |\ i\in M_t \ \textrm{and} \ i \equiv 6\ (\textrm{mod}\ 8)\};\\
&D_H&&=&&\{(\mu_2, \sigma_0)\}\cup \{(\mu_i, \sigma_{m-i-2})\ | \ i, i-3 \not \in M_t,  i\ \textrm{odd}, \ \textrm{and}\ 1 \leqslant i \leqslant m-4\} \cup \\
&&&&&\{(\mu_i, \sigma_{m-i}) \ |\ i, i-5 \not\in M_t,  i\ \textrm{even}, \ \textrm{and}\ 4 \leqslant i \leqslant m-3\}.
      \end{aligned}
  $ 
\par}
   \smallskip

It is straightforward to verify that $D=D_T\cup D_H$ is indeed a 2-factorization of $\vec{C}_2\wr \overline{K}_m$. See \ref{App} for proof that all pairs of $D_T$ and $D_H$ are truncated hamiltonian and hamiltonian, respectively. 

\noindent \underline{Case 2}: $m \equiv 7\ (\textrm{mod}\ 8)$ and $3 \leqslant c \leqslant \frac{m-1}{2}$. Let $c=2t+3$ where $0 \leqslant t \leqslant \frac{m-7}{4}$. In this case, let $\mathds{I}=\{i \ |\ i \equiv 5, 6\ (\textrm{mod}\ 8)\ \textrm{and}\ 5 \leqslant i \leqslant m-7\}$. This time, the set $\mathds{I}$ contains $\frac{m-1}{4}$ elements. Let $M_t$ be a subset of size $t$ of $\mathds{I}$ and construct $D_T$ and $D_H$ as in Case 1. \end{proof}

Next, we construct a $c$-twined 2-factorization of $\vec{C}_2 \wr \overline{K}_m$ when $\frac{m-5}{2}\leqslant c \leqslant m-2$. In fact, the constructions given in Propositions \ref{prop:1mod6}-\ref{prop:5mod6} apply to $\lceil \frac{m}{3}\rceil+5 \leqslant c \leqslant m-2$. Therefore, there exists some overlap with the constructions given in Propositions \ref{prop:1mod4<} and \ref{prop:3mod4}. 

To address the case $\frac{m-5}{2}\leqslant c \leqslant m-2$, we consider three cases according to the congruency class modulo 6 of $m$. The most complicated construction is given for  $m \equiv 1\ (\textrm{mod}\ 6)$ because, in this case, $m-1$ is not relatively prime with 3.  As for the cases  $m \equiv 3$ or $5\ (\textrm{mod}\ 6)$, we rely on the fact that  $m-1$ is relative prime with 3 to obtain the desired $c$-twined 2-factorizations of $\vec{C}_2 \wr \overline{K}_m$.

\begin{proposition}
\label{prop:1mod6}
Let $m \equiv 1\ (\textrm{mod} \ 6)$ such that $m \geqslant 7$, and let $c$ be an odd integer such that $\frac{m+2}{3} \leq c \leq m-2$. The digraph $\vec{C}_2 \wr \overline{K}_m$ admits a $c$-twined 2-factorization.  
\end{proposition}

\begin{proof} We use the following two regular sets of permutations:

\smallskip
{\centering
  $ \scriptstyle
    \begin{aligned}
    &R_1=(m-1,1,2,3) \cdot \mathcal{F}_m=\{\mu_i=(m-1, 1,2,3) \sigma_i\ | \ i=0,1, \ldots, m-1 \}; \\
    &R_2= \mathcal{F}_m=\{\sigma_0, \sigma_1, \ldots, \sigma_{m-1}\}.\\
          \end{aligned}
  $ 
\par}
\smallskip

Let $c=\frac{m+2}{3}+2t$ where $0 \leqslant 2t \leqslant \frac{2m-8}{3}$, $\mathds{I}=\{i \ |\ i \equiv 0, 4\ (\textrm{mod}\ 6)\ \textrm{and}\ 4 \leqslant i \leqslant m-3\}$, and $M_t$ be a subset of size $t$ of $\mathds{I}$.  Below, we construct two sets of pairs of permutations from $R_1\times R_2$, $D_{T}$ and $D_H$, that contain $c$ pairs and $m-c$ pairs, respectively:

 \smallskip
{\centering
  $ \scriptstyle
    \begin{aligned} 
   &D_{T}=&& \{ (\mu_{m-1}, \sigma_{m-1}), (\mu_0, \sigma_1)\} \cup \{(\mu_i, \sigma_{m-i+1}), (\mu_{i+1}, \sigma_{m-i})\ | \ i \in M_t\}\cup\\
   &&&\{(\mu_i, \sigma_{m-i+1}) \ |\ i\equiv 2, 3\ (\textrm{mod}\ 6)\ \textrm{and}\ 3 \leqslant i \leqslant m-4\};\\
&D_H=&&\{(\mu_1, \sigma_2), (\mu_2, \sigma_0)\} \cup \{(\mu_i, \sigma_{m-i})\ |\  i \not\in M_t, \  i\equiv 0, 4\ (\textrm{mod}\ 6),\textrm{and}\  4 \leqslant i \leqslant m-3\}\cup\\
&&&\{(\mu_i, \sigma_{m-i+2}) \ |\  i-1 \not\in M_t,  i\equiv 1, 5\ (\textrm{mod}\ 6), \ \textrm{and}\ 5 \leqslant i \leqslant m-2 \}.
      \end{aligned}
  $ 
\par}
   \smallskip

It is tedious but straightforward to verify that $D=D_T\cup D_H$ is indeed a 2-factorization of $\vec{C}_2\wr \overline{K}_m$. See \ref{App} for proof that all pairs of $D_T$ and $D_H$ are truncated hamiltonian and hamiltonian, respectively. \end{proof}

\begin{proposition}
\label{prop:m3cong6}
Let $m \equiv 3\ (\textrm{mod} \ 6)$ such that $m \geqslant 9$,  and let $c$ be an odd integer such that $\frac{m}{3}+4 \leqslant c \leqslant m-2$. The digraph $\vec{C}_2 \wr \overline{K}_m$ admits a $c$-twined 2-factorization.  
\end{proposition}

\begin{proof} We will construct a $c$-twined 2-factorization of  $\vec{C}_2 \wr \overline{K}_m$ by using the following two regular sets of $m$ permutations:

      \smallskip
{\centering
  $ \scriptstyle
    \begin{aligned}
    &R_1=(1,2,3,4) \cdot \mathcal{F}_m=\{\mu_i=(1,2,3,4) \sigma_i\ | \ i=0,1,\cdots, m-1\}; R_2= \mathcal{F}_m=\{\sigma_0, \sigma_1, \ldots, \sigma_{m-1}\}.\\
          \end{aligned}
  $ 
\par}
\smallskip

Let $c=\frac{m}{3}+4+2t$, where $0\leqslant 2t \leqslant \frac{2m}{3}-6$ and let $\mathds{I}=\{i \ |\ i \equiv 3, 5\ (\textrm{mod}\ 6)\ \textrm{and}\ 3\leqslant i \leqslant m-10\}$ and let $M_t$ be a subset of size $t$ of $\mathds{I}$. Note that $\mathds{I}$ contains $\frac{m}{3}-3$ elements. Below, we construct two sets of pairs of permutations from $R_1\times R_2$, $D_{T}$ and $D_H$, that contain $c$ pairs and $m-c$ pairs, respectively:

 \smallskip
{\centering
  $ \scriptstyle
    \begin{aligned} 
   &D_{T}&&=&&\{(\mu_{m-4}, \sigma_{m-4}),  (\mu_{m-2}, \sigma_{m-1}), (\mu_{m-1}, \sigma_{m-2})\}\cup\\
   &&&&&\{(\mu_i, \sigma_{m-i-5}), (\mu_{i+1}, \sigma_{m-i-4})\ | \ i \equiv 2 \ (\textrm{mod} \ 6)\  \textrm{and}\  2\leqslant i \leqslant m-15\}\cup \\
   &&&&& \{(\mu_i, \sigma_{m-i-5}), (\mu_{i+1}, \sigma_{m-i-4})\ | \ i\in \{m-10, m-8, m-6\}\}\cup \\
   &&&&&\{(\mu_i, \sigma_{m-i-5}), (\mu_{i+1}, \sigma_{m-i-4}) \ | \ i \in M_t\};\\
   &D_H&&=&&\{ (\mu_{m-3}, \sigma_0), (\mu_0, \sigma_{m-3})\}\cup \\ 
   &&&&&\{(\mu_i, \sigma_{m-i-4})\ |\ i\equiv 0,3,4,5\ (\textrm{mod}\ 6), 3\leqslant i \leqslant m-9, \textrm{and}\ i, i-1 \not\in M_t\}.
   \end{aligned}
  $ 
\par}
   \smallskip
   
It is tedious but straightforward to verify that $D=D_T\cup D_H$ is indeed a 2-factorization of $\vec{C}_2\wr \overline{K}_m$. See \ref{App} for proof that all pairs of $D_T$ and $D_H$ are truncated hamiltonian and hamiltonian, respectively. \end{proof}

\begin{proposition}
\label{prop:5mod6}
Let $m \equiv 5\ (\textrm{mod}\ 6)$ such that $m \geqslant 11$, and let $c$ be an odd integer such that $\frac{m+1}{3}+5 \leq c \leq m-2$. The digraph $\vec{C}_2 \wr \overline{K}_m$ admits a $c$-twined 2-factorization.  
\end{proposition}

\begin{proof} We will construct a $c$-twined 2-factorization of  $\vec{C}_2 \wr \overline{K}_m$ by using the following two regular sets of $m$ permutations:

      \smallskip
{\centering
  $ \scriptstyle
    \begin{aligned}
    &R_1=(1,2,3,4) \cdot \mathcal{F}_m=\{\mu_i=(1,2,3,4)\sigma_i\ | \ i=0,1, \ldots, m-1\}; &R_2= \mathcal{F}_m=\{\sigma_0, \sigma_1, \ldots, \sigma_{m-1}\}.\\
          \end{aligned}
  $ 
\par}
\smallskip

Let $c={\frac{m+1}{3}+5}+2t$ where $0\leqslant 2t \leqslant \frac{2m-1}{3}-9$ and let $\mathds{I}=\{i \ |\ i \equiv 0, 4\ (\textrm{mod}\ 6)\ \textrm{and}\ 4 \leqslant i \leqslant m-13\}$, and $M_t$ be a subset of size $t$ of $\mathds{I}$. Observe that $\mathds{I}$ contains $\frac{m-2}{3}-4$ elements. Below, we construct two sets of pairs of permutations from $R_1\times R_2$, $D_{T}$ and $D_H$, that contain $c$ pairs and $m-c$ pairs, respectively:

 \smallskip
{\centering
  $ \scriptstyle
    \begin{aligned} 
   &D_{T}&&=&&\{(\mu_{m-4}, \sigma_{m-4}),  (\mu_{m-2}, \sigma_{m-1}), (\mu_{m-1}, \sigma_{m-2})\}\cup\\
   &&&&&\{(\mu_i, \sigma_{m-i-5}), (\mu_{i+1}, \sigma_{m-i-4})\ | \ i \equiv 2 \ (\textrm{mod} \ 6)\  \textrm{and}\  2\leqslant i \leqslant m-15\}\cup \\
   &&&&& \{(\mu_i, \sigma_{m-i-5}), (\mu_{i+1}, \sigma_{m-i-4})\ | \ i\in \{m-10, m-8, m-6\}\}\cup \\
   &&&&&\{(\mu_i, \sigma_{m-i-5}), (\mu_{i+1}, \sigma_{m-i-4}) \ | \ i \in M_t\};\\
   &D_H&&=&&\{(\mu_{m-3}, \sigma_0), (\mu_0, \sigma_{m-3})\}\cup\\
   &&&&&\{(\mu_i, \sigma_{m-i-4})\ | \ i \equiv 0,1, 4,5 \ (\textrm{mod} \ 6), 1\leqslant i \leqslant m-11, \textrm{and}\ i, i-1 \not\in M_t\}.
   \end{aligned}
  $ 
\par}
   \smallskip
   
It is tedious but straightforward to verify that $D=D_T\cup D_H$ is indeed a 2-factorization of $\vec{C}_2\wr \overline{K}_m$. See \ref{App} for proof that all pairs of $D_T$ and $D_H$ are truncated hamiltonian and hamiltonian, respectively. \end{proof}

When $m=11$, neither of Propositions \ref{prop:3mod4} and \ref{prop:5mod6} gives rise to a $c$-twined 2-factorization of $\vec{C}_2 \wr \overline{K}_m$ for $c\in \{5,7\}$. We remedy this omission in Lemma \ref{lem:m11} below. 

\begin{lemma}
\label{lem:m11}
Let $m=11$ and $c\in \{5,7\}$. The digraph $\vec{C}_2 \wr \overline{K}_m$ admits a $c$-twined 2-factorization. 
\end{lemma}

\begin{proof} We consider two cases, one for each $c\in \{5,7\}$. 

\noindent \underline{Case 1}: $c=5$. Let

      \smallskip
{\centering
  $ \scriptstyle
    \begin{aligned}
    &R_1= \mathcal{F}_{11}=\{\sigma_0, \sigma_1, \ldots, \sigma_{10}\}; R_2= (0, 10)\cdot  \mathcal{F}_{11}=\{\mu_i=(0,10)\sigma_i\ | \ i=0,1, \ldots, 10\}.\\
          \end{aligned}
  $ 
\par}
   \smallskip

\noindent We construct the following two sets of pairs of permutations from $R_1\times R_2$:

      \smallskip
{\centering
  $ \scriptstyle
    \begin{aligned}
&D_T&&=&&\{(\sigma_1, \mu_0), (\sigma_5, \mu_1), (\sigma_8, \mu_6), (\sigma_9, \mu_9), (\sigma_{10}, \mu_{10})\};\\
&D_H&&=&&\{(\sigma_2, \mu_7), (\sigma_3, \mu_8), (\sigma_4, \mu_5), (\sigma_6, \mu_3), (\sigma_7, \mu_4), (\sigma_0, \mu_2) \}.
          \end{aligned}
  $ 
\par}
   \smallskip
 \noindent See \ref{App} for proof that all pairs in $D_T$ and $D_H$ are truncated hamiltonian and hamiltonian, respectively.
 
In conclusion, the set $D=D_T\cup D_H$ is the desired $5$-twined 2-factorization of $\vec{C}_{2}\wr \overline{K}_{11}$. 

\noindent \underline{Case 2}: $c=7$. Let

      \smallskip
{\centering
  $ \scriptstyle
    \begin{aligned}
    &R_1=(10,1,2,3,4,5) \cdot \mathcal{F}_{11}=\{ \mu_i=(10,1,2,3,4,5) \sigma_i\ | \ i=0,1, \ldots, 10\}; \\
    &R_2= \mathcal{F}_{11}=\{\sigma_0, \sigma_1, \ldots, \sigma_{10}\}. \\
          \end{aligned}
  $ 
\par}
   \smallskip

\noindent We construct the following two sets of pairs of permutations from $R_1\times R_2$:

      \smallskip
{\centering
  $ \scriptstyle
    \begin{aligned}
&D_T&&=&&\{(\mu_2, \sigma_6), (\mu_3, \sigma_1), (\mu_5, \sigma_9), (\mu_6, \sigma_8), (\mu_7, \sigma_7), (\mu_9, \sigma_5), (\sigma_{10}, \mu_{10})\};\\
&D_H&&=&&\{(\mu_1, \sigma_4), (\mu_4, \sigma_0), (\mu_{8}, \sigma_3), (\mu_0, \sigma_{2})\}.
          \end{aligned}
  $ 
\par}
   \smallskip
 \noindent See \ref{App} for proof that all pairs in $D_T$ and $D_H$ are truncated hamiltonian and hamiltonian, respectively.
 
In conclusion, the set $D=D_T\cup D_H$ is the desired $7$-twined 2-factorization of $\vec{C}_{2}\wr \overline{K}_{11}$. \end{proof}

\section{Acknowledgements}
The author would like to thank her Ph.D.~supervisor Mateja \v{S}ajna for her support and for verifying all the details of this paper. The author was supported by the Natural Sciences and Engineering Research Council of Canada (NSERC) Post Graduate Scholarship program and the Pacific Institute for the Mathematical Sciences (PIMS).

\pagebreak

\appendix
\section{Example of a regular permutation set of odd order}
\label{App2}

Below we give an example of Construction \ref{cons:Fodd}.  

\noindent {\bf Example A.1.} In this example, we construct $\mathcal{F}_{15}$. First, we list the elements of $G_{14}$:

   \medskip
{\centering
  $ \displaystyle
    \begin{aligned} 
&\gamma_0=id;\\   
&\gamma_1=(0, 1,2,3,4,5,6,7,8,9,10,11,12,13)(14);\\
&\gamma_2=(0, 2, 4, 6, 8, 10, 12)(1, 3, 5, 7, 9, 11, 13)(14);\\
&\gamma_3=(0, 3, 6, 9, 12, 1, 4, 7, 10, 13, 2, 5, 8, 11)(14);\\
&\gamma_4=(0, 4, 8, 12, 2, 6, 10)(1, 5, 9, 13, 3, 7, 11)(14);\\
&\gamma_5=(0, 5, 10, 1, 6, 11, 2, 7, 12, 3, 8, 13, 4, 9)(14);\\
&\gamma_6=(0, 6, 12, 4, 10, 2, 8)(1, 7, 13, 5, 11, 3, 9)(14);\\
&\gamma_7=(0,7)(1, 8)(2, 9)(3, 10)(4, 11)(5, 12)(6, 13)(14);\\
&\gamma_8=( 0, 8, 2, 10, 4, 12, 6)(1, 9, 3, 11, 5, 13, 7)(14);\\
&\gamma_9=(0, 9, 4, 13, 8, 3, 12, 7, 2, 11, 61, 10, 5)(14);\\
&\gamma_{10}=(0, 10, 6, 2, 12, 8, 4)(1, 11, 7, 3, 13, 9, 5)(14);\\
&\gamma_{11}=(0, 11, 8, 5, 2, 13, 10, 7, 4, 1, 12, 9, 6, 3 )(14);\\
&\gamma_{12}=(0, 12, 10, 8, 6, 4, 2)(1, 13, 11, 9, 7, 5, 3)(14);\\
&\gamma_{13}=(0, 13, 12, 11, 10, 9, 8, 7, 6, 5, 4, 3, 2, 1)(14).\\
      \end{aligned}
  $ 
\par}
\medskip 

Note that $m=2k+1$ where $k=7$. Below, we list each of the 15 permutations in $\mathcal{F}_{15}$:

   \medskip
{\centering
  $ \displaystyle
    \begin{aligned} 
&\sigma_0=id;\\
&\sigma_1=\gamma_1(14, 2)=(14, 2, 3, 4, 5, 6, 7, 8, 9, 10, 11, 12, 13, 0, 1);\\
&\sigma_2=\gamma_2(14, 9)=(14, 9, 11, 13, 1, 3, 5, 7)(2, 4, 6, 8, 10, 12, 0);\\
&\sigma_3=\gamma_3(14, 3)=(14, 3, 6, 9, 12, 1, 4, 7, 10, 13, 2, 5, 8, 11, 0);\\
&\sigma_4=\gamma_4(14, 10)=(14, 10, 0, 4, 8, 12, 2, 6)(1, 5, 9, 13, 3, 7, 11);\\
&\sigma_5=\gamma_5(14, 4)=(14, 4, 9, 0, 5, 10, 1, 6, 11, 2, 7, 12, 3, 8, 13);\\
&\sigma_6=\gamma_6(14, 11)=(14, 11, 3, 9, 1, 7, 13, 5)(2, 8, 0, 6, 12, 4, 10);\\
&\sigma_7=\gamma_7(14, 5)=(14, 5, 12)(1, 8)(2, 9)(3, 10)(4, 11)(6, 13)(0, 7);\\
&\sigma_8=\gamma_8(14, 12)=(14, 12, 6, 0, 8, 2, 10, 4)(1, 9, 3, 11, 5, 13, 7);\\
&\sigma_9=\gamma_{9}(14, 6)=(14, 6, 1, 10, 5, 0, 9, 4, 13, 8, 3, 12, 7, 2, 11);\\
&\sigma_{10}=\gamma_{10}(14, 13)=(14, 13, 9, 5, 1, 11, 7, 3)(2, 12, 8, 4, 0, 10, 6);\\
    &\sigma_{11}=\gamma_{11}(14, 7)=(14, 7, 4, 1, 12, 9, 6, 3, 0, 11, 8, 5, 2, 13, 10);\\
&\sigma_{12}=\gamma_{12}(14, 0)=(14, 0, 12, 10, 8, 6, 4, 2)(1, 13, 11, 9, 7, 5, 3);\\
&\sigma_{13}=\gamma_{13}(14, 8)=(14, 8, 7, 6, 5, 4, 3, 2, 1, 0, 13, 12, 11, 10, 9);\\
&\sigma_{14}=(0, 3, 13, 4, 12, 5, 11, 6, 10, 7, 9, 8, 14, 1, 2).\\      
      \end{aligned}
  $ 
\par}

\hfill $\square$

\section{Additional computations}
\label{App}

\noindent{\bf Computations for the proof of Proposition \ref{prop:m7even}}. Since $m \equiv 7, 11 \ (\textrm{mod}\ 12)$, then $m=2k+1$ where $k$ is odd. In these computations, it is understood that, for each $i \in \mathds{Z}_{m-1}$, $\gamma_i \in G_{m-1}$, where $G_{m-1}$  is defined in Definition \ref{def:gammas}. We use the following two regular permutation sets of order $m$:

   \smallskip
{\centering
  $ \displaystyle
    \begin{aligned} 
&R_1&=&&\gamma_1 \cdot \mathcal{F}_m&&=&&\{\mu_i=\gamma_1 \sigma_i\ | \ i=0,1, \ldots, m-1\};\\
&R_2&=&&\gamma_{-1} \cdot \mathcal{F}_m&&=&&\{\tau_i=\gamma_{-1} \sigma_i \ | \ i=0,1, \ldots, m-1\}. 
      \end{aligned}
  $ 
\par}
   \smallskip

Let $c=2t$ where $1 \leqslant t \leqslant \frac{m-3}{2}$, $\mathds{I}=\{i \ | \ i \equiv 1,2\ (\textrm{mod}\ 4)\ \textrm{and}\ 1 \leqslant i \leqslant m-3 \}$ and $M_t$ be a subset of $\mathds{I}$ of size $t$. Recall that

   \smallskip
{\centering
  $ \displaystyle
    \begin{aligned} 
&D_{T}=&&\{(\mu_i, \tau_{m-i-2}), (\mu_{i+3}, \tau_{m-i-3})\ | \ i\equiv 1 \ (\textrm{mod}\ 4) \ \textrm{and}\ i \in M_t\}\cup\\
&&& \{ (\mu_i, \tau_{m-i}), (\mu_{i+1}, \tau_{m-i-3})\ | \ i\equiv 2 \ (\textrm{mod}\ 4)\ \textrm{and}\ i \in M_t\};\\
    &D_{H}=&& \{(\mu_0, \tau_1), (\mu_{m-2}, \tau_{0}), (\mu_{m-1}, \tau_{m-1})\}\cup\\
    &&&\{(\mu_i, \tau_{m-i-3}), (\mu_{i+3}, \tau_{m-i-2} ) \ | \ i \equiv 1\ (\textrm{mod}\ 4) \ \textrm{and}\  i \in  \mathds{I}/M_t \}\cup \\
    &&& \{(\mu_i, \tau_{m-i-3}), (\mu_{i+1}, \tau_{m-i} ) \ | \ i \equiv 2\ (\textrm{mod}\ 4) \ \textrm{and}\   i \in \mathds{I}/M_t  \}. 
      \end{aligned}
  $ 
\par}
   \smallskip

\noindent STEP 1. We show that all pairs in $D_T$ are truncated hamiltonian pairs; we do so by considering two cases. In both cases, we use the fact that $T(\gamma_1)=T(\gamma_{-1})=2$ where $\gamma_1$ is defined in Definition \ref{def:gammas}. 

\noindent \underline{Case 1}: Pairs of the form $(\mu_i, \tau_{m-i-2}), (\mu_{i+3}, \tau_{m-i-3})$ where $\ i\equiv 1 \ (\textrm{mod}\ 4) \ \textrm{and}\ i \in M_t$. Note that $(m-1)^{\mu_i}=(m-1)^{\tau_i}=(m-1)^{\sigma_i}$ for all $i \in \{0,1,\cdots, m-1\}$. As a result, if $i \not\in\{ 0, m-1\}$, we see that $\hat{\mu}_i=\gamma_1\, \gamma_i=\gamma_{i+1}$ and $\hat{\tau}_i=\gamma_{-1}\gamma_i=\gamma_{i-1}$. This means that

   \smallskip
{\centering
  $ \scriptstyle
    \begin{aligned} 
&\hat{\mu}_{i}\hat{\tau}_{m-i-2}&&=\gamma_{i+1}\gamma_{m-i-3}\,&&=\gamma_{i+1} \gamma_{-i-2}&&=\gamma_{-1};\\
&\hat{\mu}_{i+3}\hat{\tau}_{m-i-3}&&=\gamma_{i+4}\gamma_{m-i-4}&&=\gamma_{i+4}\, \gamma_{-i-3}&&=\gamma_{1}. 
      \end{aligned}
  $ 
\par}
\smallskip

\noindent Hence, all pairs of the given form are truncated hamiltonian pairs. 

\noindent \underline{Case 2}: Pairs of the form $(\mu_i, \tau_{m-i}), (\mu_{i+1}, \tau_{m-i-3})$ where $\ i\equiv 2 \ (\textrm{mod}\ 4) \ \textrm{and}\ i \in M_t$. Below, we show that each pair of this form is a truncated hamiltonian pair:

   \smallskip
{\centering
  $ \scriptstyle
    \begin{aligned} 
&\hat{\mu}_{i}\hat{\tau}_{m-i}&&=\gamma_{i+1}\, \gamma_{m-i-1}&&=\gamma_{i+1}\, \gamma_{-i}&&=\gamma_{1};\\
&\hat{\mu}_{i+1}\hat{\tau}_{m-i-3}&&=\gamma_{i+2}\, \gamma_{m-i-4}&&=\gamma_{i+2}\, \gamma_{-i-3}&&=\gamma_{-1}. 
      \end{aligned}
  $ 
\par}
\smallskip

\noindent STEP 2. We show that all pairs in $D_H$ are  hamiltonian pairs; we do so by considering four cases.\\
\noindent \underline{Case 1}: The pair $(\mu_{m-1}, \tau_{m-1})$. We consider two sub-cases because the resulting product depends on the congruency class of $m$ modulo 12. 

\noindent SUB-CASE 1.1: $m \equiv 7 \ (\textrm{mod} \ 12)$. Then $k \equiv 0 \ (\textrm{mod}\ 3)$ and thus

   \smallskip
{\centering
  $ \displaystyle
    \begin{aligned} 
&\mu_{m-1} \, \tau_{m-1}&&=&&\gamma_1\,(0, 3, m-2, 4, m-3, 5, m-4, \ldots, k+3, k, k+2, k+1,  m-1,1,2) \gamma_{-1}\, \\
&&&&& (0, 3, m-2, 4,  m-3, 5, m-4,  \cdots, k+3, k, k+2, k+1, m-1,1,2)\\
&&&=&&(0, 2, 5, 8, 11, \ldots, k-1, m-1,3,6,9, \ldots, k, 1, 4, 7, 10, \ldots, k+1, k+2,\\
&&&&& k+3, \ldots, m-2 ).\\
      \end{aligned}
  $ 
\par}
\smallskip 

\noindent SUB-CASE 1.1: $m \equiv 11 \ (\textrm{mod} \ 12)$. Then $k \equiv 2 \ (\textrm{mod}\ 3)$ and thus

   \smallskip
{\centering
  $ \displaystyle
    \begin{aligned} 
&\mu_{m-1} \, \tau_{m-1}&&=&&(0, 2, 5, 8, \ldots, k, 1, 4, 7, \ldots, k-1, m-1, 3, 6, 9, \ldots, k+1, k+2, k+3, \ldots, m-2). 
      \end{aligned}
  $ 
\par}
\smallskip 

\noindent \underline{Case 2}: The pairs $(\mu_0, \tau_1), (\mu_{m-2}, \tau_{0})$.   Below, we show that $(\mu_0, \tau_1)$ and $(\mu_{m-2}, \tau_0)$ are hamiltonian pairs:

   \smallskip
{\centering
  $ \displaystyle
    \begin{aligned} 
&\mu_0 \, \tau_1&&=&&\gamma_1\,\gamma_{-1}\,\gamma_1\,(m-1,2)=\gamma_1\,(m-1, 2)=(0,1, m-1, 2,3,4,\ldots, m-3, m-2 );\\
&\mu_{m-2} \, \tau_{0}&&=&&\gamma_1\,\gamma_{-1}\, (m-1, k+1) \, \gamma_{-1}=(m-1, k+1)\, \gamma_{-1}\\
&&&=&&(0, m-2, m-3, m-4, \ldots, k+1, m-1, k, k-1, k-2, k-3, \ldots, 1). 
      \end{aligned}
  $ 
\par}
\smallskip 

\noindent \underline{Case 3}: Pairs of the form $(\mu_i, \tau_{m-i-3}),$ where $ i \equiv 1, 2\ (\textrm{mod}\ 4) \ \textrm{and}\  i \in \mathds{I}/M_t$. 

\noindent SUB-CASE 3.1:  $i \equiv 1 \ (\textrm{mod}\ 4)$ and $i \neq m-2$. Then $i=2j+1$, where $j$ is even and $0 \leqslant j \leqslant k-2$. Since $i$ is odd and $i\neq m-2$, Lemma \ref{lem:comput} implies that
 
 \begin{center}
 $\tau_{m-i-3}=\gamma_{-1}\, \sigma_{m-i-3}= \gamma_{-1}\ \gamma_{-i-2}\, (m-1, k-j)= \gamma_{-i-3}\, (m-1, k-j)$. 
 \end{center}

\noindent Note that, since $i=2j+1$, Lemma \ref{lem:compgam} implies that $(m-j-1)^{\gamma_{i+1}}=j+2$ and that $(j+2)^{\gamma_{-i-3}}=m-j-3$. In addition, we point out that $k-j$ is odd and $m-j-1$ is even. We then see that
 
    \smallskip
{\centering
  $ \displaystyle
    \begin{aligned} 
&\mu_{i} \, \tau_{m-i-3}&&=&&\gamma_1\, \gamma_i\, (m-1, j+2) \,\gamma_{-i-3} \, (m-1, k-j)\\
&&&=&&(0, m-3, m-5, \ldots, m-j-1, k-j, k-j-2, \ldots, 1, m-2, m-4, \\
&&&&&  \cdots, k-j+2, m-1, m-j-3, m-j-5, m-j-7, \ldots, 4, 2). 
      \end{aligned}
  $ 
\par}
\smallskip 

\noindent SUB-CASE 3.2:  $i \equiv 2 \ (\textrm{mod}\ 4)$. Then $i=2j$, where $j$ is odd and $1 \leqslant j \leqslant k-2$. Since $i$ is even, Lemma \ref{lem:comput} implies that
 
 \begin{center}
 $\tau_{m-i-3}=\gamma_{-1}\, \sigma_{m-i-3}= \gamma_{-1}\ \gamma_{-i-2}\, (m-1, m-j-1)= \gamma_{-i-3}\, (m-1, m-j-1)$. 
 \end{center}

\noindent Note that, since $i=2j$, Lemma \ref{lem:compgam} implies that $(k-j)^{\gamma_{i+1}}=k+j+1$. In addition, we point out that $k-j$ is even and   $m-j-1$ is odd. We then see that

   \smallskip
{\centering
  $ \displaystyle
    \begin{aligned} 
&\mu_{i} \, \tau_{m-i-3}&&=&&\gamma_1\gamma_{i}\, (m-1, k+j+1) \, \gamma_{-i-3} \, (m-1, m-j-1)\\
&&&=&&(0, m-3, m-5,\cdots, k-j, m-j-1,m-j-3, \ldots,  1, m-2, m-4, \\
&&&&&  \ldots, m-j+1, m-1, k-j-2, k-j-4, \ldots, 4, 2 ).
      \end{aligned}
  $ 
\par}
\smallskip 

We must now verify that pairs in the set below are hamiltonian pairs:

\begin{center}
$A=\{(\sigma_{i+3}, \sigma_{m-i-2} ) \ | \ i \equiv 1\ (\textrm{mod}\ 4)\ \textrm{and} \ i \in\mathds{I}/M_t\} \cup\{(\sigma_{i+1}, \sigma_{m-i} ) \ | \ i \equiv 2\ (\textrm{mod}\ 4)\ \textrm{and}\ i \in \mathds{I}/M_t\}$.
\end{center}

\noindent Similarly to the proof of Proposition \ref{prop:m1even}, we see that $A \subseteq \{(\sigma_i, \sigma_{m-i+1})\ | \ i \equiv 0,3 \ (\textrm{mod}\ 4)\ \textrm{and}\ 3\leqslant i \leqslant m-3\}.$ It suffices to show that all pairs in $ \{(\sigma_i, \sigma_{m-i+1})\ | \ i \equiv 0,3 \ (\textrm{mod}\ 4)\ \textrm{and}\ 3\leqslant i \leqslant m-3\}$ are hamiltonian pairs. Below, we consider two additional cases.

\noindent \underline{Case 4}: Pairs of the form $(\sigma_i, \sigma_{m-i+1})$ and $(\mu_{i+1}, \tau_{m-i-2})$, where $ i \equiv 0 \ (\textrm{mod}\ 4) \ \textrm{and}\ 4\leqslant i \leqslant m-3$. 

\noindent SUB-CASE 4.1: $m=7$ and $i=4$. Then:

   \smallskip
{\centering
  $ \displaystyle
    \begin{aligned} 
&\mu_{4} \, \tau_{4}&&=&&\gamma_1\, \gamma_4\, (6, 0) \,\gamma_{-1}\, \gamma_{4} \, (6, 0)\\
&&&=&&(0, 2, 4, 6, 3, 5, 1).
      \end{aligned}
  $ 
\par}
\smallskip 

\noindent SUB-CASE 4.2: $m \geqslant 11$ and $i=4$. Then:

   \smallskip
{\centering
  $ \displaystyle
    \begin{aligned} 
&\mu_{4} \, \tau_{m-3}&&=&&\gamma_1\, \gamma_4\, (m-1, k+3) \,\gamma_{-1}\, \gamma_{-2} \, (m-1, 0)\\
&&&=&&(0, 2, 4, \ldots, m-3, m-1,k, k+2, k+4, \ldots, m-2, 1, 3, \ldots, k-2). 
      \end{aligned}
  $ 
\par}
\smallskip 

\noindent SUB-CASE 4.3: $i=m-3$. Then:

   \smallskip
{\centering
  $ \displaystyle
    \begin{aligned} 
&\mu_{m-3} \, \tau_{4}&&=&&\gamma_1\, \gamma_{-2}\, (m-1, 0) \, \gamma_{3} \, (m-1, k+3)\\
&&&=&&(0, 2, 4, \ldots, k+1, m-1, 3, 5, 7, \ldots, m-2,1, k+3, k+5, \ldots, m-3).
      \end{aligned}
  $ 
\par}
\smallskip

\noindent SUB-CASE 4.4:  $i \not \in \{4, m-3\}$.  Since $i$ is even, Lemma \ref{lem:comput} implies that  $\tau_{m-i-2}=\gamma_{-1}\, \sigma_{m-i+2}= \gamma_{-1}\ \gamma_{-i+2}\, (m-1, m-j+1)=\gamma_{-i+1}\, (m-1, m-j+1)$. Note that, since $k$ is odd and $j$ is even, it follows that $m-j-1$ is even and $k-j$ is odd. If $i \neq m-3$, then:

\smallskip
{\centering
  $ \displaystyle
    \begin{aligned} 
&\mu_{i} \, \tau_{m-i+1}&&=&&\gamma_1\, \gamma_i\, (m-1, k+j+1) \, \gamma_{-i+1} \, (m-1, m-j+1)\\
&&&=&&(0, 2, 4, \ldots, m-j-1, m-1,k-j+2, k-j+4, k-j+6, \ldots, m-2, 1, 3, \\
&&&&&  \ldots, k-j, m-j+1, m-j+3, \ldots, m-3). 
      \end{aligned}
  $ 
\par}
\smallskip 

\noindent \underline{Case 5}: Pairs of the form $(\sigma_i, \sigma_{m-i+1})$ and $(\mu_{i+1}, \tau_{m-i-2})$, where $ i \equiv 3 \ (\textrm{mod}\ 4) \ \textrm{and}\ 3\leqslant i \leqslant m-4$. 

\noindent SUB-CASE 5.1: $i=3$. Then:

   \smallskip
{\centering
  $ \displaystyle
    \begin{aligned} 
&\mu_{3} \, \tau_{m-2}&&=&&\gamma_{1}\,\gamma_3\, (m-1, 3) \, \gamma_{-1} \,\gamma_{-1}\, (m-1, k+1)\\
&&&=&& (0, 2, 4, \ldots, k-1, m-1, 1, 3, 5, 7, \ldots, m-2, k+1, k+3, \ldots, m-3).  \\
      \end{aligned}
  $ 
\par}
\smallskip

\noindent SUB-CASE 5.4: If  $i \equiv 3 \ (\textrm{mod}\ 4)$ and $i\neq 3$, then $i=2j+1$, where $j$ is odd, and $2 \leqslant j \leqslant k-2$. Since $i$ is odd, Lemma \ref{lem:comput} implies that  $\tau_{m-i+1}=\gamma_{-1}\, \sigma_{m-i+1}= \gamma_{-1}\ \gamma_{-i+2}\, (m-1, k-j+2)=\gamma_{-i+1}\, (m-1, k-j+2)$.  Furthermore, Lemma \ref {lem:compgam} implies that, since $i=2j+1$, we have $(m-j-1)^{\gamma_{i+1}}=j+2$ and $(j+2)^{\gamma_{-i+1}}=m-j+1$. Lastly, we note that $k-j$ is even and that $m-j+1$ is odd. Therefore:

   \smallskip
{\centering
  $ \displaystyle
    \begin{aligned} 
&\mu_{i} \, \tau_{m-i+1}&&=&&\gamma_1\, \gamma_{i}\, (m-1, j+2) \,\gamma_{-i+1}\, (m-1, k-j+2)\\
&&&=&&(0, 2, 4, \ldots, k-j, m-1,m-j+1, m-j+3, \ldots, m-2, 1, 3, 5, \ldots, \\
&&&&&   m-j-1, k-j+2,  k-j+4, \ldots, m-3). 
      \end{aligned}
  $ 
\par}
\smallskip 

In summary, all pairs in $D_H$ are hamiltonian pairs because the product of the two permutations in each pair is a permutation on a single cycle. \hfill $\square$\\

\noindent{\bf Computations for the proof of Proposition \ref{prop:m3even}}. If $m \equiv 3\ ( \textrm{mod} \ 12)$, then $m=2k+1$ where $k$ is odd. We use the following two regular permutation sets:

   \smallskip
{\centering
  $ \displaystyle
    \begin{aligned} 
&R_1= \gamma_1\,\cdot \mathcal{F}_m&=&&\{ \mu_i=\gamma_1 \sigma_i \ | \ i=0,1, \ldots, m-1\};\\
&R_2=\gamma_{-1}\, \cdot \mathcal{F}_m&=&&\{\tau_i=\gamma_{-1} \sigma_i\ | \ i=0,1, \ldots, m-1\}. 
      \end{aligned}
  $ 
\par}
   \smallskip 

Permutations $\mu_0$ and $\tau_0$ are the $(m-1)$-stabilizers of $R_1$ and $R_2$, respectively. 

Recall that $c=2t+2$ where $0\leqslant t \leqslant \frac{m-5}{2}$, $\mathds{I}=\{i \ |\ i\  \textrm{is odd}\ \ \textrm{and} \ 3 \leqslant i \leqslant m-4\}$, and $M_t$ is a subset of size $t$ of $\mathds{I}$.  Recall that $D_{T}$ and $D_H$ are constructed as follows:

   \smallskip
{\centering
  $ \displaystyle
    \begin{aligned} 
&D_T&&=&&\{(\mu_1, \tau_2), (\mu_2, \tau_1)\}\cup\{(\mu_i, \tau_{m-i}), (\mu_{i+1}, \tau_{m-i+1})\ | \ i\in M_t\}; \\
&D_H&&=&& \{(\mu_0, \tau_{m-1}), (\mu_{m-1}, \tau_0)\}\cup\{(\mu_i, \tau_{m-i+1})\ |\ 3\leqslant i \leqslant m-2\ \textrm{and}\ i, i-1 \not\in M_t\}.
      \end{aligned}
  $ 
\par}
   \smallskip 

\noindent STEP 1. We show that all pairs in $D_T$ are truncated hamiltonian pairs; we do so by considering two cases. In both cases, we use the fact that $T(\gamma_3)=T(\gamma_{1})=2$ since $m-1$ is relatively prime with 3. 

\noindent \underline{Case 1}: Pairs $(\mu_1, \tau_2)$ and $(\mu_2, \tau_1)$. We see that

 \smallskip
{\centering
  $ \scriptstyle
    \begin{aligned} 
&\hat{\mu}_{1}\hat{\tau}_{2}=\gamma_{2}\, \gamma_{1}=\gamma_{3};&& \hat{\mu}_{2}\hat{\tau}_{1}=\gamma_{3}\, \gamma_0=\gamma_{3}. 
      \end{aligned}
  $ 
\par}
\smallskip

\noindent \underline{Case 2}: Pairs of the form $(\mu_i, \tau_{m-i}), (\mu_{i+1}, \tau_{m-i+1})$ where  $i\in M_t$. Lemma \ref{lem:compgam} and Construction \ref{cons:Fodd} jointly imply that $\hat{\mu}_i=\gamma_{i+1}$, and $\hat{\tau}_{m-i+1}=\gamma_{-i+1}$. It follows that, for all $i \in M_t$, we have

 \smallskip
{\centering
  $ \scriptstyle
    \begin{aligned} 
&\hat{\mu}_{i}\hat{\tau}_{m-i}=\gamma_{i+1}\, \gamma_{-i}=\gamma_{1};&& \hat{\mu}_{i+1}\hat{\tau}_{m-i+1}=\gamma_{i+2}\, \gamma_{-i+1}=\gamma_{3}. 
      \end{aligned}
  $ 
\par}
\smallskip

\noindent Since $m-1$ is relatively prime with $3$, then $T(\gamma_3)=2$ and thus $(\mu_{i+1}, \tau_{m-i+1})$ is a truncated hamiltonian pair. In conclusion, all pairs in $D_T$ are truncated hamiltonian pairs. 

\noindent STEP 2. We show that all pairs in $D_H$ are  hamiltonian pairs. We consider two cases. \\

\noindent \underline{Case 1}: The pairs $(\mu_0, \tau_{m-1})$ and $(\mu_{m-1}, \tau_0)$. We see that

 \smallskip
{\centering
  $ \scriptstyle
    \begin{aligned} 
&\mu_0 \, \tau_{m-1}&&=&&\gamma_1 \gamma_{-1}\sigma_{m-1}=&&\sigma_{m-1}; &\mu_{m-1} \, \tau_{0}&&=&&\gamma_1\sigma_{m-1} \gamma_{-1}.
      \end{aligned}
  $ 
\par}
\smallskip 

\noindent It follows that $T(\sigma_{m-1})=T(\gamma_1\sigma_{m-1} \gamma_{-1})=1$ and thus, both pairs are hamiltonian pairs. 

\smallskip 

\noindent \underline{Case 2}: Pairs of the form $(\mu_i, \tau_{m-i+1})$ where $3\leqslant i \leqslant m-2\ \textrm{and}\ i, i-1 \not\in M_t$. 

\noindent SUB-CASE 2.1: $i\in \{4,5, m-3\}$. We see that 

 \smallskip
{\centering
  $ \scriptstyle
    \begin{aligned}
&\mu_4 \, \tau_{m-3}&&= &&\gamma_5 (m-1, k+3) \gamma_{-3}\, (m-1,0) \\
&&&=&&(0,2, 4, \ldots, m-3, m-1, k, k+2, \ldots, m-2, 1, 3, 5, \ldots, k-2);\\
&\mu_5 \, \tau_{m-4}&&=&&\gamma_6 (m-1, 4) \gamma_{-4}\, (m-1,k) \\
&&&=&&(0,2,  4, \ldots, m-3, k, k+2, k+4, \ldots, m-2, 1, 3,\ldots, k-2, m-1);\\
&\mu_{m-3} \, \tau_{4}&&=&&\gamma_{-1} (m-1, 0) \gamma_{3}\, (m-1,k+3) \\
&&&=&&(0,2, 4, \ldots, k+1, m-1, 3, 5, \ldots, m-2, k+3, k+5, \ldots, m-3).\\
      \end{aligned}
  $ 
\par}
\smallskip 

\noindent The pairs $(\mu_4, \tau_{m-3})$, $(\mu_5, \tau_{m-4})$, and $(\mu_{m-3}, \tau_{4})$ are hamiltonian pairs. 

\noindent SUB-CASE 2.2: $i \equiv 0\ (\textrm{mod}\ 4)$ and $8 \leqslant i \leqslant m-7$.  Then $i=2j$ where $j$ is even and $4\leqslant j \leqslant k-3$.  Since $i=2j$, Lemma \ref{lem:comput} implies that $\tau_{m-i+1}=\gamma_{-1}\sigma_{m-i+1}=\gamma_{-1} \gamma_{-i+2}(m-1, m-j+1)=\gamma_{-i+1}(m-1, m-j+1).$ Lastly, note that $m-j-1$ is even and $k-j+2$ is odd. Therefore:

 \smallskip
{\centering
  $ \scriptstyle
    \begin{aligned} 
&\mu_i \, \tau_{m-i+1}&&=&&\gamma_{i+1}\, (m-1, k+j+1)\, \gamma_{-i+1}\, (m-1, m-j+1)\\
&&&=&& (0,2,4,\cdots, m-j-1, m-1, k-j+2, k-j+4, \ldots, m-2, 1, 3, \ldots, k-j, \\
&&&&&m-j+1, m-j+3, \ldots, m-3).  
      \end{aligned}
  $ 
\par}
\smallskip 

\noindent SUB-CASE 2.3: $i \equiv 1\ (\textrm{mod}\ 4)$ and $9 \leqslant i \leqslant m-2$.  Then $i=2j+1$ where $j$ is even and $4\leqslant j \leqslant k-1$. Lemma \ref{lem:compgam} implies that $(m-j-1)^{\gamma_{i+1}}=j+2$ and that $(j+2)^{\gamma_{-i+1}}=m-j+1$. Since $i=2j+1$, Lemma \ref{lem:comput} implies that  $\tau_{m-i+1}=\gamma_{-1}\sigma_{m-i+1}=\gamma_{-1} \gamma_{-i+2}(m-1, k-j+2)=\gamma_{-i+1}(m-1, k-j+2).$ Therefore:

 \smallskip
{\centering
  $ \scriptstyle
    \begin{aligned} 
&\mu_i \, \tau_{m-i+1}&&=&&\gamma_{i+1}\, (m-1, j+2)\, \gamma_{-i+1}\, (m-1, k-j+2)\\
&&&= &&(0,2,4,\cdots, m-j-1, k-j+2, k-j+4, \ldots, m-2, 1, 3, \ldots, k-j,m-1, \\
&&&&& m-j+1, m-j+3, \ldots, m-3).  
      \end{aligned}
  $ 
\par}
\smallskip

\noindent SUB-CASE 2.4: $i \equiv 2\ (\textrm{mod}\ 4)$ and $6\leqslant i \leqslant m-5$. Then $i=2j$ where $j$ is odd and $3 \leqslant j \leqslant k-3$:

 \smallskip
{\centering
  $ \scriptstyle
    \begin{aligned} 
&\mu_i \, \tau_{m-i+1}&&=&&\gamma_{i+1}\, (m-1, k+j+1)\, \gamma_{-i+1}\, (m-1, m-j+1)\\
&&&= &&(0,2,4,\cdots, k-j, m-j+1, m-j+3, \ldots, m-2, 1, 3, \ldots,m-j-1,\\
&&&&&  m-1, k-j+2, k-j+4, \ldots, m-3).  
      \end{aligned}
  $ 
\par}
\smallskip 

\noindent SUB-CASE 2.5: $i \equiv 3\ (\textrm{mod}\ 4)$ and $3\leqslant i \leqslant m-4$.  Then $i=2j+1$ where $j$ is odd and $1\leqslant j \leqslant k-2$. Lemma \ref{lem:compgam} implies that $(j+2)^{\gamma_{-i+1}}=m-j+1$ and $(m-j-1)^{\gamma_{i+1}}=j+2$:

 \smallskip
{\centering
  $ \scriptstyle
    \begin{aligned} 
&\mu_i \, \tau_{m-i+1}&&=&&\gamma_{i+1}\, (m-1, j+2) \gamma_{-i+1}\, (m-1, k-j+2)\\
&&&=&& (0,2,4,\cdots, k-j, m-1,m-j+1, \ldots, m-2, 1, 3, 5, \ldots, m-j-3,\\
&&&&& m-j-1, k-j+2, k-j+4, \ldots, m-3).  
      \end{aligned}
  $ 
\par}
\smallskip 

In summary, all pairs in $D_H$ are hamiltonian because the product of their respective permutations is a permutation on a single cycle. \hfill $\square$

\noindent{\bf Computations for the proof of Proposition \ref{prop:1mod4<}}.  If $m \equiv 1\ ( \textrm{mod} \ 4)$, then $m=2k+1$ where $k$ is even. We use the following two regular permutation sets:

   \smallskip
{\centering
  $ \displaystyle
    \begin{aligned} 
&R_1=(1,m-1)\cdot \mathcal{F}_m=\{\mu_i=(1, m-1)\sigma_i \ | \ i=0,1, \ldots, m-1\}; && R_2=\mathcal{F}_m=\{\sigma_0, \sigma_1, \ldots, \sigma_{m-1}\}. 
      \end{aligned}
  $ 
\par}
   \smallskip

Observe that $\mu_1$ and $\sigma_0$ are the $(m-1)$-stabilizers of $R_1$ and $R_2$ respectively. We also point out that, if $i \not\in \{1, m-2, m-1\}$, then $(m-1)^{\mu_i}=1^{\sigma_i}=i+1.$ Moreover, we see that $(m-1)^{\mu_{m-2}}=1^{\sigma_{m-2}}=0, \ \textrm{and} \ (m-1)^{\mu_{m-1}}=1^{\sigma_{m-1}}=2.$ 

Recall that $c=3+2t$, where $0\leqslant t \leqslant \frac{m-5}{4}$, $\mathds{I}=\{i \ |\ i \equiv 3\ (\textrm{mod}\ 4)\ \textrm{and} \ 3 \leqslant i \leqslant m-6\}$, and $M_t$ is a subset of size $t$ of $\mathds{I}$. The sets $D_{T}$ and $D_H$ are constructed as follows:

   \smallskip
{\centering
  $ \displaystyle
    \begin{aligned} 
&D_T&&=&& \{(\mu_{0}, \sigma_{1}), (\mu_{m-2}, \sigma_{m-2}), (\mu_{m-1}, \sigma_{m-1})\}\cup\{(\mu_{i}, \sigma_{m-i-3}), (\mu_{i+3}, \sigma_{m-i-2})\ |\ i \in M_t\};\\
&D_H&&=&& \{(\mu_2, \sigma_0)\}\cup \{(\mu_i, \sigma_{m-i-2})\ | \ i\ \textrm{odd}, \ 1 \leqslant i \leqslant m-4, \textrm{and}\ i\not\in M_t\}\cup\\
&&&&&\{(\mu_i, \sigma_{m-i})\ | \ i\ \textrm{even},\ 4 \leqslant i \leqslant m-3,  \textrm{and}\ i, i-3\not\in M_t\}.
      \end{aligned}
  $ 
\par}
   \smallskip 
   
\noindent STEP 1. We show that all pairs in $D_T$ are truncated hamiltonian pairs; we do so by considering two cases.

\noindent \underline{Case 1}: Pairs $(\mu_{m-1}, \sigma_{m-1})$, $(\mu_{0}, \sigma_{1})$, and $(\mu_{m-2}, \sigma_{m-2})$. First, we show that  $(\mu_{m-1}, \sigma_{m-1})$ is a truncated hamiltonian pair. We consider two sub-cases. 

\noindent SUB-CASE 1.1: $m=5$. Then 

   \smallskip
{\centering
  $ \scriptstyle
    \begin{aligned} 
&\hat{\mu}_{4}\hat{\sigma}_{4}&&=&& (0,3,2)(1)(4)\,(0,3,1,2)(4)=(0,1,2, 3)(4).
      \end{aligned}
  $ 
\par}
   \smallskip
   
\noindent SUB-CASE 1.2: $m\neq 5$. Then 

   \smallskip
   {\centering
  $ \scriptstyle
    \begin{aligned} 
&\hat{\mu}_{m-1}\hat{\sigma}_{m-1}&&=&& (0, 3, m-2, 4,m -3,5, \ldots,k, k+2, k+1, 2)(1)(m-1)\,(0,3, m-2,  \\
&&&&& 4, \ldots, k,  k+2, k+1, 1,2)(m-1) \\
&&& =&&(0, m-2, m-3, m-4, \ldots, k+2, 1,2,3,4,5,\cdots, k, k+1)(m-1).\\
      \end{aligned}
  $ 
\par}   
\smallskip

\noindent Next, we show that $(\mu_{0}, \sigma_{1})$ is truncated hamiltonian.  Recall that $T(\gamma_1)=2$:

   \smallskip
{\centering
  $ \scriptstyle
    \begin{aligned} 
&\hat{\mu}_{0}\hat{\sigma}_{1}&&=&&id\,\gamma_1=\gamma_1.
      \end{aligned}
  $ 
\par}
\smallskip 

\noindent Lastly, we show that $(\mu_{m-2}, \sigma_{m-2})$ is truncated hamiltonian. Recall that $k$ is even. For all $m\equiv 1 \ (\textrm{mod}\ 4)$, we have that:

 {\centering
  $ \scriptstyle
    \begin{aligned} 
 &\hat{\mu}_{m-2}&&=&&(1, m-1)\,\gamma_{-1}\,(m-1, k+1)(m-1, 0)\\
& &&=&&(0, m-2,m-3, \ldots, k+3, k+2)(1, k+1, k, k-1, \ldots, 2)(m-1). 
       \end{aligned}
  $ 
\par} 

\noindent This means that

   \smallskip
{\centering
  $ \scriptstyle
    \begin{aligned} 
&\hat{\mu}_{m-2}\hat{\sigma}_{m-2}&&=&&(0, m-2, \ldots, k+3, k+2)(1, k+1, k, k-1, \ldots, 2)\gamma_{-1}(m-1)\\
&&&=&&(0, m-3, m-5, \ldots, k+2, m-2, m-4, m-6,  \ldots, 3, 1, k, \\
&&&&&  k-2, k-4,\ldots, 2 )(m-1).
      \end{aligned}
  $ 
\par}
\smallskip 

\noindent \underline{Case 2}: Pairs $(\mu_{i}, \sigma_{m-i-3})$ and $(\mu_{i+3}, \sigma_{m-i-2})$ where $ i \in M_t$. Note that  $\hat{\mu}_i=\mu_{i}\,(m-1, i+1)$ and that  $\hat{\mu}_{i+3}=\mu_{i+3}\,(m-1, i+4)$ if $i \in M_t$. 

\noindent SUB-CASE 2.1: $(\mu_{3}, \sigma_{m-6})$ and $(\mu_{6}, \sigma_{m-5})$:

   \smallskip
{\centering
  $ \scriptstyle
    \begin{aligned}
&\hat{\mu}_3\hat{\sigma}_{m-6}&&=&& (1, m-1)\,\gamma_3\,(m-1,3)\,(m-1, 4)\, \gamma_{-5}\\
&&& =&&(0, m-2, m-4, \ldots, 1, m-3, m-5, \ldots, 2)(m-1);\\
&\hat{\mu}_{6}\hat{\sigma}_{m-5}&&= &&(1, m-1)\,\gamma_{6}\,(m-1,k+4)\,(m-1, 7)\, \gamma_{-4}\\
&&& =&&(0, 2, 4, 6, \ldots, k-2, 3,  5, 7, \ldots, m-2, 1, k,  k+2, \\ 
&&&&& k+4, \ldots, m-3)(m-1).\\
      \end{aligned}
  $ 
\par}
\smallskip

\noindent SUB-CASE 2.2: Pairs of the form $(\mu_{i}, \sigma_{m-i-3})$ where $i \equiv 3 \ (\textrm{mod}\ 4)$ and $7 \geqslant  i \leqslant m-6$. Then $i=2j+1$ where $j$ is odd and $3\leqslant j\leqslant k-3$. Lemma \ref{lem:compgam} and Remark \ref{rem:ned} jointly imply that $\hat{\sigma}_{m-i-3}=\gamma_{-i-2}$. Lastly, Lemma \ref{lem:compgam} implies that $(m-j)^{\gamma_i}=j+2$ and $(j+2)^{\gamma_{-i-2}}=m-j-2$. Therefore:

   \smallskip
{\centering
  $ \scriptstyle
    \begin{aligned}
&\hat{\mu}_i\hat{\sigma}_{m-i-3}&&=&& (1, m-1)\,\gamma_i\,(m-1, j+2)\,(m-1, i+1)\, \gamma_{-i-2}\\
&&& =&&(0, m-3, m-5, \ldots, m-j, m-2,m-4, \ldots, 3, 1, m-j-2, \\
&&&&& m-j-4, m-j-6, \ldots, 2)(m-1).\\
      \end{aligned}
  $ 
\par}
\smallskip

\noindent SUB-CASE 2.3: Pairs of the form $(\mu_{i+3}, \sigma_{m-i-2})$ where $i \equiv 3 \ (\textrm{mod}\ 4)$ and $7 \geqslant  i \leqslant m-6$. Note that  $\hat{\sigma}_{m-i-2}=\gamma_{-i-1}$. Furthermore, we note that $i+3=2j+4=2(j+2)$. Construction \ref{cons:Fodd} then implies that $\sigma_{i+3}=\gamma_{i+3}\,(m-1,k+j+3)$. This means that $\mu_{i+3}=(1, m-1)\gamma_{i+3}\,(m-1,k+j+3)$. Therefore:

   \smallskip
{\centering
  $ \scriptstyle
    \begin{aligned}
&\hat{\mu}_{i+3}\hat{\sigma}_{m-i-2}&&=&& (1, m-1)\,\gamma_{i+3}\,(m-1,k+j+3)\,(m-1, i+4)\, \gamma_{-i-1}\\
&&& =&&(0, 2, 4, 6, \ldots, k-j-1, 3,  5, \ldots, m-2, 1, k-j+1,  k-j+3, \\ 
&&&&& k-j+5, \ldots, m-3)(m-1).\\
      \end{aligned}
  $ 
\par}
\smallskip

In conclusion, we see that all pairs in $D_T$ are truncated hamiltonian pairs. 

   \smallskip
\noindent STEP 2. We show that all pairs in $D_H$ are  hamiltonian pairs. We consider two cases. 

\noindent \underline{Case 1}: The pair $(\mu_2, \tau_{0})$. Note that we assume that $m\geqslant 9$ and thus $k \geqslant 4$:

   \smallskip
{\centering
  $ \scriptstyle
    \begin{aligned} 
&\mu_2\,\sigma_0&&=&&(1, m-1)\, \sigma_2\, id=(1,m-1)\gamma_2(m-1, k+2)\\
&&&=&&(0,2,4,\ldots, k, m-1, 3, 5, \ldots, m-2,1,  k+2, k+4, \ldots, m-3).
      \end{aligned}
  $ 
\par}
\smallskip 

\noindent \underline{Case 2}: Pairs $(\mu_i, \tau_{m-i-2})$ where $i$ is odd. We consider two sub-cases. 

   \smallskip
\noindent SUB-CASE 2.1: $i=1$. Then:

   \smallskip
{\centering
  $ \scriptstyle
    \begin{aligned} 
&\mu_{1}\sigma_{m-3}&&=&&(1, m-1)\, \gamma_1\ (m-1, 2) \, \gamma_{-2}\, (m-1, 0) =&&(0, m-2, m-3, \ldots,  2, 1, m-1).\\
      \end{aligned}
  $ 
\par}
\smallskip 

\noindent SUB-CASE 2.2: $i\neq 1$. If $i$ is odd and $3 \leqslant i \leqslant m-4$, then $i=2j+1$ where $1 \leqslant j \leqslant k-2$. Since $i$ is odd,  Lemma \ref{lem:comput} implies that $\sigma_{m-i-2}= \gamma_{-i-1}\, (m-1, m-j-1)$. Lemma \ref{lem:compgam} implies that $(m-j)^{\gamma_i}=j+2$ and $(j+2)^{\gamma_{-i-1}}=m-j-1$. Therefore:

   \smallskip
{\centering
  $ \scriptstyle
    \begin{aligned} 
&\mu_{i}\sigma_{m-i-2}&&=&&(1, m-1)\, \gamma_i\ (m-1, j+2) \, \gamma_{-i-1}\, (m-1, m-j-1)\\
&&&=&&(0, m-2, m-3, \ldots, m-j, m-j-1, m-j-2, \ldots,  2, 1, m-1).\\
      \end{aligned}
  $ 
\par}
\smallskip 

In conclusion, since $T(\mu_i\tau_{m-i-2})=1$, all pairs $(\mu_i, \tau_{m-i-2})$ where $i$ is odd and $i \neq 1$ are hamiltonian pairs. 

\smallskip 
\noindent \underline{Case 3}: Pairs $(\mu_i, \tau_{m-i})$ where $i$ is even. We consider two sub-cases. 

   \smallskip
\noindent SUB-CASE 2.1: $i=m-3$. Then:

   \smallskip
{\centering
  $ \scriptstyle
    \begin{aligned} 
&\mu_{m-3}\sigma_{3}&&=(1, m-1)\, \gamma_{-2}\ (m-1, 0) \, \gamma_{3}\, (m-1, 3)=(0, 1, m-1, 2,3,4,\cdots, m-2).\\
      \end{aligned}
  $ 
\par}
\smallskip

\noindent SUB-CASE 2.2: $i\neq m-3$. If $i$ is even and $4 \leqslant i \leqslant m-5$, then $i=2j$ where $2 \leqslant j \leqslant k-2$. Lemma \ref{lem:comput} implies that  $\sigma_{m-i}= \gamma_{-i+1}\, (m-1, k-j+2)$.  Therefore:

   \smallskip
{\centering
  $ \scriptstyle
    \begin{aligned} 
&\mu_{i}\sigma_{m-i}=(1, m-1)\, \gamma_i\ (m-1, k+j+1) \, \gamma_{-i+1}\, (m-1, k-j+2)=(0, 1, m-1, 2,3,4,\cdots, m-2).\\
      \end{aligned}
  $ 
\par}
\smallskip 

In conclusion, since $T(\mu_i\tau_{m-i})=1$, all pairs of the form $(\mu_i, \tau_{m-i})$ where $i$ is even and $i \neq m-3$ are hamiltonian pairs. \hfill $\square$ \\

\noindent{\bf Computations for the proof of Proposition \ref{prop:3mod4}}.  If $m \equiv 3\ ( \textrm{mod} \ 4)$, then $m=2k+1$ where $k$ is odd. We use the following two regular permutation sets of order $m$:

   \smallskip
{\centering
  $ \displaystyle
    \begin{aligned} 
&R_1=(0,m-1)\cdot \mathcal{F}_m=\{ \mu_i=(0, m-1)\sigma_i\ |\ i=0,1,\cdots, m-1 \}; && R_2=\mathcal{F}_m=\{\sigma_0, \sigma_1, \ldots, \sigma_{m-1}\}. 
      \end{aligned}
  $ 
\par}
   \smallskip

Observe that $\mu_3$ and $\sigma_0$ are the $(m-1)$-stabilizers of $R_1$ and $R_2$ respectively. We also point out that $(m-1)^{\mu_i}=0^{\sigma_i}=i$ if $i \not\in \{3, m-1\}$ and $(m-1)^{\mu_{m-1}}=3$. This means that  $\hat{\mu}_i=\mu_i (m-1, i)$ if $i \not \in \{3,m-1\}$, and $\hat{\mu}_{m-1}=\mu_{m-1}(m-1,3)$.

We consider two cases.  

\noindent \underline{Case 1}:  $m\equiv 3 \ (\textrm{mod}\ 8)$ and $3\leqslant c \leqslant \frac{m-5}{2}$. Recall that $c=3+2t$ where $0 \leqslant t \leqslant \frac{m-11}{4}$, $\mathds{I}=\{i \ |\ i \equiv 5, 6\ (\textrm{mod}\ 8)\ \textrm{and} \ 5\leqslant i\leqslant m-13\}$, and $M_t$ is a subset of size $t$ of $\mathds{I}$. The sets $D_{T}$ and $D_H$ are constructed as follows:

   \smallskip
{\centering
  $ \displaystyle
    \begin{aligned} 
    &D_T&&=&& \{(\mu_{m-2}, \sigma_{m-2}), (\mu_{m-1}, \sigma_{m-1}), (\mu_{0}, \sigma_1)\}\cup \\
    &&&&&\{(\mu_i, \sigma_{m-i-5}), (\mu_{i+5}, \sigma_{m-i-2})\ |\ i\in M_t\ \textrm{and}\ i \equiv 5\ (\textrm{mod}\ 8)\}\cup \\ 
   &&&&& \{(\mu_i, \sigma_{m-i-5}), (\mu_{i+3}, \sigma_{m-i})\ |\ i\in M_t \ \textrm{and} \ i \equiv 6\ (\textrm{mod}\ 8)\};\\
&D_H&&=&&\{(\mu_2, \sigma_0)\}\cup \{(\mu_i, \sigma_{m-i-2})\ | \ i, i-3 \not \in M_t,  i\ \textrm{odd}, \ \textrm{and}\ 1 \leqslant i \leqslant m-4\} \cup \\
&&&&&\{(\mu_i, \sigma_{m-i}) \ |\ i, i-5 \not\in M_t,  i\ \textrm{even}, \ \textrm{and}\ 4 \leqslant i \leqslant m-3\}.
      \end{aligned}
  $ 
\par}
   \smallskip

\noindent STEP 1. We show that all pairs in $D_T$ are truncated hamiltonian. 

\noindent \underline{Sub-Case 1.1.1}: Pairs $(\mu_{m-2}, \sigma_{m-2}), (\mu_{m-1}, \sigma_{m-1})$, and $(\mu_{0}, \sigma_1)$. Note that $\hat{\mu}_0=id$ and that Lemma \ref{lem:compgam} implies that $\gamma_{m-2}=\gamma_{-1}$:

   \smallskip
{\centering
  $ \scriptstyle
    \begin{aligned}
&\hat{\mu}_{m-2}\hat{\sigma}_{m-2}&&=&&(0, m-1)\gamma_{-1}\,(m-1,k+1)(m-1,m-2)\, \gamma_{-1}\\
&&& =&&(0,k, k-2, \ldots, 3, 1, m-2, m-4, \ldots, k+2, m-3, m-5, \ldots, \\
&&&&& k+1, k-1, \ldots, 2)(m-1);\\
&\hat{\mu}_{m-1}\hat{\sigma}_{m-1}&&=&& (0,m-1)\ (0, 3,m-2,4, \ldots, k+2, k+1, m-1, 1,2 )\,(m-1,3)\, \\
&&&&&( 0,3, m-2, 4, m-3, 5, \ldots, k+2, k+1, m-1, 1,2)\, (m-1, 1) \\
&&& =&&( 0,2,3,4, 5,6, \ldots, k+1, m-2, m-3,  \ldots, k+2, 1)(m-1); \\
&\hat{\mu}_{0}\hat{\sigma}_{1}&&=&& \gamma_{1}.
      \end{aligned}
  $ 
\par}
\smallskip

In conclusion, the three pairs $(\mu_{m-2}, \sigma_{m-2}), (\mu_{m-1}, \sigma_{m-1})$, and $(\mu_{0}, \sigma_1)$ are truncated hamiltonian pairs. 

\noindent \underline{Sub-Case 1.1.2}: Pairs of the form $(\mu_i, \sigma_{m-i-5})$ where $ i\in M_t\ \textrm{and}\ i \equiv 5\ (\textrm{mod}\ 8)$. Since  $m \equiv 3\ (\textrm{mod}\ 8)$ and $m=2k+1$, then $k  \equiv 1\ (\textrm{mod}\ 4)$. Additionally, if $i \equiv 5\ (\textrm{mod}\ 8)$, then $i=2j+1$ where $j \equiv 2\ (\textrm{mod}\ 4)$ and $2 \leq j \leq k-7$. We also remind the reader that $(m-1)^{\mu_i}=0^{\sigma_i}=i$ if $i \not\in \{3, m-1\}$. It follows that. $\hat{\mu}_i=(0,m-1)\,\sigma_i (m-1, i)$ if $i \not\in \{3, m-1\}$. Lastly,  Lemma \ref{lem:compgam} implies that $\gamma_{m-i-5}=\gamma_{-i-4}$ and that $(j+2)^{\gamma_{-i-4}}=m-j-4$ and $(m-j)^{\gamma_i}=j+2$. We see that:

   \medskip
{\centering
  $ \scriptstyle
    \begin{aligned}
&\hat{\mu}_i\hat{\sigma}_{m-i-5}&&=&&(0,m-1)\,\sigma_i\,(m-1, i)\, \gamma_{-i-4}=(0,m-1)\,\gamma_i\,(m-1,j+2)\,(m-1, i)\, \gamma_{-i-4}\\
&&& =&&(0, m-j-4, m-j-8, \ldots,   5, 1, m-4, m-8, m-12, \ldots, 7,3,\\
&&&&& m-2, m-6, \ldots, m-j, m-5, m-9, \ldots, 2, m-3, m-7, \ldots, 8, 4 )(m-1). \\
      \end{aligned}
  $ 
\par}
\medskip 

\noindent \underline{Sub-Case 1.1.3}: Pairs of the form $(\mu_{i+5}, \sigma_{m-i-2})$ where $ i\in M_t\ \textrm{and}\ i \equiv 5\ (\textrm{mod}\ 8)$. If $i=2j+1$, then $i+5=2j+6=2(j+3)$. Construction \ref{cons:Fodd} then implies that $\sigma_{i+5}=\gamma_{i+5} (m-1, k+j+4)$. We then see that

   \medskip
{\centering
  $ \scriptstyle
    \begin{aligned}
&\hat{\mu}_{i+5}\hat{\sigma}_{m-i-2}&&=&& (0, m-1)\,\sigma_{i+5}\,(m-1, i+5)\, \gamma_{-i-1}\\
&&&=&& (0, m-1)\,\gamma_{i+5}\,(m-1,k+j+4)\,(m-1, i+5)\, \gamma_{-i-1}\\
&&& =&&( 0, k-j+2, k-j+6, k-j+10, \ldots,  m-2, 3, 7, \ldots, m-4, 1, \\
&&&&& 5, 9, 13, \ldots,  k-j-2, 4, 8, 12, \ldots, m-3, 2, 6, \ldots, m-5)(m-1). \\ 
      \end{aligned}
  $ 
\par}
\medskip 

\noindent \underline{Sub-Case 1.1.4}: Pairs of the form $(\mu_i, \sigma_{m-i-5})$ where  $ i\in M_t\ \textrm{and}\ i \equiv 6\ (\textrm{mod}\ 8)$.  Recall that $k  \equiv 1\ (\textrm{mod}\ 4)$. If $i \equiv 6\ (\textrm{mod}\ 8)$, then $i=2j$ where $j \equiv 3\ (\textrm{mod}\ 4)$ and $3\leq j \leq k-6$:

   \medskip
{\centering
  $ \scriptstyle
    \begin{aligned}
&\hat{\mu}_i\hat{\sigma}_{m-i-5}&&=&& (0,m-1)\,\gamma_i\,(m-1,k+j+1)\,(m-1, i)\, \gamma_{-i-4}\\
&&& =&&(0, k-j-3, k-j-7, \ldots, 3, m-2, m-6,\ldots, 5, 1, m-4, m-8, \\
&&&&&\ldots, k-j+1, m-5, m-9, \ldots,  2, m-3, m-7, \ldots, 8, 4)(m-1).  
      \end{aligned}
  $ 
\par}
\medskip

\noindent \underline{Sub-Case 1.1.5}: The pair $(\mu_{9}, \sigma_{m-6})$. Note that $(m-4)^{\gamma_{9}}=6$: 

   \medskip
{\centering
  $ \scriptstyle
    \begin{aligned}
&\hat{\mu}_{9}\hat{\sigma}_{m-6}&&=&& (0,m-1)\,\gamma_{9}\,(m-1, 6)\,(m-1, 9)\, \gamma_{-5}\\
&&& =&&(0, 1, 5, 9, 13, \ldots, m-2, 3, 7, \ldots, m-4,  4, 8, 16, \ldots, m-3, 2,6, \\
&&&&& \cdots, m-5)(m-1).  
      \end{aligned}
  $ 
\par}
\medskip

\noindent \underline{Sub-Case 1.1.6}:  Pairs of the form  $(\mu_{i+3}, \sigma_{m-i})$ where  $ i\in M_t,\ i \equiv 6\ (\textrm{mod}\ 8)$, and $10 \leqslant i \leqslant m-13$. If $i=2j$, then $i+3=2j+3=2(j+1)+1$. Construction \ref{cons:Fodd} implies that $\sigma_{i+3}=\gamma_{i+3}(m-1, j+3)$. Lastly, Lemma \ref{lem:compgam} implies that $(j+3)^{\gamma_{-i+1}}=m-j+3$ and $(m-j-1)^{\gamma_{i+3}}=j+3$:

   \medskip
{\centering
  $ \scriptstyle
    \begin{aligned}
&\hat{\mu}_{i+3}\hat{\sigma}_{m-i}&&=&& (0,m-1)\,\gamma_{i+3}\,(m-1, j+3)\,(m-1, i+3)\, \gamma_{-i+1}\\
&&& =&&(0, m-j+3, m-j+7, m-j+11, \ldots,  m-4, 1, 5, 9, 13, \ldots, m-2,\\
&&&&& 3, 7, \ldots, m-j-1, 4, 8, \ldots, m-3, 2, 6, \ldots, m-5)(m-1). \\ 
      \end{aligned}
  $ 
\par}
\medskip

\noindent STEP 2. We show that all pairs in $D_H$ are hamiltonian pairs. 

\noindent \underline{Sub-Case 1.2.1}: Pairs $(\mu_{2}, \sigma_{0})$. We see that:

 \medskip
{\centering
  $ \displaystyle
    \begin{aligned} 
&\mu_{2} \, \sigma_{0}&&=&&(0, m-1)\, \gamma_2\, (m-1, k+2) id\\
&&&=&&(0, k+2, k+4, \ldots, m-2, 1, 3, \ldots, k, m-1, 2, 4, 6, \ldots, m-3). 
      \end{aligned}
  $ 
\par}
\medskip 

\noindent Since $T(\mu_{2}\sigma_{0})=1$, the pair $(\mu_{2}, \sigma_{0})$ is a hamiltonian pair. 

\smallskip

\noindent \underline{Sub-Case 1.2.2}:  The pair $(\mu_{1}, \sigma_{m-3})$. We see that:

   \medskip
{\centering
  $ \displaystyle
    \begin{aligned} 
&\mu_1 \, \sigma_{m-3}&&=&&(0, m-1)\, \gamma_1\,(m-1,2) \,\gamma_{-2}\,(m-1, 0)=(0,m-1,m-2,m-3, m-4, \ldots, 1).\\
      \end{aligned}
  $ 
\par}
\medskip 

\noindent Since $T(\mu_{1}\sigma_{m-3})=1$, the pair $(\mu_{1}, \sigma_{m-3})$ is a hamiltonian pair. 

\smallskip

\noindent \underline{Sub-Case 1.2.3}: Pairs of the form $(\mu_i, \sigma_{m-i-2})$ where $i\ \textrm{is odd} \ \textrm{and}\ 3 \leqslant i \leqslant m-4$.  then $i=2j+1$ where $1\leqslant j \leqslant k-2$. Lemma \ref{lem:comput} implies that  $\sigma_{m-i-2}= \gamma_{-i-1}\, (m-1, m-j-1)$. Moreover, Lemma \ref{lem:compgam} implies that $(j+2)^{\gamma_{-i-1}}=m-j-1$ and $(m-j)^{\gamma_{i}}=j+2$. We see that

   \medskip
{\centering
  $ \displaystyle
    \begin{aligned} 
&\mu_{i} \, \sigma_{m-i-2}&&=&&(0, m-1)\, \gamma_i\, (m-1, j+2)\, \gamma_{-i-1} \, (m-1, m-j-1)\\
&&&=&&(0,m-1, m-2, m-3, m-4, \ldots, m-j, m-j-1, m-j-2, \ldots, 1). 
      \end{aligned}
  $ 
\par}
\medskip 

\noindent Since $T(\mu_i \sigma_{m-i-2})=1$ when  $i\ \textrm{is odd} \ \textrm{and}\ 3 \leqslant i \leqslant m-4$, pairs of the form $(\mu_i, \sigma_{m-i-2})$ are hamiltonian pairs. 

\noindent \underline{Sub-Case 1.2.4}: The pair $(\mu_{m-3}, \sigma_{3})$. We see that:

   \medskip
{\centering
  $ \displaystyle
    \begin{aligned} 
&\mu_{m-3} \, \sigma_{3}&&=&&(0, m-1)\, \gamma_{-2}\, (m-1,0)\, \gamma_{3} \, (m-1, 3)=(0,m-1, 1,2,3,4, \ldots, m-3, m-2). 
      \end{aligned}
  $ 
\par}
\medskip 

\noindent Since $T(\mu_{m-3}\sigma_{3})=1$, the pair $(\mu_{m-3}, \sigma_{3})$ is a hamiltonian pair. 

\smallskip

\noindent \underline{Sub-Case 1.2.5}: Pairs of the form $(\mu_i, \sigma_{m-i})$ where $i\ \textrm{is even} \ \textrm{and}\ 4 \leqslant i \leqslant m-3$.  then $i=2j+1$ where $1\leqslant j \leqslant k-2$. If $i$ is even and $4 \leqslant i \leqslant m-5$, then $i=2j$ and $2\leqslant j \leqslant k-2$. Lemma \ref{lem:comput} implies that  $\sigma_{m-i}= \gamma_{-i+1}\, (m-1, k-j+2)$. Then

   \medskip
{\centering
  $ \displaystyle
    \begin{aligned} 
&\mu_{i} \, \sigma_{m-i}&&=&&(0, m-1)\, \gamma_i\, (m-1, k+j+1)\, \gamma_{-i+1} \, (m-1, k-j+2)\\
&&&=&&(0,m-1, 1,2,3,\cdots, k-j+1, k-j+2, k-j+3, \ldots, m-3, m-2). 
      \end{aligned}
  $ 
\par}
\medskip 

\noindent Since $T(\mu_i \sigma_{m-i})=1$ when $i$ is even and $ 4 \leqslant i \leqslant m-3$, pairs of the form $(\mu_i, \sigma_{m-i})$ are hamiltonian pairs. 

In summary, all pairs in $D_H$ are hamiltonian pairs. This concludes all computations for the case $m \equiv 7\ (\textrm{mod}\ 8)$. 

\bigskip

\noindent \underline{Case 2}:  $m \equiv 7\ (\textrm{mod}\ 8)$ and $5 \leqslant c \leqslant \frac{m-1}{2}$. Recall that $c=2t+3$ where $0\leqslant t \leqslant \frac{m-7}{4}$, $\mathds{I}=\{i \ |\ i \equiv 5, 6\ (\textrm{mod}\ 8)\ \textrm{and}\ 5 \leqslant i \leqslant m-7\}$ and $M_t$ is a subset of $\mathds{I}$ of size $t$.  Sets $D_T$ and $D_H$ are constructed as in Case 1. 

\noindent STEP 1. We show that all pairs in $D_T$ are truncated hamiltonian pairs. 

\noindent \underline{Sub-Case 2.1.1}: Pairs of the form $(\mu_i, \sigma_{m-i-5})$ where $ i\in M_t\ \textrm{and}\ i \equiv 5\ (\textrm{mod}\ 8)$. See Sub-Case 1.1.2 of Case 1. 

\noindent \underline{Sub-Case 2.1.2}: Pairs of the form $(\mu_{i+5}, \sigma_{m-i-2})$ where $ i\in M_t\ \textrm{and}\ i \equiv 5\ (\textrm{mod}\ 8)$. Since  $m \equiv 7\ (\textrm{mod}\ 8)$ and $m=2k+1$, then $k  \equiv 3\ (\textrm{mod}\ 4)$. If $i \equiv 5\ (\textrm{mod}\ 8)$, then $i=2j+1$ where $j \equiv 2\ (\textrm{mod}\ 4)$ and $2\leq j \leq k-5$. Therefore, we see that

   \medskip
{\centering
  $ \scriptstyle
    \begin{aligned}
&\hat{\mu}_{i+5}\hat{\sigma}_{m-i-2}&&=&& (0, m-1)\,\gamma_{i+5}\,(m-1,k+j+4)\,(m-1, i+5)\, \gamma_{-i-1}\\
&&& =&&( 0, k-j+2, k-j+6, k-j+10, \ldots,  m-4, 1, 5, 9, \ldots,  m-2, \\
&&&&& 3, 7, \ldots, k-j-2, 4, 8, 12, \ldots, m-3, 2, 6, \ldots, m-5)(m-1). \\ 
      \end{aligned}
  $ 
\par}
\medskip 

\noindent Since $T(\hat{\mu}_{i+5}\hat{\sigma}_{m-i-2})=2$, it follows that  $(\mu_{i+5}, \sigma_{m-i-2})$is a truncated hamiltonian pair under the given hypothesis. 

\noindent \underline{Sub-Case 2.1.3}: Pairs of the form $(\mu_i, \sigma_{m-i-5})$ where $ i\in M_t\ \textrm{and}\ i \equiv 6\ (\textrm{mod}\ 8)$. Recall that $k  \equiv 3\ (\textrm{mod}\ 4)$.  Moreover, if $i \equiv 6\ (\textrm{mod}\ 8)$ and $i=2j$, then $j \equiv 3\ (\textrm{mod}\ 4)$ and $3\leq j \leq k-4$. Consequently, we see that

   \medskip
{\centering
  $ \scriptstyle
    \begin{aligned}
&\hat{\mu}_i\hat{\sigma}_{m-i-5}&&=&& (0,m-1)\,\gamma_i\,(m-1,k+j+1)\,(m-1, i)\, \gamma_{-i-4}\\
&&& =&&(0, k-j-3, k-j-7, \ldots, 1, m-4, m-8,\cdots, 3, m-2, m-6, \ldots,\\
&&&&& k-j+1, m-5, m-9, \ldots,  2, m-3, m-7, \ldots, 8, 4)(m-1).  
      \end{aligned}
  $ 
\par}
\medskip

\noindent Since $T(\hat{\mu}_{i}\hat{\sigma}_{m-i-5})=2$, it follows that  $(\mu_i, \sigma_{m-i-5})$ is a truncated hamiltonian pair under the given hypothesis. 

\noindent \underline{Sub-Case 2.1.4}: Pairs of the form $(\mu_{i+3}, \sigma_{m-i})$ where $ i\in M_t\ \textrm{and}\ i \equiv 6\ (\textrm{mod}\ 8)$.   See Sub-Case 1.1.4 of Case 1. 

In conclusion, all pairs in $D_T$ are truncated hamiltonian pairs. 

\noindent STEP 2. See STEP 2  of Case 1 for all computations that show that pairs in $D_H$ are hamiltonian pairs. \hfill $\square$

\bigskip

\noindent{\bf Computations for the proof of Proposition \ref{prop:1mod6}}. If $m \equiv 1\ ( \textrm{mod} \ 6)$, then $m=2k+1$ where $k \equiv 0 \ (\textrm{mod}\ 3)$. We use the following two regular sets of permutations:

\smallskip
{\centering
  $ \scriptstyle
    \begin{aligned}
    &R_1=(m-1,1,2,3) \cdot \mathcal{F}_m=\{\mu_i=(m-1, 1,2,3) \sigma_i\ | \ i=0,1, \ldots, m-1 \}; \\
    &R_2= \mathcal{F}_m=\{\sigma_0, \sigma_1, \ldots, \sigma_{m-1}\}.\\
          \end{aligned}
  $ 
\par}
\smallskip

\noindent Recall that $c=\frac{m+2}{3}+2t$ where $0 \leqslant 2t \leqslant \frac{2m-8}{3}$, $\mathds{I}=\{i \ |\ i \equiv 0, 4\ (\textrm{mod}\ 6)\ \textrm{and}\ 4 \leqslant i \leqslant m-3\}$, and $M_t$ is a subset of size $t$ of $\mathds{I}$. Lastly, recall that $D_{T}$ and $D_H$ are constructed as follows:

 \smallskip
{\centering
  $ \scriptstyle
    \begin{aligned} 
   &D_{T}=&& \{ (\mu_{m-1}, \sigma_{m-1}), (\mu_0, \sigma_1)\} \cup \{(\mu_i, \sigma_{m-i+1}), (\mu_{i+1}, \sigma_{m-i})\ | \ i \in M_t\}\cup\\
   &&&\{(\mu_i, \sigma_{m-i+1}) \ |\ i\equiv 2, 3\ (\textrm{mod}\ 6)\ \textrm{and}\ 3 \leqslant i \leqslant m-4\};\\
&D_H=&&\{(\mu_1, \sigma_2), (\mu_2, \sigma_0)\} \cup \{(\mu_i, \sigma_{m-i})\ |\  i \not\in M_t, \  i\equiv 0, 4\ (\textrm{mod}\ 6),\textrm{and}\  4 \leqslant i \leqslant m-3\}\cup\\
&&&\{(\mu_i, \sigma_{m-i+2}) \ |\  i-1 \not\in M_t,  i\equiv 1, 5\ (\textrm{mod}\ 6), \ \textrm{and}\ 5 \leqslant i \leqslant m-2 \}.
      \end{aligned}
  $ 
\par}
   \smallskip

\noindent STEP 1. We show that all pairs in $D_T$ are truncated hamiltonian pairs. 

\noindent \underline{Case 1}: The pair $ (\mu_{m-1}, \sigma_{m-1})$. First, we compute the expression for $\hat{\mu}_{m-1}$:

\medskip
{\centering
  $ \scriptstyle
    \begin{aligned}
 &   \hat{\mu}_{m-1}&&=&&(m-1,1,2,3)\,(0,3,m-2, 4, m-3, \ldots, k, k+2, k+1, m-1, 1, 2)\,(m-1, 2)\\
&&&=&&(0, 3, 1)\, (2, m-2, 4, m-3, 5, \ldots, k, k+2, k+1) (m-1).\\
          \end{aligned}
  $ 
\par}
\medskip

\noindent We then see that

   \medskip
{\centering
  $ \scriptstyle
    \begin{aligned}
    &\hat{\mu}_{m-1} \hat{\sigma}_{m-1}&&=&&(0, 3, 1)\, (2, m-2, 4, m-3, 5, \ldots, k+3, k, k+2, k+1) (m-1)(0, 3, \\
    &&&&& m-2, 4, \ldots, k, k+2, k+1, 1,2)(m-1)\\
    &&&=&&(0, m-2, m-3, m-4, \ldots, k+3, k+2, 1, 3, 2, 4, 5, \ldots, k , k+1)(m-1).\
           \end{aligned}
  $ 
\par}
\medskip

\noindent \underline{Case 2}: The pair $(\mu_0, \sigma_1)$. Note that $\hat{\mu}_0=(1,2,3)$. This means that

   \medskip
{\centering
  $ \scriptstyle
    \begin{aligned}
     &\hat{\mu}_{0} \hat{\sigma}_{1}&&=&&(1,2,3)\, \gamma_1=(0, 1, 3, 2, 4, 5, 6, \ldots, m-2)(m-1). 
          \end{aligned}
  $ 
\par}
\bigskip

\noindent \underline{Case 3}: Pairs of the form $(\mu_i, \sigma_{m-i+1})$ and $(\mu_{i+1}, \sigma_{m-i})$ where $\ i \in M_t$. If $i \in \mathds{I}$ and $4 \leqslant i \leqslant m-4$, then $(m-1)^{\mu_{i+1}}=i+2$. Therefore, $\hat{\mu}_{i+1}=\mu_{i+1}\, (m-1, i+2)$. If $i=m-3$, then $\mu_{i+1}=\mu_{m-2}$ and $(m-1)^{\mu_{m-2}}=0$. This means that $\hat{\mu}_{m-3}=\mu_{m-2}\, (m-1, 0)$. We consider ten sub-cases. \\

\noindent SUB-CASE 3.1. The pair $(\mu_{5}, \sigma_{m-4})$. Lemma \ref{lem:compgam} implies that $\gamma_{m-4}=\gamma_{-3}$. We then see that

   \medskip
{\centering
  $ \scriptstyle
    \begin{aligned}
     &\hat{\mu}_{5} \hat{\sigma}_{m-4}&&=&&(m-1, 1,2,3)\, \gamma_{5} \, (m-1, 4)\,(m-1, 6) \gamma_{-3}\\
    &&&=&&(0, 2, 5, 7, 9,  \ldots, m-2, 3, 1, 4, 6, \ldots, m-3)(m-1). \\
          \end{aligned}
  $ 
\par}
\medskip

\noindent SUB-CASE 3.2. The pair $(\mu_{m-2}, \sigma_{3})$ and  $m\equiv 1 \ (\textrm{mod}\ 12)$. In that case , if $m=1+6k$, then $k$ is even. This means that

   \medskip
{\centering
  $ \scriptstyle
    \begin{aligned}
     &\hat{\mu}_{m-2} \hat{\sigma}_{3}&&=&&(m-1, 1,2,3)\, \gamma_{m-2} \, (m-1, k+1)\,(m-1, 0) \gamma_{3}\\
     &&&=&&(m-1, 1,2,3)\, \gamma_{-1} \, (m-1, k+1)\,(m-1, 0) \gamma_{3}\\
    &&&=&&(0, 2,5, 7, 9, \ldots, m-2, 1, 4, 6, 8, \ldots, k+2, 3, k+4, k+6,\ldots, m-3 )(m-1). \\
          \end{aligned}
  $ 
\par}
\medskip

\noindent SUB-CASE 3.3. The pair $(\mu_{m-2}, \sigma_{3})$ with $m\equiv 7 \ (\textrm{mod}\ 12)$. In that case, if $m=1+6k$, then $k$ is odd.  This means that

   \medskip
{\centering
  $ \scriptstyle
    \begin{aligned}
     &\hat{\mu}_{m-2} \hat{\sigma}_{3}&&=&&(m-1, 1,2,3)\, \gamma_{-1} \, (m-1, k+1)\,(m-1, 0) \gamma_{3}\\
    &&&=&&(0, 2,5, 7, 9, \ldots, k+2, 3, k+4, k+6, \ldots,  m-2, 1, 4, 6, 8,  \ldots, m-3)(m-1). \\
          \end{aligned}
  $ 
\par}
\medskip

\noindent SUB-CASE 3.4. Pairs of the form $(\mu_{i+1}, \sigma_{m-i})$ where $i\equiv 0, 4 \ (\textrm{mod}\ 12)$ and $12 \leqslant i \leqslant m-3$. In that case, if $i=2j$, then $j$ is even and $6 \leq j\leq k-2$. Moreover, Lemma \ref{lem:compgam} implies that $(m-j)^{\gamma_i}=j+1$ and $(j+2)^{\gamma_{-i+1}}=m-j+2$. Therefore:

   \medskip
{\centering
  $ \scriptstyle
    \begin{aligned}
&\hat{\mu}_{i+1} \hat{\sigma}_{m-i}&&=&&(m-1, 1,2,3)\, \gamma_{i+1} \, (m-1, j+2)\,(m-1, i+2) \gamma_{-i+1}\\
      &&&=&&( 0, 2, 5, 7, 9,  \ldots, m-j, 3, m-j+2, m-j+4,\ldots, m-2, 1, 4, 6, \ldots, m-3)(m-1). \\
          \end{aligned}
  $ 
\par}
\medskip

\noindent SUB-CASE 3.5. Pairs of the form $(\mu_{i+1}, \sigma_{m-i})$ where $i\equiv 6, 10 \ (\textrm{mod}\ 12)$ and $12 \leqslant i \leqslant m-3$. Then $i=2j$ where $j$ is odd and $3 \leq j\leq k-2$. Therefore:

   \medskip
{\centering
  $ \scriptstyle
    \begin{aligned}
&\hat{\mu}_{i+1} \hat{\sigma}_{m-i}&&=&&(m-1, 1,2,3)\, \gamma_{i+1} \, (m-1, j+2)\,(m-1, i+2) \gamma_{-i+1}\\
      &&&=&&( 0, 2, 5, 7, 9, \ldots, m-2,  1, 4, 6,   \ldots, m-j, 3, m-j+2,  m-j+4,\ldots, m-3)(m-1). \\
          \end{aligned}
  $ 
\par}
\medskip

\noindent SUB-CASE 3.6. The pair $({\mu}_{m-3}, {\sigma}_{4})$. Then

   \medskip
{\centering
  $ \scriptstyle
    \begin{aligned}
    &\hat{\mu}_{m-3} \hat{\sigma}_{4}&&=&&(m-1, 1,2,3)\, \gamma_{-2} \, (m-1, 0)\,(m-1, m-2) \gamma_{4} \\
        &&&=&&(0, 2, 5, 7, 9, \ldots, m-2, 1, 3, 4, 6, \ldots, m-3)(m-1).
          \end{aligned}
  $ 
\par}
\medskip

\noindent SUB-CASE 3.7. Pairs of the form $(\mu_{i}, \sigma_{m-i+1})$ where $i \equiv 0, 4\ (\textrm{mod}\ 12)$, $4 \leqslant i\leqslant m-9$, and $m \equiv 1\ (\textrm{mod}\ 12)$. Then $i=2j$, where $j$ is even, and $k$ is even. We then see that $2\leqslant j \leqslant k-4$. Therefore

   \medskip
{\centering
  $ \scriptstyle
    \begin{aligned}
    &\hat{\mu}_{i} \hat{\sigma}_{m-i+1}&&=&&(m-1, 1,2,3)\, \gamma_{i} \, (m-1, k+j+1)\,(m-1, i+1) \gamma_{-i+2}\\
    &&&=&&(0, 2, 5 , 7, 9, \ldots, k-j+1, 3, k-j+3, k-j+5, \ldots, m-4,  m-2, 1,\\
   &&&&&  4, 6, 8, \ldots, m-3 )(m-1). \\
          \end{aligned}
  $ 
\par}
\medskip

\noindent SUB-CASE 3.8. Pairs of the form $(\mu_{i}, \sigma_{m-i+1})$ where $i \equiv 0, 4\ (\textrm{mod}\ 12)$, $4 \leqslant i\leqslant m-7$, and $m \equiv 7\ (\textrm{mod}\ 12)$. Then $i=2j$, where $j$ is even, and $k$ is odd. In addition, we have $2\leq j \leq k-3$. As a result,

   \medskip
{\centering
  $ \scriptstyle
    \begin{aligned}
    &\hat{\mu}_{i} \hat{\sigma}_{m-i+1}&&=&&(m-1, 1,2,3)\, \gamma_{i} \, (m-1, k+j+1)\,(m-1, i+1) \gamma_{-i+2}\\
    &&&=&&(0, 2, 5, 7, 9, \ldots, m-2, 1, 4, 6,8, \ldots, k-j+1, 3, \\
   &&&&&  k-j+3, k-j+5, \ldots, m-3 )(m-1). \\
          \end{aligned}
  $ 
\par}
\medskip

\noindent SUB-CASE 3.9. Pairs of the form $(\mu_{i}, \sigma_{m-i+1})$ where  $i \equiv 6, 10 \ (\textrm{mod}\ 12)$,  $6 \leqslant i\leqslant m-7$, and  $m \equiv 1\ (\textrm{mod}\ 12)$. Then $i=2j$ where $j$ is odd, $k$ is even, and $3 \leqslant j \leqslant k-3$. The computation of $\hat{\mu}_{i} \hat{\sigma}_{m-i+1}$ is the same as SUB-CASE 3.8.\\

\noindent SUB-CASE 3.10. Pairs of the form $(\mu_{i}, \sigma_{m-i+1})$ where  $i \equiv 6, 10 \ (\textrm{mod}\ 12)$, $6 \leqslant i\leqslant m-9$, and  $m \equiv 7\ (\textrm{mod}\ 12)$. Then $i=2j$ where $j$ is odd, $k$ is even, and $3 \leqslant j \leqslant k-3$. hen $i=2j$ where $j$ is odd, $k$ is odd, and $3 \leq j\leq k-4$. The computation of $\hat{\mu}_{i} \hat{\sigma}_{m-i+1}$ is the same as SUB-CASE 3.7.\\

\bigskip

\noindent \underline{Case 4}: Pairs of the form $(\mu_i, \sigma_{m-i+1})$ where $i \equiv 2,3\ (\textrm{mod}\ 6)$ and $3 \leqslant i \leqslant m-4$. \\

\noindent SUB-CASE 4.1: $i=m-5$. Then

   \medskip
{\centering
  $ \scriptstyle
    \begin{aligned}
     &\hat{\mu}_{m-5} \hat{\sigma}_{6}&&=&&(m-1,1,2,3)\, \gamma_{-4} (m-1, m-2)\, (m-1, m-4)\, \gamma_6\\
    &&&=&&(0, 2, 3, 5, 7, 9, \ldots, m-2, 1,4,6,8, \ldots, m-3)(m-1). 
          \end{aligned}
  $ 
\par}
\medskip
\smallskip
\noindent  SUB-CASE 4.2: $i\equiv 2 \ (\textrm{mod}\ 12)$, $14 \leqslant i \leqslant m-11$, and $m \equiv 1\ (\textrm{mod}\ 12)$.  Then $i=2j$, where $j$ is odd, and $k$ is even. This also means that $1\leqslant j \leqslant k-5$.  Therefore:

   \medskip
{\centering
  $ \scriptstyle
    \begin{aligned}
    &\hat{\mu}_{i} \hat{\sigma}_{m-i+1}&&=&&(m-1,1,2,3)\, \gamma_i\, (m-1, k+j+1)\, (m-1, i+1)\, \gamma_{-i+2}\\
    &&&=&&(0, 2,5, 7, \ldots,  m-2, 1,4,6,8,\cdots, k-j+1, 3, k-j+3, k-j+5, \dots, m-3)(m-1).\\
          \end{aligned}
  $ 
\par}
\medskip

\noindent  SUB-CASE 4.3:  $i\equiv 2 \ (\textrm{mod}\ 12)$, $14 \leqslant i \leqslant m-17$, and  $m \equiv 7\ (\textrm{mod}\ 12)$. Then $i=2j$, where $j$ is odd, and $k$ is odd. This also means that $1\leqslant j \leqslant k-8$. Therefore:

   \medskip
{\centering
  $ \scriptstyle
    \begin{aligned}
    &\hat{\mu}_{i} \hat{\sigma}_{m-i+1}&&=&&(m-1,1,2,3)\, \gamma_i\, (m-1, k+j+1)\, (m-1, i+1)\, \gamma_{-i+2}\\
    &&&=&&(0,  2,5, 7,\cdots, k-j+1, 3, k-j+3, k-j+5, \ldots,  m-2, 1,4,6,8, \dots, m-3)(m-1).\\
          \end{aligned}
  $ 
\par}
\medskip

\noindent  SUB-CASE 4.4:  $i\equiv 3 \ (\textrm{mod}\ 12 )$, $3 \leqslant i \leqslant m-5$. Then $i=2j+1$ where $j$ is odd and $1 \leqslant j \leqslant k-2$.  Furthermore, observe that $(m-j)^{\gamma_i}=j+2$ and $(j+2)^{\gamma_{-i+2}}=m-j+2$. Therefore:

   \medskip
{\centering
  $ \scriptstyle
    \begin{aligned}
    &\hat{\mu}_{i} \hat{\sigma}_{m-i+1}&&=&&(m-1,1,2,3)\, \gamma_i\, (m-1, j+2)\, (m-1, i+1)\, \gamma_{-i+2}\\
    &&&=&&( 0, 2,5, 7, \ldots, m-2, 1,4,6,\cdots, m-j , 3, m-j+2, m-j+4, \ldots,  m-3)(m-1).
          \end{aligned}
  $ 
\par}
\medskip

\noindent  SUB-CASE 4.5:  $i\equiv 8 \ (\textrm{mod}\ 12 )$, $8 \leqslant i \leqslant m-17$, and $m \equiv 1\ (\textrm{mod}\ 12)$. Then $i=2j$, where $j$ is even, and $k$ is even.  This also means that $4\leqslant j \leqslant k-8$. Computation of $\hat{\mu}_{i} \hat{\sigma}_{m-i+1}$ is the same as SUB-CASE 4.3. \\

\noindent  SUB-CASE 4.6: $i\equiv 8 \ (\textrm{mod}\ 12 )$, $8 \leqslant i \leqslant m-11$, and $m \equiv 7\ (\textrm{mod}\ 12)$. Then $i=2j$, where $j$ is even, and $k$ is odd. This also means that $4\leqslant j \leqslant k-5$. The computation of $\hat{\mu}_{i} \hat{\sigma}_{m-i+1}$ is the same as SUB-CASE 4.2. \\

\noindent   SUB-CASE 4.7:  $i\equiv 9 \ (\textrm{mod}\ 12)$ and $9 \leqslant i \leqslant m-4$. Then $i=2j+1$ where $j$ is even and $4 \leqslant j \leqslant k-2$. Recall that Lemma \ref{lem:compgam} implies that $(m-j)^{\gamma_i}=j+2$ and $(j+2)^{\gamma_{-i+2}}=m-j+2$. Therefore:

   \medskip
{\centering
  $ \scriptstyle
    \begin{aligned}
    &\hat{\mu}_{i} \hat{\sigma}_{m-i+1}&&=&&(m-1,1,2,3)\, \gamma_i\, (m-1, j+2)\, (m-1, i+1)\, \gamma_{-i+2}\\
    &&&=&&(0, 2,5, 7, \ldots, m-j , 3, m-j+2, m-j+4, \ldots, m-2, 1,4,6,8,\ldots, m-3)(m-1).\\
          \end{aligned}
  $ 
\par}
\medskip

In conclusion, all pairs in $D_T$ are truncated hamiltonian pairs. 

\noindent STEP 2. We show that all pairs in $D_H$ are hamiltonian pairs.

\noindent \underline{Case 1}. The pair $(\mu_1, \sigma_2)$. We consider two sub-cases. \\

\noindent SUB-CASE 1.1:  $m=7$. Then

\medskip
{\centering
  $ \scriptstyle
    \begin{aligned}
    &\mu_1 \sigma_{2}&&=&&(6,1,2,3)\, \gamma_1 \, (6, 2) \gamma_{2}(6, 5)=( 0, 3,4,1,6,5,2).
          \end{aligned}
  $ 
\par}
\medskip

\noindent SUB-CASE 1.2: $m \geqslant 13$.  Recall that $k \equiv 0 \ (\textrm{mod}\ 3)$. Then

\medskip
{\centering
  $ \scriptstyle
    \begin{aligned}
    &\mu_1 \sigma_{2}&&=&&(m-1,1,2,3)\, \gamma_1 \, (m-1, 2) \gamma_{2}(m-1, k+2)\\
    &&&=&&( 0, 3, 4,  7, 10, \ldots, m-3, 1, 5, 8, \ldots, k-1, m-1,  k+2 , k+5,\ldots, m-2, 2, 6, 9, \ldots, m-4). 
          \end{aligned}
  $ 
\par}
\medskip

\noindent \underline{Case 2}: The pair $(\mu_2, \sigma_0)$. We consider three sub-cases. \\

\noindent SUB-CASE 2.1:  $m \equiv 1\ ( \textrm{mod}\ 12)$ and $m=2k+1$. Then $k$ is even. Therefore:

\medskip
{\centering
  $ \scriptstyle
    \begin{aligned}
    &\mu_2 \sigma_{0}&&=&&(m-1,1,2,3)\, \gamma_2 \, (m-1, k+2) \, id\\
    &&&=&&( 0, 2, 5, 7, \ldots, m-2, 1, 4, 6, 8, \ldots, k, m-1, 3, k+2, k+4, k+6, \ldots, m-3). 
          \end{aligned}
  $ 
\par}
\medskip

\noindent SUB-CASE 2.2:  $m=7$. Therefore:

\medskip
{\centering
  $ \scriptstyle
    \begin{aligned}
    &\mu_2 \sigma_{0}&&=&& (6,1,2,3)\, \gamma_2 \, (6, 5) \, id=(0, 2, 6, 3, 5, 1, 4). \\
          \end{aligned}
  $ 
\par}
\medskip

\noindent SUB-CASE 2.3: $m \equiv 7\ ( \textrm{mod}\ 12)$, $m \geqslant 19$, and $m=2k+1$. Then $k$ is odd. Therefore:

\medskip
{\centering
  $ \scriptstyle
    \begin{aligned}
    &\mu_2 \sigma_{0}&&=&& (m-1,1,2,3)\, \gamma_2 \, (m-1, k+2) \, id\\
    &&&=&&(0, 2, 5, 7, 9, \ldots, k, m-1, 3, k+2, k+4, k+6, \ldots, m-2, 1, 4, 6, \ldots, m-3). 
          \end{aligned}
  $ 
\par}
\medskip

\noindent \underline{Case 3}: Pairs of the form $(\mu_{i}, \sigma_{m-i})$ where $i \equiv 0, 4\ (\textrm{mod}\ 6)$ and $4 \leqslant i \leqslant m-3$. We consider two sub-cases. \\ 

\noindent SUB-CASE 3.1: $i=m-3$. We see that

   \medskip
{\centering
  $ \scriptstyle
    \begin{aligned}
    &\mu_{m-3} \sigma_{3}&&=&&(m-1,1,2,3)\, \gamma_{-2}\, (m-1,0)\, \, \gamma_{3}\, (m-1, 3)\\
    &&&=&&(0, 1, 3, m-1, 2,4,5, 6, \ldots, m-2).\\
          \end{aligned}
  $ 
\par}
\bigskip

\noindent SUB-CASE 3.2: $i \equiv 0, 4\ (\textrm{mod}\ 6)$ and $4 \leqslant i < m-3$. Then $i=2j$ where $2 \leqslant j \leqslant k-2$. Lemma \ref{lem:comput} implies that  $\sigma_{m-i}=\gamma_{-i+1}\, (m-1, k-j+2)$. Therefore:

   \medskip
{\centering
  $ \scriptstyle
    \begin{aligned}
    &\mu_{i} \sigma_{m-i}&&=&&(m-1,1,2,3)\, \gamma_i\, (m-1, k+j+1)\, \, \gamma_{-i+1}\, (m-1, k-j+2)\\
    &&&=&&(0, 1, 3, m-1, 2,4,5, 6,\cdots, k-j+1, k-j+2, k-j+3, \ldots, m-2).\\
          \end{aligned}
  $ 
\par}
\bigskip

\noindent \underline{Case 4}: Pairs of the form $(\mu_{i}, \sigma_{m-i}+2)$ where $i \equiv 1, 5\ (\textrm{mod}\ 6)$ and $5 \leqslant i \leqslant m-2$.
Note that, if $i \equiv 1, 5 \ (\textrm{mod}\ 6)$ and $11 \leqslant i \leqslant m-2$, then $i=2j+1$ and $5 \leqslant j \leqslant k-1$. Lemma \ref{lem:comput} implies that  $\sigma_{m-i+2}=\gamma_{-i+3}\, (m-1, m-j+1)$. We consider four sub-cases. \\ 

\noindent SUB-CASE 4.1: $i=5$. We see that

 \medskip
{\centering
  $ \scriptstyle
    \begin{aligned}
    &\mu_{5} \sigma_{m-3}&&=&&(m-1,1,2,3)\, \gamma_5\, (m-1, 4)\, \, \gamma_{-2}\, (m-1, 0)\\
    &&&=&&(0, 3, 2, 6, 9, \ldots, m-4, m-1, 4, 7, \ldots, m-3,1, 5, 8, \ldots, m-2).\\
          \end{aligned}
  $ 
\par}
\bigskip

\noindent SUB-CASE 4.2: $i=7$. We see that

   \medskip
{\centering
  $ \scriptstyle
    \begin{aligned}
    &\mu_{7} \sigma_{m-5}&&=&&(m-1,1,2,3)\, \gamma_7\, (m-1, 5)\, \, \gamma_{-4}\, (m-1, m-2)\\
    &&&=&&(0, 3, 1, 5, 8, \ldots, m-5, m-1, 4, 7, \ldots, m-3, m-2, 2 6, 9, \ldots, m-4).\\
          \end{aligned}
  $ 
\par}
\medskip

\noindent SUB-CASE 4.3: $i \equiv 1 \ (\textrm{mod}\ 6)$ and $13 \leqslant i \leqslant m-6$. If $i=2j$, then $j \equiv 0 \ (\textrm{mod}\ 3)$.  Moreover, Lemma \ref{lem:compgam} implies that $(j+2)^{\gamma_{-i+3}}=m-j+3$ and $(m-j)^{\gamma_i}=j+2$. We then see that

   \medskip
{\centering
  $ \scriptstyle
    \begin{aligned}
    &\mu_{i} \sigma_{m-i+2}&&=&&(m-1,1,2,3)\, \gamma_i\, (m-1, j+2)\, \, \gamma_{-i+3}\, (m-1, m-j+1)\\
    &&&=&&(0, 3, m-j+3, m-j+6, m-j+9, \ldots, m-3,  1, 5, 8, \ldots,  m-j-2, m-1, 4, \\ 
    &&&&&  7, 10,  \ldots, m-j,m-j+1, m-j+4, \ldots, m-2, 2, 6, 9, \ldots, m-4).\\
          \end{aligned}
  $ 
\par}
\medskip

\noindent SUB-CASE 4.4: $i \equiv 5\ (\textrm{mod}\ 6)$ and $11 \leqslant i \leqslant m-2$. Then $i=2j+1$ where $j \equiv 2 \ (\textrm{mod}\ 3)$ and $5 \leqslant j \leqslant k-1$. Therefore, 
 
   \medskip
{\centering
  $ \scriptstyle
    \begin{aligned}
    &\mu_{i} \sigma_{m-i+2}&&=&&(m-1,1,2,3)\, \gamma_i\, (m-1, j+2)\, \, \gamma_{-i+3}\, (m-1, m-j+1)\\
    &&&=&&(0, 3, m-j+3, m-j+6, \ldots, m-2, 2, 6, 9, \ldots, m-j-2,  m-1,4,7, 10, \ldots, \\
    &&&&&  m-3, 1, 5, 8, 11, \ldots, m-j, m-j+1, m-j+4,  m-j+7, \ldots, m-4).\\
          \end{aligned}
  $ 
\par}
\medskip

In conclusion, all pairs in $D_H$ are hamiltonian pairs. \hfill $\square$\\
\bigskip

\noindent{\bf Computations for the proof of Proposition \ref{prop:m3cong6}}. Since $m \equiv 3 \ (\textrm{mod}\ 6)$, and $m=2k+1$, then $k \equiv 1 \ (\textrm{mod}\ 3)$. We use the following two regular sets of $m$ permutations:

      \smallskip
{\centering
  $ \scriptstyle
    \begin{aligned}
    &R_1=(1,2,3,4) \cdot \mathcal{F}_m=\{\mu_i=(1,2,3,4) \sigma_i\ | \ i=0,1,\cdots, m-1\}; R_2= \mathcal{F}_m=\{\sigma_0, \sigma_1, \ldots, \sigma_{m-1}\}.\\
          \end{aligned}
  $ 
\par}
\smallskip

 Observe that $\mu_0$ and $\sigma_0$ are the $(m-1)$-stabilizers of $R_1$ and $R_2$, respectively. We also point out that  $(m-1)^{\mu_i}=(m-1)^{\sigma_i}$ and $\hat{\sigma}_i=\gamma_i$.  Therefore, if $i \not\in\{0, m-1\}$, then  $\hat{\mu}_i=(1,2,3,4)\,\gamma_i$. 

Recall that $c=\frac{m}{3}+4+2t$, where $0\leqslant 2t \leqslant \frac{2m}{3}-6$, $\mathds{I}=\{i \ |\ i \equiv 3, 5\ (\textrm{mod}\ 6)\ \textrm{and}\ 3\leqslant i \leqslant m-10\}$, and $M_t$ is a subset of size $t$ of $\mathds{I}$. The sets $D_{T}$ and $D_H$ are constructed as follows:

 \smallskip
{\centering
  $ \scriptstyle
    \begin{aligned} 
   &D_{T}&&=&&\{(\mu_{m-4}, \sigma_{m-4}),  (\mu_{m-2}, \sigma_{m-1}), (\mu_{m-1}, \sigma_{m-2})\}\cup\\
   &&&&&\{(\mu_i, \sigma_{m-i-5}), (\mu_{i+1}, \sigma_{m-i-4})\ | \ i \equiv 2 \ (\textrm{mod} \ 6)\  \textrm{and}\  2\leqslant i \leqslant m-15\}\cup \\
   &&&&& \{(\mu_i, \sigma_{m-i-5}), (\mu_{i+1}, \sigma_{m-i-4})\ | \ i\in \{m-10, m-8, m-6\}\}\cup \\
   &&&&&\{(\mu_i, \sigma_{m-i-5}), (\mu_{i+1}, \sigma_{m-i-4}) \ | \ i \in M_t\};\\
   &D_H&&=&&\{ (\mu_{m-3}, \sigma_0), (\mu_0, \sigma_{m-3})\}\cup \\ 
   &&&&&\{(\mu_i, \sigma_{m-i-4})\ |\ i\equiv 0,3,4,5\ (\textrm{mod}\ 6), 3\leqslant i \leqslant m-9, \textrm{and}\ i, i-1 \not\in M_t\}.
   \end{aligned}
  $ 
\par}
   \smallskip
   
\noindent STEP 1. We show that all pairs in $D_T$ are truncated hamiltonian pairs. We consider three cases. \\

\noindent \underline{Case 1:} The pairs $(\mu_{m-4}, \sigma_{m-4}),  (\mu_{m-2}, \sigma_{m-1})$ and $(\mu_{m-1}, \sigma_{m-2})$. See below:

   \medskip
{\centering
  $ \scriptstyle
    \begin{aligned}
     &\hat{\mu}_{m-4} \hat{\sigma}_{m-4}&&=&&(1,2,3,4)\, \gamma_{-3} \, \gamma_{-3}\\
     &&&=&& (0, m-7, m-13, \ldots,  8, 2, m-4, m-10, \ldots, 5, m-2, m-8,  7, 1, m-5,  \\
     &&&&& \ldots, m-11, \ldots, 4, m-6, m-12, \ldots, 9, 3, m-3, m-9, \ldots, 6)(m-1);\\
    &\hat{\mu}_{m-2} \hat{\sigma}_{m-1}&&=&&(1,2,3,4)\, \gamma_{-1} \, (0, 3, m-2, 4, \ldots, k,  k+2, k+1, 1,2)\, \\
    &&&=&&( 0, 4, 3, m-2, 5, m-3, 6, \ldots, k, k+3, k+1, k+2, 1, 2)(m-1);\\
    &\hat{\mu}_{m-1} \hat{\sigma}_{m-2}&&=&&(1,2,3,4)\, (0, 3, m-2, 4, m-3, 5, \ldots, k, k+2, k+1, 1,2) \gamma_{-1}\\
    &&&=&&(0, 2, m-3, 4, 1, m-2, 3, m-4, 5, m-5, 6, m-6, 7, \ldots, k+2, k, k+1)(m-1).
            \end{aligned}
  $ 
\par}
\medskip

\underline{Case 2:} Pairs of the form $(\mu_i, \sigma_{m-i-5})$ where $i \equiv 1,3,5 \ (\textrm{mod}\ 6)$. We consider two sub-cases. 

\noindent SUB-CASE 2.1: $m\equiv 1 \ (\textrm{mod}\ 4)$. If $m=2k+1$, then $k$ is even. Therefore:

   \medskip
{\centering
  $ \scriptstyle
    \begin{aligned}
    &\hat{\mu}_{i} \hat{\sigma}_{m-i-5}&&=&&(1,2,3,4)\, \gamma_{i} \, \gamma_{-i-4}\\
    &&&=&&( 0, m-5, m-9, \ldots, 4, m-4, m-8, \ldots, 5, 1, m-3, m-7, \ldots, 2, m-2,\\
    &&&&& m-6, \ldots, 3)(m-1).
          \end{aligned}
  $ 
\par}
\medskip

\noindent SUB-CASE 2.2: $m\equiv 3 \ (\textrm{mod}\ 4)$. If $m=2k+1$, then $k$ is odd. Therefore:

   \medskip
{\centering
  $ \scriptstyle
    \begin{aligned}
    &\hat{\mu}_{i} \hat{\sigma}_{m-i-5}&&=&&(1,2,3,4) \, \gamma_{-4}\\
    &&&=&&( 0, m-5, m-9, \ldots, 2, m-2, m-6, \ldots, 5, 1, m-3, m-7, \ldots, 4,m-4,\\
    &&&&& m-8, \ldots, 3)(m-1).
          \end{aligned}
  $ 
\par}
\medskip

\underline{Case 3:} Pairs of the form $(\mu_{i+1}, \sigma_{m-i-4})$ where $i \equiv 1,3,5 \ (\textrm{mod}\ 6)$. We see that

   \medskip
{\centering
  $ \scriptstyle
    \begin{aligned}
    &\hat{\mu}_{i+1} \hat{\sigma}_{m-i-4}&&=&&(1,2,3,4)\, \gamma_{i+1} \, \gamma_{-i-3}\\
    &&&=&&(0, m-3, m-5, \ldots, 6, 4, m-2, m-4, \ldots, 7, 5, 3, 2, 1)(m-1).
          \end{aligned}
  $ 
\par}
\medskip

In summary, all pairs in $D_T$ are truncated hamiltonian pairs. \\

\noindent STEP 2. We show that all pairs in $D_H$ are hamiltonian pairs. All pairs in $D_H$ are of the form $(\mu_i, \sigma_{m-i-4})$, where $i \equiv 0,3,4,5\ (\textrm{mod}\ 6)$ and $3 \leqslant i \leqslant m-9$. Recall that $k \equiv 1 \ (\textrm{mod}\ 3)$. We will consider five subcases. 

For SUB-CASES 1 and 2, we shall assume that $i=2j$ and $2 \leqslant i \leqslant m-9$. Then $1 \leqslant j \leqslant k-4$, and Lemma \ref{lem:comput} implies that $\sigma_{m-i-4}=\gamma_{-i-3}\, (m-1, k-j)$. 

For SUB-CASES 4 and 5, we shall assume that $i=2j+1$ and $1 \leqslant i \leqslant m-6$. Then $0 \leqslant j \leqslant k-3$ and Lemma \ref{lem:comput} implies that  $\sigma_{m-i-4}=\gamma_{-i-3}\, (m-1, m-j-2)$. 

\noindent SUB-CASE 1: $i \equiv 0\ (\textrm{mod} \ 6)$ and $6\leq i \leq m-9$. Then  $i=2j$ where $j\equiv 0\ (\textrm{mod}\ 3)$ and $3 \leqslant j \leqslant k-4$.  Therefore:

\medskip
{\centering
  $ \scriptstyle
    \begin{aligned}
    &\mu_i \sigma_{m-i-4}&&=&&(1,2,3,4)\, \gamma_i \, (m-1, k+j+1) \gamma_{-i-3}(m-1, k-j)\\
   &&&=&&( 0, m-4, m-7, \ldots, k-j+1, k-j, k-j-3, k-j-6,  \ldots, 7,4, m-3, m-6,  \\
    &&&&& \ldots, 3, 1, m-2, m-5, \ldots, k-j+3, m-1,k-j-2, k-j-5, \ldots, 2). 
          \end{aligned}
  $ 
\par}
\bigskip

\noindent SUB-CASE 2: $i \equiv 4\ (\textrm{mod} \ 6)$ and $4\leq i \leq m-11$. Then  $i=2j$ where $j \equiv 2 \ (\textrm{mod}\ 3)$ and $2 \leq j \leq k-6$. Therefore:

\medskip
{\centering
  $ \scriptstyle
    \begin{aligned}
    &\mu_i \sigma_{m-i-4}&&=&&(1,2,3,4)\, \gamma_i \, (m-1, k+j+1) \gamma_{-i-3}(m-1, k-j)\\
    &&&=&&(0, m-4,m-7,  \ldots,  k-j+3, m-1, k-j-2, k-j-5, \ldots, 3, 1,  \\
    &&&&& m-2,  m-5, \ldots, 4, m-3, m-6, \ldots, k-j+1, k-j,  k-j-3,  \ldots, 2). 
          \end{aligned}
  $ 
\par}
\bigskip

\noindent SUB-CASE 3: $i=3$. Therefore:

\medskip
{\centering
  $ \scriptstyle
    \begin{aligned}
        &\mu_3 \sigma_{m-7}&&=&&(1,2,3,4)\, \gamma_3\, (m-1, 3) \, \gamma_{-6}\, (m-1, m-3)\\
        &&&=&&( 0, m-3, m-6,  \ldots, 6, 3, 1, m-2, m-5, \ldots, 4, m-1, m-4, m-7, \ldots, 5, 2).
       \end{aligned}
        $ 
\par}
\bigskip

\noindent SUB-CASE 4: $i \equiv 3 \ (\textrm{mod} \ 6)$ and $9 \leq i \leq m-12$. Then $i=2j+1$ where $j \equiv 1 \ (\textrm{mod} \ 3)$ and $4 \leq j \leq k-6$. Lemma \ref{lem:compgam} implies that $(m-j)^{\gamma_i}=j+2$ and $(j+2)^{\gamma_{-i-3}}=m-j-3$. Therefore:
  
    \medskip
{\centering
  $ \scriptstyle
    \begin{aligned}
    &\mu_i \sigma_{m-i-4}&&=&&(1,2,3,4)\, \gamma_i \, (m-1, j+2) \gamma_{-i-3}(m-1, m-j-2)\\
    &&&=&&( 0, m-4, m-7, \ldots, m-j, m-j-2, m-j-5,  \ldots , 6, 3, 1, m-2,  \\
    &&&&& m-5, \ldots, 4, m-3, \ldots, m-j+1, m-1, m-j-3, m-j-6, \ldots, 5, 2)\\
    \end{aligned}   
  $ 
\par}
\bigskip

\noindent SUB-CASE 5: $i \equiv 5\  (\textrm{mod} \ 6)$ and $5 \leq i \leq m-10$. Then $i=2j+1$ where $j \equiv 2 \ (\textrm{mod} \ 3)$ and $2 \leqslant j \leqslant k-5$. Therefore:

   \medskip
{\centering
  $ \scriptstyle
    \begin{aligned}
    &\mu_i \sigma_{m-i-4}&&=&&(1,2,3,4)\, \gamma_i \, (m-1, j+2) \gamma_{-i-3}(m-1, m-j-2)\\
    &&&=&&( 0,  m-4, m-7, \ldots,  m-j+1, m-1,m-j-3, m-j-6, \ldots, 7, 4, \\
    &&&&&  m-3, m-6, \ldots, 3, 1, m-2, m-5,  \ldots, m-j, m-j-2,  m-j-5,  \ldots , 5, 2). 
          \end{aligned}
  $ 
\par}
\medskip

In conclusion, all pairs in $D_H$ are hamiltonian pairs. \hfill $\square$\\

\noindent{\bf Computations for the proof of Proposition \ref{prop:5mod6}}. Let $m \equiv 5 \ (\textrm{mod}\ 6)$, where $m=2k+1$. Then $k \equiv 2 \ (\textrm{mod}\ 3)$. We use the following two regular sets of $m$ permutations:

      \smallskip
{\centering
  $ \scriptstyle
    \begin{aligned}
    &R_1=(1,2,3,4) \cdot \mathcal{F}_m=\{\mu_i=(1,2,3,4)\sigma_i\ | \ i=0,1, \ldots, m-1\}; &R_2= \mathcal{F}_m=\{\sigma_0, \sigma_1, \ldots, \sigma_{m-1}\}.\\
          \end{aligned}
  $ 
\par}
\smallskip

Recall that $c={\frac{m+1}{3}+5}+2t$, where $0\leqslant 2t \leqslant \frac{2m-1}{3}-9$, $\mathds{I}=\{i \ |\ i \equiv 0, 4\ (\textrm{mod}\ 6)\ \textrm{and}\ 4 \leqslant i \leqslant m-13\}$, and $M_t$ is a subset of size $t$ of $\mathds{I}$. We construct the following two sets of pairs:

 \smallskip
{\centering
  $ \scriptstyle
    \begin{aligned} 
   &D_{T}&&=&&\{(\mu_{m-4}, \sigma_{m-4}),  (\mu_{m-2}, \sigma_{m-1}), (\mu_{m-1}, \sigma_{m-2})\}\cup\\
   &&&&&\{(\mu_i, \sigma_{m-i-5}), (\mu_{i+1}, \sigma_{m-i-4})\ | \ i \equiv 2 \ (\textrm{mod} \ 6)\  \textrm{and}\  2\leqslant i \leqslant m-15\}\cup \\
   &&&&& \{(\mu_i, \sigma_{m-i-5}), (\mu_{i+1}, \sigma_{m-i-4})\ | \ i\in \{m-10, m-8, m-6\}\}\cup \\
   &&&&&\{(\mu_i, \sigma_{m-i-5}), (\mu_{i+1}, \sigma_{m-i-4}) \ | \ i \in M_t\};\\
   &D_H&&=&&\{(\mu_{m-3}, \sigma_0), (\mu_0, \sigma_{m-3})\}\cup\\
   &&&&&\{(\mu_i, \sigma_{m-i-4})\ | \ i \equiv 0,1, 4,5 \ (\textrm{mod} \ 6), 1\leqslant i \leqslant m-11, \textrm{and}\ i, i-1 \not\in M_t\}.
   \end{aligned}
  $ 
\par}
   \smallskip

\noindent STEP 1. We show that all pairs in $D_T$ are truncated hamiltonian pairs. We consider five cases. 

\noindent \underline{Case 1}: Pairs $ (\mu_{m-2}, \sigma_{m-1})$ and $(\mu_{m-1}, \sigma_{m-2})$. The products are the same as Case 1 of STEP 1 of the computations of Proposition \ref{prop:m3cong6}.\\

\noindent \underline{Case 2}:  The pair  $(\mu_{m-4}, \sigma_{m-4})$. We see that

\medskip
{\centering
  $ \scriptstyle
    \begin{aligned}
     &\hat{\mu}_{m-4} \hat{\sigma}_{m-4}&&=&&(1,2,3,4)\, \gamma_{-3} \, \gamma_{-3}=(1,2,3,4)\gamma_{-6}\\
     &&&=&& (0, m-7, m-13, \ldots, 10, 4, m-6, m-12, \ldots, 5, m-2, m-8,  \ldots, 9, 3, m-3,  \\
     &&&&&  m-9, \ldots, 8, 2, m-4, m-10, \ldots, 7, 1, m-5, m-11, \ldots, 6)  (m-1).\\
      \end{aligned}
  $ 
\par}
\medskip

\noindent \underline{Case 3}: Pairs of the form $(\mu_i, \sigma_{m-i-5}), (\mu_{i+1}, \sigma_{m-i-4})$ where $ i \equiv 2 \ (\textrm{mod} \ 6)\  \textrm{and}\ 2\leqslant i \leqslant m-15$ or $i\in \{m-10, m-8, m-6\}$. See Case 2 and Case 3 of STEP 1 of the computations of Proposition \ref{prop:m3cong6}.

In conclusion, all pairs in $D_T$ are truncated hamiltonian pairs. \\

\noindent STEP 2. We show that all pairs in $D_H$ are hamiltonian pairs. All pairs in $D_H$ are of the form $(\mu_i, \sigma_{m-i-4})$, where $i \equiv 0,3,4,5\ (\textrm{mod}\ 6)$ and $3 \leqslant i \leqslant m-9$. Recall that $k \equiv 2 \ (\textrm{mod}\ 3)$. We will consider five subcases. 

For SUB-CASES 1 and 2, we shall assume that $i=2j$ and $2 \leqslant i \leqslant m-11$. Then $1 \leqslant j \leqslant k-5$, and Lemma \ref{lem:comput} implies that $\sigma_{m-i-4}=\gamma_{-i-3}\, (m-1, k-j)$. 

For SUB-CASES 4 and 5, we shall assume that $i=2j+1$ and $1 \leqslant i \leqslant m-10$. Then $0 \leqslant j \leqslant k-4$ and Lemma \ref{lem:comput} implies that  $\sigma_{m-i-4}=\gamma_{-i-3}\, (m-1, m-j-2)$. 

\noindent SUB-CASE 1:$i=1$. Therefore:

   \medskip
{\centering
  $ \scriptstyle
    \begin{aligned}
       &\mu_1 \sigma_{m-5}&&=&&(1,2,3,4)\gamma_1\, (m-1, 2)\, \gamma_{-4} (m-1,m-2)\\
       &&&=&&( 0, m-4, m-7,  \cdots, 4, m-2, m-5,\cdots, 6, 3, 1, m-1, m-3, m-7,\\
       &&&&& \cdots, 5, 2).
          \end{aligned}
  $ 
\par}
\medskip

\noindent SUB-CASE 2: $i \equiv 1 \ (\textrm{mod} \ 6)$ and $7\leq i \leq m-16$. Then $i=2j+1$ where $j \equiv 0 \ (\textrm{mod} \ 3)$ and $3 \leq j \leq k-8$. Note that Lemma \ref{lem:compgam} implies that  $(m-j)^{\gamma_i}=j+2$ and $(j+2)^{\gamma_{-i-3}}=m-j-3$. Therefore:

   \medskip
{\centering
  $ \scriptstyle
    \begin{aligned}
    &\mu_i \sigma_{m-i-4}&&=&&(1,2,3,4)\, \gamma_i \, (m-1, j+2) \gamma_{-i-3}(m-1, m-j-2)\\
    &&&=&&( 0, m-4, m-7, \ldots, 4, m-3, m-6,  \cdots,  m-j, m-j-2, m-j-5,    \\
    &&&&&\cdots, 6, 3, 1, m-2, m-5, \ldots, m-j+1, m-1,m-j-3,\\
    &&&&& m-j-6, \ldots, 5, 2).
          \end{aligned}
  $ 
\par}
\medskip

\noindent SUB-CASE 3: $i=5$. Therefore:

   \medskip
{\centering
  $ \scriptstyle
    \begin{aligned}
    &\mu_5 \sigma_{m-9}&&=&&(1,2,3,4)\, \gamma_5 \, (m-1, 4) \gamma_{-8}(m-1, m-4)\\
    &&&=&&(0, m-1, m-5, m-8, \ldots, 6, 3, 1, m-2, m-4, m-7, \ldots, 7, 4, m-3, \\
    &&&&& m-6, \ldots, 5, 2).
          \end{aligned}
  $ 
\par}
\bigskip

\noindent SUB-CASE 4: $i \equiv 5\  (\textrm{mod} \ 6)$ and $11\leq i \leq m-12$. Then $i=2j+1$ where $j \equiv 2 \ (\textrm{mod} \ 3)$ and $5 \leq j\leq k-6$. Note that Lemma \ref{lem:compgam} implies that $(m-j)^{\gamma_i}=j+2$ and $(j+2)^{\gamma_{-i-3}}=m-j-3$. Therefore:

   \medskip
{\centering
  $ \scriptstyle
    \begin{aligned}
    &\mu_i \sigma_{m-i-4}&&=&&(1,2,3,4)\, \gamma_i \, (m-1, j+2) \gamma_{-i-3}(m-1, m-j-2)\\
    &&&=&&(0, m-4, m-7, \ldots, m-j+1, m-1,m-j-3, m-j-6, \ldots, 6,  \\
    &&&&& 3, 1,  m-2, m-5, \ldots, m-j, m-j-2, m-j-5,  \cdots, 7, 4,\\
    &&& &&  m-3,  m-6, \ldots, 5, 2). \\
          \end{aligned}
  $ 
\par}
\bigskip

\noindent SUB-CASE 5: $i \equiv 0\ (\textrm{mod} \ 6)$ and $6\leq i \leq m-11$. Then  $i=2j$ where $j \equiv 0 \ (\textrm{mod}\ 3)$ and $3 \leq j \leq k-5$. Therefore:

\medskip
{\centering
  $ \scriptstyle
    \begin{aligned}
    &\mu_i \sigma_{m-i-4}&&=&&(1,2,3,4)\, \gamma_i \, (m-1, k+j+1) \gamma_{-i-3}(m-1, k-j)\\
    &&&=&&(0, m-4, m-7, \ldots, 4, m-3, m-6, \ldots,  k-j+3, m-1,k-j-2,    \\
    &&&&& k-j-5, \ldots, 3, 1, m-2, m-5,\cdots , k-j+1, k-j, \\
    &&&&&k-j-3, \ldots, 5, 2 ). 
          \end{aligned}
  $ 
\par}
\medskip

\noindent SUB-CASE 5: $i \equiv 4\ (\textrm{mod} \ 6)$ and $4 \leqslant i \leq m-13$. Then  $i=2j$ where $j\equiv 2\ (\textrm{mod}\ 3)$ and $2\leqslant j \leqslant k-6$. Therefore:

\medskip
{\centering
  $ \scriptstyle
    \begin{aligned}
    &\mu_i \sigma_{m-i-4}&&=&&(1,2,3,4)\, \gamma_i \, (m-1, k+j+1) \gamma_{-i-3}(m-1, k-j)\\
   &&&=&&(0, m-4,m-7, \ldots, k-j+1, k-j, k-j-3, k-j-6,  \cdots,  6, 3 ,1 ,\\
    &&&&&   m-2,m-5, \ldots, k-j+3, m-1,k-j-2, k-j-5, \ldots, 7, 4,\\
    &&&&&  m-3,  m-6, \ldots, 5, 2 ). 
          \end{aligned}
  $ 
\par}
\medskip

In conclusion, all pairs in $D_H$ are hamiltonian pairs. \hfill $\square$\\

\noindent{\bf Computations for the proof of Lemma \ref{lem:m11}}. Let $m=11$ and $c\in \{5,7\}$. We consider two cases. 

\noindent \underline{Case 1}: $c=5$. Let

   \medskip
{\centering
  $ \scriptstyle
    \begin{aligned}
    &R_1= \mathcal{F}_{11}=\{\sigma_0, \sigma_1, \ldots, \sigma_{10}\};\\
    &R_2= (0, 10)\cdot  \mathcal{F}_{11}=\{\mu_i=(0,10)\sigma_i\ | \ i=0,1, \ldots, 10\}.\\
          \end{aligned}
  $ 
\par}
\medskip

\noindent Recall that $D_H$ and $D_T$ are constructed as follows:

      \smallskip
{\centering
  $ \scriptstyle
    \begin{aligned}
&D_T&&=&&\{(\sigma_1, \mu_0), (\sigma_5, \mu_1), (\sigma_8, \mu_6), (\sigma_9, \mu_9), (\sigma_{10}, \mu_{10})\};\\
&D_H&&=&&\{(\sigma_2, \mu_7), (\sigma_3, \mu_8), (\sigma_4, \mu_5), (\sigma_6, \mu_3), (\sigma_7, \mu_4), (\sigma_0, \mu_2) \}.
          \end{aligned}
  $ 
\par}
   \smallskip

\noindent STEP 1. See below for proof that all pairs of $D_T$ are truncated  hamiltonian:

   \medskip
{\centering
  $ \scriptstyle
    \begin{aligned}
&\hat{\sigma}_{1}\hat{\mu}_{0}&&=&&(1 ,2 ,3 ,4 ,5 ,6 ,7 ,8 ,9 ,0)(10);\\
&\hat{\sigma}_{5}\hat{\mu}_{1}&&=&&(1 ,7 ,3 ,9 ,5 ,2 ,8 ,4 ,0 ,6)(10);\\
&\hat{\sigma}_{8}\hat{\mu}_{6}&&=&&(1 ,5 ,6 ,0 ,4 ,8 ,2 ,9 ,3 ,7)(10);\\
&\hat{\sigma}_{9}\hat{\mu}_{9}&&=&&(1 ,6 ,4 ,2 ,0 ,8 ,9 ,7 ,5 ,3)(10);\\
&\hat{\sigma}_{10}\hat{\mu}_{10}&&=&&(1 ,0 ,9 ,8 ,7 ,3 ,4 ,5 ,6 ,2)(10).  
          \end{aligned}
  $ 
\par}
\medskip

\noindent STEP 2. See below for proof that all pairs of $D_H$ are  hamiltonian:

   \medskip
{\centering
  $ \scriptstyle
    \begin{aligned}
& \sigma_{2} \mu_{7}&&=&&(10 ,4 ,3 ,2 ,1 ,0 ,9 ,8 ,5 ,7 ,6);\\
& \sigma_{3} \mu_{8}&&=&&(10 ,1 ,2 ,3 ,4 ,5 ,6 ,7 ,0 ,8 ,9);\\
& \sigma_{4} \mu_{5}&&=&&(10 ,3 ,2 ,1 ,0 ,9 ,8 ,7 ,6 ,4 ,5);\\
& \sigma_{6} \mu_{3}&&=&&(10 ,2 ,1 ,0 ,9 ,8 ,7 ,6 ,5 ,4 ,3);\\
& \sigma_{7} \mu_{4}&&=&&(10 ,9 ,0 ,1 ,2 ,3 ,8 ,4 ,5 ,6 ,7);\\
& \sigma_{0} \mu_{2}&&=&&(10 ,2 ,4 ,6 ,8 ,0 ,7 ,9 ,1 ,3 ,5).
          \end{aligned}
  $ 
\par}

\noindent Hence $D=D_T\cup D_H'$ is a $5$-twined 2-factorization.

\noindent \underline{Case 2}: $c=7$. Let

   \medskip
{\centering
  $ \scriptstyle
    \begin{aligned}
    &R_1=(10,1,2,3,4,5) \cdot \mathcal{F}_{11}=\{ \mu_i=(10,1,2,3,4,5) \sigma_i\ | \ i=0,1, \ldots, 10\};\\
    &R_2= \mathcal{F}_{11}=\{\sigma_0, \sigma_1, \ldots, \sigma_{10}\}. \\
          \end{aligned}
  $ 
\par}
\medskip

\noindent Sets $D_T$ and $D_H$ are constructed as follows:

      \smallskip
{\centering
  $ \scriptstyle
    \begin{aligned}
&D_T&&=&&\{(\mu_2, \sigma_6), (\mu_3, \sigma_1), (\mu_5, \sigma_9), (\mu_6, \sigma_8), (\mu_7, \sigma_7), (\mu_9, \sigma_5), (\sigma_{10}, \mu_{10})\};\\
&D_H&&=&&\{(\mu_1, \sigma_4), (\mu_4, \sigma_0), (\mu_{8}, \sigma_3), (\mu_0, \sigma_{2})\}.
          \end{aligned}
  $ 
\par}
   \smallskip

\noindent STEP 1. See below for proof that all pairs in $D_T$ are truncated hamiltonian pairs:

   \medskip
{\centering
  $ \scriptstyle
    \begin{aligned}
&\hat{\mu}_{2}\hat{\sigma}_{6}&&=&&(1,0,8,6,4,9,7,5,3,2)(10);\\
&\hat{\mu}_{3}\hat{\sigma}_{1}&&=&&(1,6,0,5,4,9,3,8,2,7)(10);\\
&\hat{\mu}_{5}\hat{\sigma}_{9}&&=&&(1,6,0,4,9,5,3,8,2,7)(10);\\
&\hat{\mu}_{6}\hat{\sigma}_{8}&&=&&(1,6,0,4,9,3,8,2,5,7)(10);\\
&\hat{\mu}_{7}\hat{\sigma}_{7}&&=&&(1,6,0,4,9,3,8,5,2,7)(10);\\
&\hat{\mu}_{9}\hat{\sigma}_{5}&&=&&(1,6,0,4,9,3,8,2,7,5)(10);\\
&\hat{\mu}_{10}\hat{\sigma}_{10}&&=&&(1,3,5,2,4,6,0,9,8,7)(10).\\
          \end{aligned}
  $ 
\par}
\medskip

\noindent STEP 2. See below for proof that all pairs in $D_H$ are hamiltonian pairs:

   \medskip
{\centering
  $ \scriptstyle
    \begin{aligned}
&\mu_{1}\sigma_{4}&&=&&(10,8,3,9,4,0,5,6,1,7,2);\\
&\mu_{4}\sigma_{0}&&=&&(10,5,8,2,7,1,6,0,4,9,3);\\
&\mu_{8}\sigma_{3}&&=&&(10,2,4,6,7,8,9,0,1,3,5);\\
&\mu_{0}\sigma_{2}&&=&&(10,3,6,8,0,2,5,7,9,1,4).\\
          \end{aligned}
  $ 
\par}
\medskip

\noindent Then $D=D_T\cup D_H$ is the desired $7$-twined 2-factorization.

\end{document}